\documentclass[11pt,twoside]{preprint}
\usepackage[full]{textcomp}
\usepackage[osf]{newtxtext}
\usepackage{amsmath,amsthm,amssymb,enumerate}
\usepackage{hyperref}
\usepackage{mathrsfs}

\usepackage{bbm}

\usepackage[toc,page]{appendix}

\usepackage{secdot}
\sectiondot{subsection}

\usepackage{etoolbox}
\patchcmd{\numberline}{\hfil}{.\hfil}{}{}  

\usepackage{breakurl}
\usepackage{color}
\usepackage{tikz-cd}

\definecolor{DGreen}{rgb}{0,0.55,0}
\usepackage{xcolor}
\usepackage{mhequ}
\usepackage{comment}

\usepackage{microtype}
\usepackage[a4paper,innermargin=1.2in,outermargin=1.2in,bottom=1.5in,marginparwidth=1in,marginparsep=3mm]{geometry}

\renewcommand{\theequation}{\thesection.\arabic{equation}}

\usepackage{mhequ}

\def\E{\hskip.15ex\mathbb{E}\hskip.10ex}
\def\P{\mathbb{P}}

\def\Var{\mathop{\mbox{\rm Var}}}

\def\eps{\varepsilon}
\def\phi{\varphi}

\newtheorem{Theorem} {Theorem}[section]
\newtheorem{theorem}[Theorem]{Theorem}
\newtheorem{Lemma}[Theorem]{Lemma}
\newtheorem{lemma}[Theorem]{Lemma}
\newtheorem{Proposition}[Theorem]{Proposition}
\newtheorem{proposition}[Theorem]{Proposition}
\newtheorem{Corollary}[Theorem]{Corollary}
\theoremstyle{definition}
\theoremstyle{definition}\newtheorem{Remark}[Theorem]{Remark}
\theoremstyle{definition}

\theoremstyle{remark}
\newtheorem{remark}[Theorem]{Remark}

\usepackage{mhequ}

\def\db#1{[\hspace{-.8mm}] #1 [\hspace{-.8mm}]}

\renewcommand{\theequation}{\thesection.\arabic{equation}}
\numberwithin{equation}{section}

\renewcommand{\ge}{\geqslant}
\renewcommand{\le}{\leqslant}
\renewcommand{\geq}{\ge}
\renewcommand{\leq}{\le}

\newcommand{\nn}{\nonumber}

\newcommand{\wt}{\widetilde}

\renewcommand{\d}{\partial}

\newcommand{\id}{\mathrm{id}}

\newcommand{\A}{\mathcal{A}}
\newcommand{\C}{\mathcal{C}}
\newcommand{\F}{\mathcal{F}}
\newcommand{\bF}{\mathbb{F}}

\newcommand{\bP}{\mathbb{P}}

\newcommand{\cR}{\mathcal{R}}

\newcommand{\cQ}{\mathcal{Q}}

\newcommand{\N}{\mathbb{N}}
\newcommand{\R}{\mathbb{R}}

\def\cX{\mathcal{X}}
\def\scC{\mathscr{C}}

\newcommand{\D}{\partial}
\newcommand{\cP}{\mathcal{P}}
\newcommand{\cD}{\mathscr{D}}
\newcommand{\cM}{\mathscr{M}}
\newcommand{\ccD}{\mathcal{D}}
\newcommand{\cI}{\mathcal{I}}
\newcommand{\cW}{\mathcal{W}}
\newcommand{\<}{\langle}
\renewcommand{\>}{\rangle}
\newcommand{\bone}{\mathbf{1}}

\definecolor{Brown}{rgb}{.75,.5,.25}
\definecolor{DGreen}{rgb}{0,0.55,0}
\definecolor{Olive}{rgb}{0.41,0.55,0.13}

\title{Approximation of SDEs -- a stochastic sewing approach}
\author{Oleg Butkovsky\thanks{Weierstrass Institute, Mohrenstra\ss e 39, 10117 Berlin, Germany$\qquad$\url{oleg.butkovskiy@gmail.com}}\,,
Konstantinos Dareiotis\thanks{University of Leeds, Woodhouse, LS2 9JT Leeds, United Kingdom$\qquad$\url{k.dareiotis@leeds.ac.uk}}, and M\'at\'e Gerencs\'er\thanks{TU Vienna, Wiedner Hauptstrasse 8-10, 1040 Wien, Austria$\qquad$\url{mate.gerencser@tuwien.ac.at}}}

\begin{document}
\maketitle
\begin{abstract}
We give a new take on the error analysis of approximations of stochastic differential equations (SDEs), utilizing and developing the stochastic sewing lemma of L{\^e} (2020). This approach allows one to exploit regularization by noise effects in obtaining convergence rates.
In our first application we show convergence (to our knowledge for the first time)
of the Euler-Maruyama scheme for SDEs driven by fractional Brownian motions with non-regular drift.
When the Hurst parameter is $H\in(0,1)$ and the drift is $\C^\alpha$, $\alpha\in[0,1]$ and $\alpha>1-1/(2H)$, we show the strong $L_p$ and almost sure rates of convergence to be $((1/2+\alpha H)\wedge 1) -\eps$, for any $\eps>0$. Our conditions on the regularity of the drift are optimal in the sense that they coincide with the conditions needed for the strong uniqueness of solutions from Catellier, Gubinelli (2016).
In a second application we consider the approximation of SDEs driven by multiplicative standard Brownian noise where we derive the almost optimal rate of convergence $1/2-\eps$ of the Euler-Maruyama scheme for $\C^\alpha$ drift, for any $\eps,\alpha>0$.
\end{abstract}


\tableofcontents

\section{Introduction}

Since the 1970s, it has been observed that the addition of a random forcing into an ill-posed
deterministic system could make it well-posed. Such phenomenon is called \textit{regularization by noise}. One of the prime examples concerns differential equations of the form
\begin{equation}\label{ode}
dX_t = b(X_t)\, dt,
\end{equation}
where $b$ is a bounded vector field. While equation \eqref{ode} might have infinitely many solutions when $b$ fails to be Lipschitz continuous and might possess no solution when $b$ fails to be continuous, Zvonkin \cite{Zvonya} and Veretennikov \cite{Veret80} (see also the paper of Davie \cite{Davie}) showed that the stochastic differential equation (SDE)
\begin{equation}\label{sdebr}
dX_t = b(X_t)\,dt +  dB_t
\end{equation}
driven by a Brownian motion $B$, has a unique strong solution when $b$ is merely bounded measurable. This result was extended to the case of the fractional Brownian noise in \cite{Nualart, NualartUnbounded, Cat-Gub, BNP, Khoa}. These papers study the equation
\begin{equ}\label{eq:main frac}
dX_t=b(X_t)\,dt+\,dB^H_t,\qquad X_0=x_0
\end{equ}
where $B^H$ is a $d$-dimensional fractional Brownian motion with Hurst parameter $H\in(0,1)$.  It is known  \cite[Theorem~1.9]{Cat-Gub} that this equation has a unique strong solution if $b$ belongs to the H\"older--Besov space $\C^\alpha$ and $\alpha>1-1/(2H)$.
Thus, the presence of the noise not only produces solutions in situations where there was none but also singles out a unique physical solution in situations where there were multiple. However, to the best of our knowledge, no construction of this solution through discrete approximations has been known (unless $H=1/2$). In this article, we develop a new approach which allows to construct this solution and even obtain rate of convergence of the discrete approximations.
Before the formal setup of  Section~\ref{sec:formulation}, let us informally overview the results.

First, let us recall that in the standard Brownian case ($H=1/2$) the seminal work of Gy\"ongy and Krylov \cite{GyK} established the convergence in probability of the Euler-Maruyama scheme
\begin{equ}\label{eq:main approx frac}
dX^n_t=b(X^n_{\kappa_n(t)})\,dt+\,dB^H_t,\qquad X_0^n=x_0^n,\quad t\ge0
\end{equ}
to the solution of \eqref{eq:main frac}. Here $b$ is a bounded measurable function and
\begin{equation}\label{kappadef}
\kappa_n(t):=\lfloor nt\rfloor/n, \quad n\in\N.
\end{equation}

In the present paper, we significantly extend these results by a) establishing the convergence  of the Euler--Maruyama scheme for all $H\in(0,1)$; b) showing that the convergence takes place in a stronger ($L_p(\Omega)$ and almost sure) sense; c) obtaining the explicit rate of convergence. More precisely, in Theorem \ref{thm:main fractional} we show that if $b$ is bounded and H\"older-continuous with exponent $\alpha>1-1/(2H)$, then the Euler-Maruyama scheme converges with rate $((1/2+\alpha H)\wedge1)-\eps$ for any $\eps>0$.
Thus, the approximation results are obtained under the minimal assumption on the drift $b$ that is needed for strong uniqueness of solutions \cite{Nualart, Cat-Gub} and for the well-posedness of
scheme~\eqref{eq:main approx frac}. Let us also point out that in particular, for $H<1/2$, one does not need to require any continuity from $b$ to obtain a convergence rate $1/2-\eps$.
Concerning approximations of SDEs driven by fractional Brownian motions with regular coefficients, we refer the reader to the recent works
\cite{Friz-Riedel, Nualart-Approx}
and references therein.
Concerning the implementation of such schemes and in particular the simulation of increments of fractional Brownian motions we refer to \cite[Section~6]{Sheva} and its references.

Our second application is to study equations with multiplicative noise in the standard Brownian case:
\begin{equ}\label{eq:main mult}
dX_t=b(X_t)\,dt+\sigma(X_t)\,dB_t,\qquad X_0=x_0,\quad t\ge0
\end{equ}
and their discretisations
\begin{equ}\label{eq:approx EM}
dX^{n}_t=b(X^{n}_{\kappa_n(t)})\,dt+\sigma(X^{n}_{\kappa_n(t)})\,dB_t,\quad X_0^{n}=x_0^n,
\quad t\ge0.
\end{equ}
Here $b$, $\sigma$ are measurable functions, $B$ is a $d$-dimensional Brownian motion, and $\kappa_n$ is defined in \eqref{kappadef}. To ensure well-posedness, a nondegeneracy assumption on $\sigma$ has to be assumed. In the standard Brownian case the rate of convergence for irregular $b$ has been recently actively studied, see among many others \cite{MX, Szo3, Y2, PT, Bao2} and their references. However, the obtained rate deteriorates as $b$ becomes more irregular: in the setting of \eqref{eq:main mult}-\eqref{eq:approx EM}, the best known rate is only proven to be (at least) $\alpha/2$ for $b\in\C^\alpha$, $\alpha>0$ in \cite{Bao2}.

It was first shown in \cite{DK} that, at least for additive noise, the strong rate does \textit{not} vanish as the regularity $\alpha$ approaches $0$, and one in fact recovers the rate $1/2-\eps$ for arbitrary $\eps>0$, for all $\alpha>0$.
In the present paper we establish the same for multiplicative noise, in which
case the rate $1/2$ is well-known to be optimal.
Our proof offers several other improvements to earlier results:
all moments of the error can be treated in the same way, the scalar and multidimensional cases are also not distinguished, and the main error bound \eqref{eq:main bound mult EM} is uniform in time, showing that $X_\cdot$ and $X^{n}_\cdot$ are close as paths. The topology (in time) where the error is measured is in fact even stronger, see Remark \ref{remark:topology}.

To obtain these results we develop a new strategy which utilizes the stochastic sewing lemma
(SSL) of L\^e \cite{Khoa} as well as some other specially developed tools. We believe that these tools might be also of independent interest; let us briefly describe them here.

First,  we obtain a new stochastic sewing--type lemma, see Theorem~\ref{thm:Sewing-lemma-better-p}. It provides bounds on the
$L_p$-norm of the increments of a process, with the correct dependence on $p$. This improves the corresponding bounds from SSL of
L\^e (although, under more restrictive conditions). This improved bound is used for proving
stretched exponential moment bounds that play a key role 
in the convergence analysis of the Euler--Maruyama scheme for \eqref{eq:main frac}, see Section~\ref{S:proof21}. In particular, using this new sewing-type lemma, we are able to extend the key bound of Davie \cite[Proposition~2.1]{Davie} (this bound was pivotal in his paper for establishing uniqueness of solutions to \eqref{sdebr} when the driving noise is the standard Brownian motion) to the case of the fractional Brownian noise,
see Lemma~\ref{L:Daviebound}.

Second, in Section \ref{sec:Malliavin} we derive density estimates of (a drift-free version of) the solution of \eqref{eq:approx EM} via Malliavin calculus. Classical results in this direction include that of Gy\"ongy and Krylov \cite{GyK},
and of Bally and Talay \cite{Bally-Talay, Bally-Talay2}:
the former gives sharp short time asymptotics but no smoothness of the density, and the latter vice versa (see Remark \ref{rem:density} below). Since our approach requires both properties at the same time, we give a self-contained proof of such an estimate \eqref{eq:main Malliavin}.

Finally let us mention that, as in \cite{DK, Neu-Sz, Issoglo}, efficient quadrature bounds play a crucial role in the analysis.
These are interesting approximation problems in their own right, see, e.g., \cite{KHiga_quad} and the references therein.
Such questions in the non-Markovian setting of fractional Brownian motion have only been addressed recently in \cite{Altmeyer}.
However, there are a few key differences to our quadrature bounds from Lemma \ref{lem:(ii)}.
First, we derive bounds in $L_p(\Omega)$ for all $p$, which by Proposition \ref{prop:almost sure} also imply the corresponding almost sure rate (as opposed to $L_2(\Omega)$ rates only in \cite{Altmeyer}).
Second, unlike the standard fractional Brownian motions considered here, \cite{Altmeyer} requires starting them at time $0$ from a random variable with a density, which provides a strong smoothing effect.
Third, when approximating the functional of the form
\begin{equ}
\Gamma_t:=\int_0^tf(B^H_s)\,ds,
\end{equ}
also called `occupation time functional',
by the natural discretisation
\begin{equ}
\Gamma^n_t=\int_0^tf(B^H_{\kappa_n(s)})\,ds,
\end{equ}
our results not only imply pointwise error estimates on $|\Gamma_T-\Gamma^n_T|$, but also on the error of the whole path $\|\Gamma_{\cdot}-\Gamma^n_\cdot\|_{\C^\beta}$ measured in a H\"older norm $\C^\beta$ with some $\beta>1/2$.
This is an immediate consequence of the bounds \eqref{DKBound frac} in combination with Kolmogorov's continuity theorem.

The rest of the article is structured as follows. Our main results are presented in Section~\ref{sec:formulation}. 
In Section~\ref{sec:prelim-bigsec} we outline the main strategy and collect some necessary auxiliary results, including the new sewing lemma--type bound Theorem \ref{thm:Sewing-lemma-better-p}.
Section~\ref{sec:fractional} is devoted to the error analysis in the additive fractional noise case. In Section~\ref{sec:Malliavin} we prove an auxiliary bound on the probability distribution of the Euler-Maruyama approximation of certain sufficiently nice SDEs. The proofs of the convergence in the multiplicative standard Brownian noise case are given in Section~\ref{sec:mult}.

\medskip

\bigskip

\noindent\textbf{Acknowledgments.} OB has received funding from the European Research
Council (ERC) under the European Union’s Horizon 2020 research and innovation program
(grant agreement No. 683164)
and from the DFG Research Unit FOR 2402.
MG was supported by the Austrian Science Fund (FWF) Lise Meitner programme M2250-N32. Part of the work on the project has been done during
the visits of the authors to IST Austria, Technical University Berlin, and Hausdorff Research Institute for Mathematics (HIM).  We thank them all
for providing excellent working conditions, support and hospitality.
Finally, we thank the referee for the careful reading of our paper and several useful comments.

\section{Main results}\label{sec:formulation}

We begin by introducing the basic notation. Consider a probability space $(\Omega, \mathcal{F}, \mathbb{P})$ carrying a $d$-dimensional  two-sided
Brownian motion $(W_t)_{t \in \mathbb{R}}$. Let $\mathbb{F}=(\mathcal{F}_t)_{t \in \mathbb{R}}$ be the filtration generated by the increments of $W$.
The conditional expectation given $\mathcal{F}_s$ is denoted by $\E^s$.
For $H \in (0,1)$ we define the fractional Brownian motion with Hurst parameter $H$ by the Mandelbrot-van Ness representation \cite[Proposition 5.1.2]{Nualart-Malliavin}
\begin{equs}\label{eq:Mandelbrot}
B^H_t := \int_{-\infty}^0 \bigl(|t-s|^{H-1/2}- |s|^{H-1/2}\bigr) \, dW_s + \int_0^t |t-s|^{H-1/2} \, dW_s.
\end{equs}
Recall that the components of $B^H$ are independent and each component is a Gaussian process with zero mean and covariance
\begin{equation}\label{covfunct}
C(s,t):=\frac{c_H}2(s^{2H}+t^{2H}-|t-s|^{2H}),\quad s,t\ge0,
\end{equation}
where $c_H$ is a certain positive constant, see \cite[(5.1)]{Nualart-Malliavin}.

For $\alpha\in(0,1]$ and a function $f\colon Q\to V$, where $Q\subset \R^k$ and $(V,|\cdot|)$ is a  normed space, we set
\begin{equation*}
[f]_{\C^\alpha(Q,V)}:=\sup_{x\neq y\in Q}\frac{|f(x)-f(y)|}{|x-y|^\alpha}.
\end{equation*}
For $\alpha\in (0,\infty)$ we denote by $\C^\alpha (Q,V)$ the space of all functions $f\colon Q\to V$ having  derivatives $\d^\ell f$ for  all multi-indices $\ell \in ( \mathbb{Z}_+)^k$ with $|\ell|<\alpha$ such that 
\begin{equ}
\|f\|_{\C^\alpha(Q,V)}:=\sum_{|\ell|< \alpha} \sup_{x\in Q}|\d^\ell f(x)|+
\sum_{\alpha-1<  |\ell|< \alpha}[\d^\ell f]_{\C^{\alpha-|\ell|}(Q,V)}< \infty. 
\end{equ}
If $\ell=(0,\ldots,0)$, then as usual, we use the convention $\d^\ell f=f$. In particular, the $\C^\alpha$ norm always includes the supremum of the function.
We also  set $\C^0(Q,V)$ to be the space of bounded measurable functions with the supremum norm.
We emphasize that in our notation elements of $\C^0$ need \emph{not} be continuous!
If $\alpha<0$, then by $\C^\alpha(\R^d,\R)$ we  denote the space of all distributions $f \in \mathcal{D}'( \R^d)$, such that 
\begin{equs}
\|f \|_{\mathcal{C}^\alpha} := \sup_{\eps\in(0,1]}  \eps^{-\alpha/2} \|\cP_\eps  f\|_{\C^0(\R^d,\R)}< \infty,
\end{equs}
where $\cP_\eps f $ is the convolution of $f$ with the $d$-dimensional Gaussian heat kernel at time $\eps$.

In some cases  we use shorthands: if $Q=\R^d$, or $V=\R^d$ or $V=\R^{d\times d}$, they are omitted from the notation.
For instance, the reader understands that requiring the diffusion coefficient $\sigma$ of \eqref{eq:main mult} to be of class $\C^\alpha$ is to require it to have finite $\|\cdot\|_{\C^\alpha(\R^d,\R^{d\times d})}$ norm.
If $V=L_p(\Omega)$ for some $p\geq 2$, we write 
\begin{equation}\label{boxes}
\db{f}_{\scC^\alpha_p,Q}:=\|f\|_{\C^\alpha(Q,L_p(\Omega))}.
\end{equation}

\textbf{Convention on constants}. Throughout the paper $N$ denotes a positive constant
whose value may change from line to line; its dependence is always specified in the corresponding statement.

\subsection{Additive fractional noise}\label{S:AFN}

Our first main result establishes the convergence of the numerical scheme
\eqref{eq:main approx frac} to the solution of equation \eqref{eq:main frac}. Fix $H\in(0,1)$.
It is known (\cite[Theorem~1.9]{Cat-Gub}) that if the drift $b\in\C^\alpha$ with $\alpha\in [0,1]$ satisfying $\alpha>1-1/(2H)$, then for any fixed $x_0\in\R^d$, equation \eqref{eq:main frac} admits a unique strong solution, which we denote by $X$.
For any $n\in \N$ we take $x_0^n\in\R^d$ and denote the solution of \eqref{eq:main approx frac} by $X^n$.
For a given $\alpha\in [0,1]$ and $H\in(0,1)$, we set
\begin{equ}\label{eq:gamma}
\gamma=\gamma(\alpha,H):=(1/2+\alpha H)\wedge1.
\end{equ}
Now we are ready to present our first main result. Its proof is placed in Section~\ref{sec:fractional}, a brief outline of it is provided in Section~\ref{sec:outline}.
\begin{theorem}\label{thm:main fractional}
Let $\alpha\in[0,1]$ satisfy
\begin{equ}\label{eq:exponent main}
\alpha>1-1/(2H).
\end{equ}
Suppose $b\in\C^\alpha$, let $\eps,\delta>0$ and $p\geq 2$.
Then there exists a constant $\tau=\tau(\alpha,H,\eps)>1/2$ such that for all $n\in\N$ the following bound holds
\begin{equ}\label{eq:frac main bound}
\|X-X^n\|_{\C^\tau([0,1],L_p(\Omega))}\leq N n^{\delta}|x_0-x^n_0| + N n^{-\gamma+\eps+\delta}
\end{equ}
with some constant $N=N(p,d,\alpha,H,\eps,\delta,\|b\|_{\C^\alpha})$.
\end{theorem}
\begin{remark}
An interesting question left open is whether one can reach $\alpha=0$ in the $H=1/2$ case. In dimension $1$, this is positively answered \cite{DK} using PDE methods, but the sewing approach at the moment does not seem to handle such endpoint situations. For $H\neq 1/2$ even weak existence or uniqueness is not known for the endpoint $\alpha=1-1/(2H)$.
\end{remark}
\begin{remark}\label{remark:topology}
From \eqref{eq:frac main bound}, Kolmogorov's continuity theorem, and Jensen's inequality, one gets the bound
\begin{equ}\label{eq:frac main bound var}
\big\| \|X-X^n\|_{\C^{\tau-\eps'}([0,1],\R^d)}\big\|_{L_p(\Omega)}\leq N n^\delta|x_0-x^n_0| + N n^{-\gamma+\eps+\delta}.
\end{equ}
for any $\eps'>0$ (with $N$ also depending on $\eps'$).
In the literature it is more common to derive error estimates in supremum norm, which of course follows:
\begin{equ}
\big\|\sup_{t\in[0,1]}|X_t-X^n_t|\big\|_{L_p(\Omega)}\leq N n^\delta|x_0-x^n_0| + N n^{-\gamma+\eps+\delta},
\end{equ}
but \eqref{eq:frac main bound var} is quite a bit stronger.
\end{remark}
\begin{remark}
A trivial lower bound on the rate of convergence of the solutions is the rate of convergence of the initial conditions. In \eqref{eq:main approx frac} we lose $\delta$ compared to this rate, but $\delta>0$ can be chosen arbitrarily small. This becomes even less of an issue if one simply chooses $x_0^n=x_0$.
\end{remark}
\begin{remark}\label{remark:topology2}
The fact that the error is well-controlled even between the gridpoints is related to the choice of how we extend $X^n$ to continuous time from the points $X_0^n,X_{1/n}^n,\ldots$. For other type of extensions and their limitations we refer the reader to \cite{NEU06}.
\end{remark}
\begin{Corollary}\label{cor:as fractional}
Assume $\alpha\in[0,1]$ satisfies \eqref{eq:exponent main} and suppose $b\in\C^\alpha$. Take
$x_0=x_0^n$ for all $n\in\N$. Then for a sufficiently small $\theta>0$ and any $\eps>0$ there exists an almost surely finite random variable $\eta$ such that 
for all $n\in\N$, $\omega\in\Omega$ the following bound holds
\begin{equ}
\sup_{t\in[0,1]}|X_t-X^n_t|\leq
\|X-X^n\|_{\C^{1/2+\theta}([0,1],\R^d)}
\leq
 \eta n^{-\gamma+\eps},
\end{equ}
where $\gamma$ was defined in \eqref{eq:gamma}.
\end{Corollary}
\begin{proof}
An immediate consequence of \eqref{eq:frac main bound var}, Proposition \ref{prop:almost sure} below, and the fact that $\tau>1/2$.
\end{proof}

\subsection{Multiplicative Brownian noise}
In the multiplicative case we work under the ellipticity and regularity conditions
\begin{equ}\label{eq:elliptic}
\sigma\in\C^2,\qquad\qquad\sigma\sigma^T\succeq\lambda I,
\end{equ}
in the sense of positive definite matrices, with some $\lambda>0$.
This, together with $b\in\C^0$, guarantees the strong well-posedness of equations \eqref{eq:main mult} and \eqref{eq:approx EM} \cite[Theorem~1]{Veret80}, whose solutions we denote by
$X$ and $X^{n}$, respectively.
The second main result then reads as follows, its proof is the content of Section \ref{sec:mult}.
\begin{theorem}\label{thm:main multiplicative}
Let $\alpha\in(0,1]$. Suppose $b\in\C^\alpha$, let $\eps>0$, $\tau\in[0,1/2)$, and $p\geq 2$. Suppose $\sigma$ satisfies \eqref{eq:elliptic}.
Then for all $n\in\N$ the following bound holds
\begin{equ}\label{eq:main bound mult EM}
\|X-X^n\|_{\C^\tau([0,1],L_p(\Omega))}\leq N|x_0-x_0^n| + N n^{-1/2+\eps}
\end{equ}
with some $N=N(p,d,\alpha,\eps,\tau,\lambda,\|b\|_{\C^\alpha}, \|\sigma\|_{\C^2})$.
\end{theorem}

\begin{Corollary}\label{cor:as multiplicative}
Let $\alpha\in(0,1]$, assume $x_0=x_0^n$ for all $n\in\N$, suppose $b\in\C^\alpha$, and suppose $\sigma$ satisfies \eqref{eq:elliptic}.
Let $\eps>0$, $\tau\in[0,1/2)$. Then there exists an almost surely finite random variable $\eta$ such that  for all $n\in \N$, $\omega\in\Omega$ the following bound holds
\begin{equ}
\sup_{t\in[0,1]}|X_t-X^n_t|\leq
\|X-X^n\|_{\C^{\tau}([0,1],\R^d)}
\leq
 \eta n^{-1/2+\eps}.
\end{equ}
\end{Corollary}
\begin{proof}
An immediate consequence of \eqref{eq:main bound mult EM}, Kolmogorov's continuity theorem, and Proposition \ref{prop:almost sure} below.
\end{proof}

Let us conclude by invoking a simple fact used in the proof of Corollaries \ref{cor:as fractional} and \ref{cor:as multiplicative},
which goes back to at least \cite[proof of Theorem 2.3]{Gy} (see also \cite[Lemma 2]{Faure}).
\begin{proposition}\label{prop:almost sure}
Let $\rho>0$ and let $(Z_n)_{n\in\N}$ be a sequence of random variables such that for all $p>0$ and all $n\in\N$ one has the bound
\begin{equ}
\|Z_n\|_{L_p(\Omega)}\leq N n^{-\rho}
\end{equ}
for some $N=N(p)$. Then for all $\eps>0$ there exists an almost surely random variable $\eta$ such that for all $n\in\N$, $\omega\in\Omega$
\begin{equ}
|Z_n|\leq \eta n^{-\rho+\eps}.
\end{equ}
\end{proposition}
\begin{proof}
Notice that for any $q>0$
\begin{equ}
\sum_{n\in\N}\bP(|Z_n|>n^{-\rho+\eps})\leq \sum_{n\in\N}\frac{\E|Z_n|^q}{n^{q(-\rho+\eps)}}\leq \sum_{n\in\N} N n^{-q\eps}.
\end{equ}
Choosing $q=2/\eps$, the above sum is finite, so by the Borel-Cantelli lemma there exists an almost surely finite $\N$-valued random variable $n_0$ such that $|Z_n|\leq n^{-\rho+\eps}$ for all $n>n_0$. This yields the claim by setting
\begin{equation*}
\eta:=1\vee\max_{n\leq n_0}(|Z_n|n^{\rho-\eps}).\qedhere
\end{equation*}
\end{proof}

\section{Preliminaries}\label{sec:prelim-bigsec}

\subsection{The outline of the strategy}\label{sec:outline}
The purpose of this section is to outline the main steps
in a simple example.
Hopefully this gives a clear picture of the strategy to the reader, which otherwise may be blurred by the some complications arising
in the proofs of Theorems \ref{thm:main fractional} and \ref{thm:main multiplicative}.

The `simple example' will be the setting of \eqref{eq:main frac} and \eqref{eq:main approx frac} with $H=1/2$ and $f\in\C^\alpha$ for some $\alpha>0$.
We furthermore assume $x_0=x_0^n$ and that the time horizon is given by $[0,T_0]$ instead of $[0,1]$, with some small $1\geq T_0>0$ to be chosen later. Finally, we will only aim to prove \eqref{eq:frac main bound} with $\tau=1/2$.

\emph{Step 1 ("Quadrature bounds")} Our first goal is to bound the quantity 
\begin{equ}\label{eq:referee}
\A_{T_0}:=\int_0^{T_0} b(B_r)-b(B_{\kappa_n(r)})\,dr.
\end{equ}
From the H\"older continuity of $b$, one would have the trivial bound of order $n^{-\alpha/2}$ in any $L_p(\Omega)$ norm, but in fact one can do much better, as follows.
Fix $\eps\in(0,1/2)$ and define (recall that by $\E^s$ we denote the conditional expectation given $\F_s$)
\begin{equ}
A_{s,t}=\E^s(\A_t-\A_s)=\E^s\int_s^t b(B_r)-b(B_{\kappa_n(r)})\,dr.
\end{equ}
The stochastic sewing lemma, Proposition \ref{thm:Sewing-lemma} below, allows one to bound $\A$ through bounds on $A$. 
Given the preceding field $A_{s,t}$, provided that the conditions \eqref{SSL1} and \eqref{SSL2} are satisfied, it is easy to check that the unique adapted
process $\A$ constructed in Proposition \ref{thm:Sewing-lemma} coincides with the one in \eqref{eq:referee}.
Indeed, the process in \eqref{eq:referee} satisfies \eqref{SSL1 cA} and \eqref{SSL2 cA} with $\eps_1=\eps$, $\eps_2=1$,
$K_1=\|b\|_{\C^0}$ and $K_2=0$.
Therefore it remains to find $C_1$ and $C_2$. In fact, it is immediate that one can choose $C_2=0$, since
$\E^s\delta A_{s,u,t}=\E^s(A_{s,t}-A_{s,u}-A_{u,t})=0$.

We now claim that one can take $C_1=Nn^{-1/2-\alpha/2+\eps}$ in \eqref{SSL1}.
Since $\|b(B_r)-b(B_{\kappa_n(r)})\|_{L_p(\Omega)}\leq \|b\|_{\C^\alpha}n^{-\alpha/2}$,
if $|t-s|\leq 2 n^{-1}$, then one easily gets by the conditional Jensen's inequality
\begin{equ}\label{eq:explanation1}
\|A_{s,t}\|_{L_p(\Omega)}\leq N |s-t|n^{-\alpha/2}\leq N |s-t|^{1/2+\eps}n^{-1/2-\alpha/2+\eps}.
\end{equ}
If $|t-s|>2n^{-1}$, let $s'=\kappa_n(s)+2n^{-1}$ be the second gridpoint to the right of $s$. In particular, $r\geq s'$ implies $\kappa_n(r)\geq s$.
Let us furthermore notice that for any $u\geq v$ and any bounded measurable function $f$, one has $\E^v f(B_u)=\cP_{u-v}f(B_v)$,
where $\cP$ is the standard heat kernel (see \eqref{eq:p-def} below for a precise definition).
One can then write
\begin{equs}
\|A_{s,t}\|_{L_p(\Omega)}&\leq \int_s^{s'}\|b(B_r)-b(B_{\kappa_n(r)})\|_{L_p(\Omega)}\,dr+\big\|\int_{s'}^t\E^s b(B_r)-\E^s b(B_{\kappa_n(r)})\,dr\big\|_{L_p(\Omega)}
\\
&\leq N n^{-1-\alpha/2}+\int_{s'}^t\|(\cP_{r-s}-\cP_{\kappa_n(r)-s})b\|_{\C^0}\,dr
\\
&\leq N n^{-1-\alpha/2}+N\int_{s'}^t(r-s')^{-1/2+\eps}n^{-1/2-\alpha/2+\eps}\,dr
\\
&\leq N|t-s|^{1/2+\eps}n^{-1/2-\alpha/2+\eps}\label{eq:explanation2}
\end{equs}
where in the third line we used a well-known estimate for heat kernels, see Proposition \ref{prop:HK} (ii) with exponents $\beta=0$, $\delta=1/2+\alpha/2-\eps$, and time points $\kappa_n(r)-s$ in place of $s$, $r-s$ in place of $t$. 
We also used that for $r\geq s'$, one has $\kappa_n(r)-s\geq r-s'$.
By \eqref{eq:explanation1} and \eqref{eq:explanation2} we indeed get \eqref{SSL1} with $C_1=N n^{-1/2-\alpha/2+\eps}$.
Applying the stochastic sewing lemma, \eqref{SSL3 cA} yields
\begin{equ}
\|\A_t-\A_s\|_{L_p(\Omega)}=\big\|\int_s^t b(B_r)-b(B_{\kappa_n(r)})\,dr\big\|_{L_p(\Omega)}\leq N|t-s|^{1/2+\eps}n^{-1/2-\alpha/2+\eps}
\end{equ}
for all $0\leq s\leq t\leq T_0$. 
Here the constant $N$ depends on $p,\eps,\alpha,d,\|b\|_{C^\alpha}$, but \emph{not} on $T_0$.

\emph{Step 1.5 (Girsanov transform)} An easy application of Girsanov's theorem yields
\begin{equ}\label{blabla1}
\big\|\int_s^t b(X_r^n)-b(X^n_{\kappa_n(r)})\,dr\big\|_{L_p(\Omega)}\leq N|t-s|^{1/2+\eps}n^{-1/2-\alpha/2+\eps}.
\end{equ}
In general (for example, for fractional Brownian motions) the Girsanov transformation can become involved, but for our present example this is completely straightforward.

\emph{Step 2 ("regularization bound")} Next, we estimate the quantity
\begin{equ}
\A_{T_0}=\int_0^{T_0}b(B_r+\psi_r)-b(B_r+\varphi_r)\,dt
\end{equ}
for some adapted processes $\psi,\varphi$ whose Lipschitz norm is bounded by some constant $K$.
As suggested by the above notation, we use the stochastic sewing lemma again, with $A_{s,t}$ defined as
\begin{equ}
A_{s,t}=\E^s\int_s^t b(B_r+\psi_s)-b(B_r+\varphi_s)\,dr.
\end{equ}
We do not give the details of the calculations at this point.
It is an instructive exercise to the interested reader to verify that \eqref{SSL1} and \eqref{SSL2} are satisfied with $\eps_1=\alpha/2$, $C_1=N\db{\psi-\varphi}_{\scC^0_p,[0,T_0]}$ and $\eps_2=\alpha/2$, $C_2=N\db{\psi-\varphi}_{\scC^{1/2}_p,[0,T_0]}$.
Here $N$ depends on $p,\alpha,d,K,\|b\|_{\C^\alpha}$, but \emph{not} on $T_0$.
The bound \eqref{SSL1 cA} is straightforward, with $K_1=\|b\|_{\C^0}$. Concerning \eqref{SSL2 cA}, one can write
\begin{equ}
|\E^s(\A_t-\A_s-A_{s,t})|\leq\E^s\int_s^t\big|b(B_r+\psi_r)-b(B_r+\psi_s)\big|+\big|b(B_r+\varphi_r)-b(B_r+\varphi_s)\big|\,dr,
\end{equ}
and so $K_2=2K\|b\|_{\C^\alpha}$ does the job.
Therefore, by \eqref{SSL3 cA}, we get
\begin{equs}
\|\A_t-\A_s\|_{L_p(\Omega)}&=\big\|\int_s^t  b(B_r+\psi_r)-b(B_r+\varphi_r)\,dr\big\|_{L_p(\Omega)}
\\
&\leq N
|t-s|^{1/2+\alpha/2}\db{\psi-\varphi}_{\scC^0_p,[0,T_0]}+N
|t-s|^{1+\alpha/2}\db{\psi-\varphi}_{\scC^{1/2}_p,[0,T_0]}.
\end{equs}
We will only apply the following simple corollary of this bound: if $\psi_0=\varphi_0$, then
\begin{equ}\label{blabla2}
\big\|\int_s^t b(B_r+\psi_s)-b(B_r+\varphi_s)\,dr\big\|_{L_p(\Omega)}\leq N
|t-s|^{1/2+\alpha/2}\db{\psi-\varphi}_{\scC^{1/2}_p,[0,T_0]}.
\end{equ}

\emph{Step 3 ("Buckling")} Let $\psi$ and $\psi^n$ be the drift component of $X$ and $X^n$, respectively:
\begin{equ}
\psi_t=x_0+\int_0^t b(X_r)\,dr,\qquad\psi^n_t=x_0+\int_0^t b(X^n_{\kappa_n(r)})\,dr.
\end{equ}
We apply \eqref{blabla1} and \eqref{blabla2} with $\varphi=\psi^n$, to get
\begin{equs}
\|(\psi-\psi^n)_t-(\psi-\psi^n)_s\|_{L_p(\Omega)}
&\leq Nn^{-1/2-\alpha/2+\eps}|t-s|^{1/2+\eps}
\\
&\quad+ N
|t-s|^{1/2+\alpha/2}\db{\psi-\psi^n}_{\scC^{1/2}_p,[0,T_0]}.
\end{equs}
Dividing by $|t-s|^{1/2}$ and take supremum over $0\leq s\leq t\leq T_0$, one gets
\begin{equ}
\db{\psi-\psi^n}_{\scC^{1/2}_p,[0,T_0]}\leq
Nn^{-1/2-\alpha/2+\eps}
+NT_0^{\alpha/2} \db{\psi-\psi^n}_{\scC^{1/2}_p,[0,T_0]}.
\end{equ}
Since so far $N$ does not depend on $T_0$, one can choose $T_0$ sufficiently small so that $NT_0^{\alpha/2}\leq 1/2$. This yields the desired bound
\begin{equ}
\db{X-X^n}_{\scC^{1/2}_p,[0,T_0]}=\db{\psi-\psi^n}_{\scC^{1/2}_p,[0,T_0]}
\leq Nn^{-1/2-\alpha/2+\eps}.
\end{equ}
\qed

Let us point out that the rate of convergence is determined by only the first step. Also, the second step is similar in spirit to the `averaging bounds' appearing in sewing-based uniqueness proofs for SDEs (see e.g. \cite{Cat-Gub, Khoa}).

In the proof of Theorem \ref{thm:main fractional}, the more difficult part will be the regularization bound.
Applying only the stochastic sewing lemma of L\^e apparently does not lead to an optimal result for $H>1/2$.
Therefore at some point one has to move from almost sure bounds (which are similar to \cite{Cat-Gub}) to $L_p$ bounds.
This requires an extension of the Davie's moment bound \cite[Proposition~2.1]{Davie} to the case of the fractional Brownian motion. This is done in Lemma~\ref{L:Daviebound} using the new stochastic sewing lemma (Theorem~\ref{thm:Sewing-lemma-better-p}).

In contrast, for Theorem \ref{thm:main multiplicative} establishing the quadrature bound will be more difficult.
In the above arguments, the heat kernel bounds have to be replaced by estimates on the transition densities of the Euler-Maruyama scheme.
These bounds are established via Malliavin calculus, this is the content of Section \ref{sec:Malliavin}.

\subsection{Sewing lemmas}

As mentioned above, the proof strategy
relies on the sewing and stochastic sewing lemmas. For the convenience of the reader, we recall them here. The first two lemmas are well-known, the third one is new.

We define for $0\leq S\leq T\leq 1$ the set $[S,T]_\leq:=\{(s,t):\,S\leq s\leq t\leq T\}$.
If $A_{\cdot,\cdot}$ is a function $[S,T]_\leq\to\R^d$, then for $s\leq u\leq t$ we put $\delta A_{s,u,t}:=A_{s,t}-A_{s,u}-A_{u,t}$. The first statement is the sewing lemma of Gubinelli.

\begin{Proposition} [{\cite[Lemma~2.1]{FeyelPradelle}, \cite[Proposition~1]{Gubin04}}]           \label{thm:det-Sewing-lemma}
Let $0\leq S\leq T\leq 1$ and let $A_{\cdot,\cdot}$ be a continuous function from $[S,T]_\leq$ to $\R^d$. Suppose that for some $\eps>0$ and $C>0$ the bound
\begin{equation}\label{SSLdelta}
|\delta A_{s,u,t}| \leq C |t-s|^{1+\eps}
\end{equation}
holds for all $S\leq s\leq u\leq t\leq T$.
Then there exists a unique function $\A:[S,T]\to \R^d$ such that $\A_S=0$ and the following bound holds for some constant $K>0$:
\begin{equation}\label{DSLuniqueness}
|\A_t	-\A_s-A_{s,t}|\leq K |t-s|^{1+\eps}, \quad (s,t)\in[S,T]_\leq.
\end{equation}
Moreover, there exists a constant $K_0$ depending only on $\eps$, $d$ such that $\A$
 in fact satisfies the above bound with $K\leq K_0 C$.
\end{Proposition}

The next statement is the stochastic extension of the above result obtained by L\^e. Recall that for any $s\ge0$ we are using the convention $\E^s[...]:=\E[...|\F_s]$.
\begin{Proposition} [{\cite[Theorem 2.4]{Khoa}}]           \label{thm:Sewing-lemma}
Let $p\geq 2$, $0\leq S\leq T\leq 1$ and let $A_{\cdot,\cdot}$ be a function $[S,T]_\leq\to L_p(\Omega,\R^d)$ such that for any $(s,t)\in[S,T]_\leq$ the random vector $A_{s,t}$ is $\F_t$-measurable. Suppose that for some $\eps_1,\eps_2>0$ and $C_1,C_2$ the bounds
\begin{equs}
\|A_{s,t}\|_{L_p(\Omega)} & \leq C_1|t-s|^{1/2+\eps_1},\label{SSL1}
\\
\|\E^s\delta A_{s,u,t}\|_{L_p(\Omega)} & \leq C_2 |t-s|^{1+\eps_2}\label{SSL2}
\end{equs}
hold for all $S\leq s\leq u\leq t\leq T$.
Then there exists a unique (up to modification) $\bF$-adapted 
process $\A:[S,T]\to L_p(\Omega,\R^d)$ such that $\A_S=0$ and the following bounds hold for some constants $K_1,K_2>0$:
\begin{align}
\|\A_t	-\A_s-A_{s,t}\|_{L_p(\Omega)} & \leq K_1 |t-s|^{1/2+\eps_1}+K_2 |t-s|^{1+\eps_2},\quad (s,t)\in[S,T]_\leq,\label{SSL1 cA}
\\
\|\E^s\big(\A_t	-\A_s-A_{s,t}\big)\|_{L_p(\Omega)} & \leq K_2|t-s|^{1+\eps_2},\quad (s,t)\in[S,T]_\leq\label{SSL2 cA}.
\end{align}
Moreover, there exists a constant $K$ depending only on $\eps_1,\eps_2$, $d$ such that $\A$
satisfies the bound
\begin{equation}\label{SSL3 cA}
\|\A_t-\A_s\|_{L_p(\Omega)}  \leq  KpC_1 |t-s|^{1/2+\eps_1}+KpC_2 |t-s|^{1+\eps_2},\quad (s,t)\in[S,T].
\end{equation}
\end{Proposition}

The final statement of this section is new. It provides bounds on $\|\A_s-\A_t\|_{L_p(\Omega)}$ with the correct dependence on $p$: namely these bounds are of order $\sqrt p$, rather than $p$ as in \eqref{SSL3 cA}. This will be crucial for the proof of Theorem~\ref{thm:main fractional}; in particular, this would allow to extend the corresponding Davie bound \cite[Proposition~2.1]{Davie} to the case of fractional Brownian motion. The price to pay though is that the assumptions of this theorem are more restrictive than the corresponding assumptions of \cite[Theorem 2.4]{Khoa}.

\begin{Theorem}\label{thm:Sewing-lemma-better-p}
Fix $0\leq S\leq T\leq 1$. Let $(\A_t)_{t\in[S,T]}$ be an $\bF$--adapted process with values in
$\R^d$. For $(s,t)\in [S,T]_\leq$ we will write $\A_{s, t}:=\A_t-\A_s$.
Let $p\geq 2$. Suppose that for some $m\ge 2$, $\eps_1>0$, $\eps_2\ge0$, $\eps_3\ge0$, and $C_1,C_2, C_3>0$ the bounds
\begin{align}
&\|\A_{s,t}\|_{L_{p\vee m}(\Omega)}\le C_1 |t-s|^{1/2+\eps_1}\label{ass1}\\
&\|\E^s\A_{u,t}-\E^u\A_{u,t}\|_{L_m(\Omega)}\le C_1 |u-s|^{1/m+\eps_1}\label{ass11}\\
&\|\E^s\A_{s,t}\|_{L_p(\Omega)}\le C_2 |t-s|^{\eps_2}\label{ass1.1}\\
&\bigl\|\E^s[(\E^s\A_{u,t}-\E^u\A_{u,t})^2]\bigr\|_{L_{p/2}(\Omega)}\le C_3 |u-s||t-s|^{\eps_3}\label{ass2}
\end{align}
hold for all $S\leq s\leq u\leq t\leq T$.
Then there exist a universal constant $K=K(d,\eps_2,\eps_3)>0$ which does not depend on $p$, $C_j$, such that
\begin{equation}\label{mainresult}
\|\A_{t}-\A_s\|_{L_p(\Omega)}\le C_2K|t-s|^{\eps_2}+K\sqrt p\,C_3^{1/2}|t-s|^{1/2+\eps_3/2}.
\end{equation}

\end{Theorem}

\begin{Remark}
Note that the right--hand side of bound \eqref{mainresult} does not depend on $C_1$.
\end{Remark}

\begin{Remark}
Let us recall that the proof of stochastic sewing lemma in \cite{Khoa} requires to apply the BDG inequality infinitely many times but each time to a discrete-time martingale, thus yielding a constant $p$ in the right--hand side of bound \eqref{SSL3 cA}. In our proof we apply the BDG inequality only once, but to a continuous time martingale. This allows to get a better constant (namely $\sqrt p$ instead of $p$), since the constant in the BDG inequality for the continuous-time martingales is better than in the BDG inequality for general martingales.
\end{Remark}

\begin{proof}[Proof of Theorem~\ref{thm:Sewing-lemma-better-p}]
This proof is inspired by the ideas of \cite[proof of Proposition 3.2]{BM} and \cite[proof of Theorem~4.3]{Cat-Gub}. For the sake of brevity, in this proof we will write $L_p$ for
$L_p(\Omega)$. Fix $s,t\in[S,T]_{\le}$ and for $i\in\{1,\hdots,d\}$ consider a martingale $M^{i}=(M^i_r)_{r\in[s,t]}$, where
$$
M^i_r:=\E^r[\A^i_{s,t}],\quad r\in[s,t].
$$

We will frequently use the following inequality.  For $s\le u\le v\le t $ one has
\begin{equation}\label{usefulrepresentation}
|M^i_u-M^i_v|\le |\A^i_{u,v}|+|\E^u \A^i_{u,v}|+|\E^u\A^i_{v,t}-\E^v\A^i_{v,t}|.
\end{equation}
We begin by observing that
\begin{align}\label{step0}
\|\A_{s,t}\|_{L_p(\Omega)}&\le\sum_{i=1}^d\|\A^{i}_{s,t}\|_{L_p(\Omega)}=\sum_{i=1}^d\|M^{i}_{t}\|_{L_p(\Omega)}\nonumber\\
&\le \sum_{i=1}^d\|M^{i}_{s}\|_{L_p(\Omega)}+\sum_{i=1}^d\|M^{i}_{t}-M^{i}_{s}\|_{L_p(\Omega)}\nonumber\\
&=:\sum_{i=1}^dI_1^{i}+\sum_{i=1}^dI_2^{i}.
\end{align}
The first term in \eqref{step0} is easy to bound. By assumption \eqref{ass1.1} we have
\begin{equation}\label{step0.1}
I_1^{i}=\|\E^{s} \A_{s,t}^{i}\|_{L_p(\Omega)}\le C_2 |t-s|^{\eps_2}.
\end{equation}
To estimate $I_2^i$ we first observe that for each $i=1,\dots,d$ the martingale $M^{i}$ is continuous. Indeed, for any $s\le u\le v\le t$ we have using \eqref{usefulrepresentation}, \eqref{ass1}, and \eqref{ass11}
\begin{align*}
\|M^{i}_u-M^{i}_v\|_{L_m}&\le 2\|\A^{i}_{u,v}\|_{L_m}+\|\E^u\A^{i}_{v,t}-\E^v\A^{i}_{v,t}\|_{L_m}\\
&\le 3C_1|u-v|^{1/m+\eps_1}.\nonumber
\end{align*}
Therefore, the Kolmogorov continuity theorem implies that the martingale $M^{i}$ is continuous. Hence, its quadratic variation $[M^{i}]$ equals its predictable quadratic variation
$\langle M^{i}\rangle$ \cite[Theorem~I.4.52]{Shirya}. Thus, applying a version of the Burkholder--Davis--Gundy inequality with a precise bound on the constant \cite[Proposition 4.2]{Barlya}, we get that there exists a universal constant $N>0$ such that
\begin{equation}\label{martingaleestimate}
\|M^{i}_t-M^{i}_s\|_{L_p(\Omega)}\le N\sqrt p\,\|\langle M^{i}\rangle_t\|_{L_{p/2}}^{1/2}.
\end{equation}
For $n\in\N$, $j\in\{1,\hdots,n\}$ put $t^n_j:=s+(t-s)j/n$. Then, it follows from
 \cite[Theorem~2]{Jakub} that $\sum_{j=0}^{n-1}\E^{t_j^n}[(M^i_{t^n_{j+1}}-M^i_{t^n_j})^2]$
converges to $\langle M^{i}\rangle_t$ in $L_1(\Omega)$. In particular, a subsequence indexed over $n_k$ converges almost surely. Therefore, applying  Fatou's lemma,  Minkowski's inequality,  \eqref{usefulrepresentation} and using the assumptions of the theorem, we deduce
\begin{align*}
\|\langle M^{i}\rangle_t\|_{L_{p/2}}&=\Bigl\|\lim_{k\to\infty}\sum_{j=0}^{n_k-1}\E^{t_j^{n_k}}(M^{i}_{t^{n_k}_{j+1}}-M^{i}_{t^{n_k}_j})^2\Bigr\|_{L_{p/2}}
\\
&\le \liminf_{k \to\infty}\sum_{j=0}^{n_k-1}\bigl\|\E^{t_j^{n_k}}(M^{i}_{t^{n_k}_{j+1}}-M^{i}_{t^{n_k}_j})^2\bigr\|_{L_{p/2}}\\
&\le 3\lim_{k\to\infty}\sum_{j=0}^{n_k-1}\bigl(2\|\A^{i}_{t^{n_k}_{j},t^{n_k}_{j+1}}\|_{L_p(\Omega)}^2+\|\E^{t_j^{n_k}}(\E^{t_j^{n_k}}\A^{i}_{t^{n_k}_{j+1},t}-\E^{t_{j+1}^{n_k}}\A^{i}_{t^{n_k}_{j+1},t})^2\bigr\|_{L_{p/2}}\bigr)\\
&\le \lim_{k\to\infty}6C_1^2T^{1+2\eps_1}n_k^{-2\eps_1}+3\lim_{k\to\infty}C_3|t-s|^{1+\eps_3}n_k^{-1-\eps_3}\sum_{j=0}^{n_k-1}(n_k-j)^{\eps_3}\\
&\le N C_3|t-s|^{1+\eps_3}.
\end{align*}
Substituting this into \eqref{martingaleestimate} and combining this with \eqref{step0} and \eqref{step0.1}, we obtain \eqref{mainresult}.
\end{proof}

\subsection{Some useful estimates}\label{sec:prelim}

In this section we establish a number of useful technical bounds related to Gaussian kernels. Their proofs are mostly standard, however we were not able to find them in the literature. Therefore for the sake of completeness, we provide the proofs of these results in the Appendix~\ref{A:AAAAA}.

Fix an arbitrary $H\in(0,1)$. Define
\begin{equ}
c(s,t):=\sqrt{(2H)^{-1}}|t-s|^{H},\quad 0\le  s\le  t\le 1.
\end{equ}
Let $p_t$, $t>0$, be the density of a $d$-dimensional vector with independent Gaussian components each of mean zero and variance $t$:
\begin{equ}\label{eq:p-def}
p_t(x)=\frac{1}{(2\pi t)^{d/2}}\exp\Bigl(-\frac{|x|^2}{2t}\Bigr),\quad x\in\R^d.
\end{equ}
For a measurable function $f\colon\R^d\to\R$ we write $\cP_t f:=p_t\ast f$, and occasionally we denote by $p_0$ the Dirac delta function.

Our first statement provides a number of technical bounds related to the fractional Brownian motion. Its proof is placed in the Appendix~\ref{A:AAAAA}.
\begin{proposition}\label{prop:fractional}
Let $p\ge1$. The process $B^H$ has the following properties:
\begin{enumerate}[$($i$)$]
\item\label{prop 1} $\|B^H_t-B^H_s\|_{L_p(\Omega)}= N |t-s|^H$, for all $0\leq s\leq t\leq 1$, with $N=N(p,d,H)$;
\item\label{prop H} for all $0\le s\le u\le t\leq 1$, $i=1,\hdots,d$, the random variable $\E^sB_t^{H,i}-\E^uB_t^{H,i}$ is independent of $\F^s$; furthermore, this random variable is Gaussian with mean $0$ and variance
\begin{equation}\label{funcitonv}
\E(\E^sB_t^{H,i}-\E^uB_t^{H,i})^2= c^2(s,t)-c^2(u,t)=:v(s,u,t);
\end{equation}
\item\label{prop 2} $\E^s f(B^H_t)=\cP_{c^2(s,t)}f(\E^sB^H_t)$, for all $0\leq s\leq t\leq 1$;
\item\label{prop 3} $|c^2(s,t)-c^2(s,u)|\leq N|t-u||t-s|^{2H-1}$,
for all $0\leq s\leq u\leq t$ such that $|t-u|\leq |u-s|$, with $N=N(H)$;
\item\label{prop 5} $\|\E^sB^H_t-\E^sB^H_u\|_{L_p(\Omega)}\leq N|t-u||t-s|^{H-1}$, for all $0\leq s\leq u\leq t$ such that $|t-u|\le |u-s|$, with $N=N(p,d,H)$;
\end{enumerate}
\end{proposition}

The next statement gives the  heat kernel bounds which are necessary for the proofs of the main results. Its proof is also placed in the Appendix~\ref{A:AAAAA}. Recall the definition of the function $v$ in \eqref{funcitonv}.
\begin{proposition}\label{prop:HK}
Let $f\in \C^\alpha$, $\alpha\le 1$ and $\beta \in[0,1]$. The following  hold:
\begin{enumerate}[$($i$)$]
\item\label{eq:HK bound2}
There exists $N=N(d, \alpha, \beta)$ such that 
$$
\|\cP_tf\|_{\C^\beta(\R^d)}\le N t^{\frac{(\alpha-\beta)\wedge0}{2}} \|f\|_{\C^\alpha(\R^d)},
$$
for all $t\in(0,1]$.
\item \label{eq:HK bound}
For all $\delta \in (0,1]$ with $\delta\ge\frac\alpha2-\frac\beta2$, there exists $N=N(d, \alpha, \beta, \delta)$ such that 
$$
\|\cP_tf-\cP_sf\|_{\C^\beta(\R^d)}\leq N  \|f\|_{\C^{\alpha}(\R^d)} s^{\frac\alpha2-\frac\beta2-\delta}(t-s)^{\delta},
$$
for all $0\le s\leq t \le 1$.

\item \label{eq:diffnewbound}
For all  $H\in(0,1)$, there exists $N=N(d,\alpha,\beta, H)$ such that 
$$
\|\cP_{c^2(s,t)}f-\cP_{c^2(u,t)}f\|_{\C^\beta(\R^d)}\leq  N\|f\|_{\C^\alpha(\R^d)}(u-s)^{\frac12}(t-u)^{(H(\alpha-\beta)-\frac12)\wedge0},
$$
for all $0<s\leq u \le t\leq 1$.
\item\label{randomresult}
For all  $H\in(0,1)$, $p\ge2$,  there exists $N=N(d,\alpha,H,p)$ such that
$$
\|\cP_{c^2(u,t)}f(x)-\cP_{c^2(u,t)}f(x+\xi)\|_{L_p(\Omega)}
\le N\|f\|_{\C^\alpha} (u-s)^{\frac12}(t-u)^{(H\alpha-\frac12)\wedge0};
$$
for all $x\in\R^d$,  $0<s\le u\le t\le 1$ and all random vectors $\xi$ whose components are independent, $\mathcal{N}(0,v(s,u,t))$ random variables.
\end{enumerate}
\end{proposition}

Our next statement relates to the properties of H\"older norms. Its proof can be found in Appendix \ref{A:AAAAA}.
\begin{Proposition}\label{P:Holdernorms}
Let $\alpha\in\R$, $f\in\C^\alpha(\R^d,\R^k)$, $\delta\in[0,1]$. Then there  exists $N=N(\alpha,\delta,d , k)$ such that for any $x\in\R^d$
\begin{equation*}
\|f(x+\cdot)-f(\cdot)\|_{\C^{\alpha-\delta}}\le N|x|^\delta \|f\|_{\C^\alpha}.
\end{equation*}
\end{Proposition}

Finally, we will also need the following integral bounds. They follow immediately
from a direct calculation.
\begin{Proposition}
\begin{enumerate}[$($i$)$]
 \item Let $a,b>-1$, $t>0$. Then for some $N=N(a,b)$ one has
\begin{equation}
\int_0^t(t-r)^ar^b\,dr=N t^{a+b+1}.\label{eq:integral}
\end{equation}
\item Let $a>-2$, $b<1$, $t>0$. Then for some $N=N(a,b)$ one has
\begin{equation}
\Big|\int_0^t(t-r)^{a}(t^{b}r^{-b}-1)\,dr\Big|=N t^{a+1}.\label{eq:integral2}
\end{equation}
\end{enumerate}
\end{Proposition}

\subsection{Girsanov theorem for fractional Brownian  motion}\label{S:Girsan}

One of the tools which are important for the proof of Theorem~\ref{thm:main fractional} is the Girsanov theorem for fractional Brownian motion \cite[Theorem 4.9]{Ustunel}, \cite[Theorem~2]{Nualart}. We will frequently use the following technical corollary of this theorem. For the convenience of the reader we put its proof into Appendix~\ref{B:BBBBB}.

\begin{Proposition}\label{Pr:Girsanov}
Let $u\colon \Omega\times[0,1]\to\R^d$ be an  $\bF$--adapted process such that with a constant $M>0$ we have  
\begin{equation}\label{globalboundu}
\| u \|_{L_\infty(0,1)} \le M, 
\end{equation}
almost surely. Further, assume that one of the following holds:
\begin{enumerate}
\item[$($i$)$] $H\le 1/2$;
\end{enumerate}
or
\begin{enumerate}
\item[\noindent$($ii$)$] $H>1/2$ and there exists a random variable $\xi$ such that
\begin{equation}\label{intassump}
\int_0^1\Bigl(\int_0^t\frac{(t/s)^{H-1/2}|u_t-u_s|}{(t-s)^{H+1/2}}\,ds\Bigr)^2\,dt\le \xi
\end{equation}
and  $\E \exp(\lambda\xi)<\infty$ for any $\lambda>0$.
\end{enumerate}
Then there exists a probability measure $\wt \P$ which is equivalent to $\P$ such that the process $\wt B^H:=B^H+\int_0^\cdot u_s\,ds$ is a fractional Brownain motion with Hurst parameter $H$ under $\wt \P$. Furthermore for any $\lambda>0$ we have
\begin{equation}\label{densitybound}
\E\Bigl(\frac{d \P}{d\wt \P}\Bigr)^\lambda\le 
\begin{cases}
\exp(\lambda^2 NM^2)\qquad\qquad\qquad\qquad\text{if }H\in (0,1/2]\\
\exp(\lambda^2 NM^2)\E[\exp(\lambda N\xi)]\qquad \text{if }H\in(1/2,1)
\end{cases}
<\infty,
\end{equation}
where $N=N(H)$.
\end{Proposition}

In order to simplify the calculation of the integral in \eqref{intassump}, we provide the following technical but useful lemma. Since the proof is purely technical, we put its proof in the Appendix~\ref{B:BBBBB}.
\begin{Lemma}\label{L:annoyingintegral}
Let $H\in(1/2,1)$ and let $\rho\in(H-1/2,1]$. Then there exists a constant $N=N(H,\rho)$, such that for any function $f\in\C^\rho([0,1],\R^d)$ and any $n\in\N$ one has
\begin{align}\label{annoyingintegralbound}
&\int_0^1\Bigl(\int_0^t\frac{(t/s)^{H-1/2} |f_{\kappa_n(t)}-f_{\kappa_n(s)}|} {(t-s)^{H+1/2}}\,ds\Bigr)^2\,dt\le N[f]_{\C^\rho}^2.\\
&\int_0^1\Bigl(\int_0^t\frac{(t/s)^{H-1/2} |f_{t}-f_{s}|} {(t-s)^{H+1/2}}\,ds\Bigr)^2\,dt\le N[f]_{\C^\rho}^2.\label{notannoyingintegralbound}
\end{align}
\end{Lemma}

\section{Additive fractional noise}\label{sec:fractional}

In this section we provide the proof of Theorem~\ref{thm:main fractional}. We follow the strategy outlined on Section \ref{sec:outline}: In Sections \ref{SS:QE} and \ref{SS:RL} we prove the quadrature bound and the regularization bound, respectively.
Based on these bounds, the proof of the theorem is placed in Section~\ref{S:proof21}.

\subsection{Quadrature estimates}\label{SS:QE}

The goal of this subsection is to prove the quadrature bound \eqref{DKBound Xn fBM}. The proof consists of two steps. First, in Lemma~\ref{lem:(ii)} we prove this bound for the case of fractional Brownian motion; then we extend this result to the process $X$ by applying the Girsanov theorem.

Recall the definition of functions $\kappa_n$ in \eqref{kappadef} and $\gamma$ in \eqref{eq:gamma}.
\begin{Lemma}\label{lem:(ii)}
Let $H\in(0,1)$, $\alpha\in[0,1]$, $p>0$, and take $\eps\in(0,1/2]$.
Then
for all $f\in \C^\alpha$, $0\leq s\leq t\leq 1$, $n\in\N$, one has the bound
\begin{equation}\label{DKBound frac}
\Bigl\|\int_s^t (f(B^H_r)-f(B^H_{\kappa_n(r)}))\, dr\Bigr\|_{L_p(\Omega)}
\leq N\|f\|_{\C^\alpha} n^{-\gamma(\alpha, H)+\eps}|t-s|^{1/2+\eps} ,
\end{equation}
with some $N=N(p, d,\alpha,\eps, H)$.
\end{Lemma}
\begin{proof}
It suffices to prove the bound for $p\geq 2$.
Define for $0\leq s\leq t\leq 1$
$$
A_{s,t}:=\E^s \int_s^t (f(B^H_r)-f(B^H_{\kappa_n(r)}))\, dr.
$$
Then, clearly, for any $0\leq s\leq u\leq t\leq 1$
\begin{align*}
\delta A_{s,u,t}:&=A_{s,t}-A_{s,u}-A_{u,t}\\
&=\E^s \int_u^t (f(B^H_r)-f(B^H_{\kappa_n(r)}))\, dr-\E^u \int_u^t(f(B^H_r)-f(B^H_{\kappa_n(r)}))\, dr.
\end{align*}
Let us check that all the conditions of the stochastic sewing lemma (Proposition~\ref{thm:Sewing-lemma}) are satisfied.
Note that
\begin{equation*}
\E^s \delta A_{s,u,t}=0,
\end{equation*}
and so condition \eqref{SSL2} trivially holds, with $C_2=0$.
To establish \eqref{SSL1}, let $s \in [k/n, (k+1)/n)$ for some $k \in \{0,\hdots,n-1\}$.
Suppose first that $t \in [(k+4)/n, 1]$. We write
\begin{equation}\label{Astdec}
|A_{s,t}|\leq  \Big(\int_s^{(k+4)/n}
 +\int_{(k+4)/n}^t\Big) |\E^s \big( f(B^H_r)-f(B^H_{\kappa_n(r)})\big)|\, dr=:I_1+I_2.
\end{equation}

The bound for $I_1$ is straightforward: by conditional Jensen's inequality, the definition of $\C^\alpha$ norm,  and Proposition~ \ref{prop:fractional} \eqref{prop 1} we have
\begin{align}\label{AstI1}
\| I_1\|_{L_p(\Omega)} &\leq
\int_s^{(k+4)/n} \| f(B^H_r)-f(B^H_{\kappa_n(r)}) \|_{L_p(\Omega)} \, dr
\nn\\
&\leq N
\| f\|_{\C^\alpha}n^{-1-\alpha H} \leq N \| f\|_{\mathcal{C}^\alpha} n^{-\gamma+\eps}|t-s|^{1/2+\eps},
\end{align}
where the last inequality follows from the fact that $n^{-1}\leq|t-s|$.

Now let us estimate $I_2$. Using Proposition \ref{prop:fractional} \eqref{prop 2}, we derive
\begin{align}\label{AstI2}
I_2\le&\int_{(k+4)/n}^t|\cP_{c^2(s,r)}f(\E^sB^H_r)-\cP_{c^2(s,\kappa_n(r))}f(\E^sB^H_r)|\,dr
\nn\\
&\quad+\int_{(k+4)/n}^t |\cP_{c^2(s,\kappa_n(r))}f(\E^sB^H_r)-\cP_{c^2(s,\kappa_n(r))}f(\E^sB^H_{\kappa_n(r)})|\,dr\nn\\
=:&I_{21}+I_{22}.
\end{align}
To bound $I_{21}$, we apply
Proposition \ref{prop:HK} \eqref{eq:HK bound} with $\beta=0$, $\delta=1$ and Proposition \ref{prop:fractional} \eqref{prop 3}. We get
\begin{align}\label{AstI21}
\|I_{21}\|_{L_p(\Omega)}&\leq N\| f\|_{\C^\alpha} \int_{(k+4)/n}^t\big(c^2(s,r)-c^2(s,\kappa_n(r))\big)c^{\alpha-2}(s,\kappa_n(r))\,dr
\nn\\
&\leq
N\| f\|_{\C^\alpha}\int_{(k+4)/n}^t
n^{-1}|r-s|^{2H-1}|r-s|^{H(\alpha-2)}\,dr
\nn\\
&\leq
N\| f\|_{\C^\alpha}n^{-1}\int_s^t |r-s|^{-1+\alpha H}\,dr
\nn\\
&\leq N\| f\|_{\C^\alpha} n^{-1}|t-s|^{\alpha H}.
\end{align}

To deal with $I_{22}$,  we use Proposition~\ref{prop:HK} \eqref{eq:HK bound2} with $\beta=1$ and Proposition~\ref{prop:fractional} \eqref{prop 5}. We deduce
\begin{align}\label{AstI22}
\|I_{22}\|_{L_p(\Omega)}&
\leq N\| f\|_{\C^\alpha}\int_{(k+4)/n}^t \|\E^sB^H_r-\E^sB^H_{\kappa_n(r)}\|_{L_p(\Omega)}c^{\alpha-1}(s,\kappa_n(r))\,dr
\nn\\
&\leq N\| f\|_{\C^\alpha}\int_{(k+4)/n}^t n^{-1}|r-s|^{H-1}|r-s|^{-H(1-\alpha)}\,dr
\nn\\
&\leq N\| f\|_{\C^\alpha}n^{-1}|t-s|^{\alpha H},
\end{align}
where in the second inequality we have also used that $\kappa_n(r)-s\ge (r-s)/2$.
Combining \eqref{AstI21} and \eqref{AstI22}, and taking again into account that $n^{-1}\leq|t-s|$, we get
\begin{equ}
\| I_2\|_{L_p(\Omega)} \leq  N\| f\|_{\C^\alpha} n^{-\gamma+\eps}|t-s|^{1/2+\eps}.
\end{equ}
Recalling \eqref{AstI1}, we finally conclude
\begin{equ}
\| A_{s,t}\|_{L_p(\Omega)}\leq N \| f\|_{\mathcal{C}^\alpha} n^{-\gamma+\eps}|t-s|^{1/2+\eps}.
\end{equ}
It remains to show the same bound for $t \in (s, (k+4)/n]$. However this is almost straightforward. We write
\begin{align*}
\|A_{s,t}\|_{L_p(\Omega)} &\leq \int_s^t \| f(B^H_r)-f(B^H_{\kappa_n(r)}) \|_{L_p(\Omega)} \, dr
\\ &  \leq N \|f\|_{\mathcal{C}^\alpha} n^{-\alpha H}|t-s|
 \leq N \| f\|_{\mathcal{C}^\alpha} n^{-\gamma+\eps}|t-s|^{1/2+\eps},
\end{align*}
where the last inequality uses that in this case $|t-s|\leq 4 n^{-1}$.
Thus, \eqref{SSL1} holds, with $C_1:=N \| f\|_{\mathcal{C}^\alpha} n^{-\gamma+\eps}$, $\eps_1:=\eps$.

Thus all the conditions of the stochastic sewing lemma are satisfied.
The process
$$
\tilde \A_t:=\int_0^t  (f(B^H_r)-f(B^H_{\kappa_n(r)}))\,dr
$$
is also $\bF$-adapted, satisfies \eqref{SSL2 cA} trivially (the left-hand side is $0$), and
$$
\| \tilde{\mathcal{A}}_t-\tilde{\mathcal{A}}_s-A_{s,t}\|_{L_p(\Omega)} \leq \|f\|_{\C^0} |t-s|\leq N |t-s|^{1/2+\eps},
$$
which shows that it also satisfies \eqref{SSL1 cA}.
Therefore by uniqueness $\A_t=\tilde \A_t$. The bound \eqref{SSL3 cA} then yields precisely \eqref{DKBound frac}.
\end{proof}

\begin{Lemma}\label{lem:we need better labels}
Let $H\in(0,1)$, $\alpha\in[0,1]$ such that $\alpha>1-1/(2H)$, $p>0$, $\eps\in(0,1/2]$. Let $b\in\C^\alpha$ and $X^n$ be the solution of \eqref{eq:main approx frac}.
Then for all $f\in\C^\alpha$, $0\leq s\leq t\leq 1$, $n\in\N$, one has the bound
\begin{equation}\label{DKBound Xn fBM}
\Bigl\|\int_s^t (f(X_r^n)-f(X^n_{\kappa_n(r)}))\, dr\Bigr\|_{L_p(\Omega)}
\leq N\|f\|_{\C^\alpha} |t-s|^{1/2+\eps} n^{-\gamma+\eps}
\end{equation}
with some $N=N(\|b\|_{\C^\alpha},p, d,\alpha,\eps, H)$.
\end{Lemma}

\begin{proof}
Without loss of generality, we assume $\alpha<1$. Let
$$
\psi^n(t):=\int_0^t b(X^n_{\kappa_n(t)})\,dt.
$$
Let us apply the Girsanov theorem (Theorem~\ref{Pr:Girsanov}) to the function $u(t)=b(X^n_{\kappa_n(t)})$. First let us check that all the conditions of this theorem hold.

First, we obviously have $|u(t)|\le  \|b\|_{\C^0}$, and thus  \eqref{globalboundu} holds with $M= \|b\|_{\C^0}$.

Second, let us check condition \eqref{intassump} in the case $H>1/2$. Fix $\lambda>0$ and small $\delta>0$ such that $\alpha(H-\delta)>H-1/2$; such $\delta$ exists thanks to the assumption $\alpha>1-1/(2H)$. We apply Lemma~\ref{L:annoyingintegral} for the function $f:=b(X^n)$ and $\rho:=\alpha(H-\delta)$. We have
\begin{align*}
\int_0^1\Bigl(\int_0^t\frac{(t/s)^{H-1/2}|b(X^n_{\kappa_n(t)})-b(X^n_{\kappa_n(s)})|}
{(t-s)^{H+1/2}}\,ds\Bigr)^2\,dt&\le N[b(X^n)]_{\C^{\alpha(H-\delta)}}^2\\
&=N\|b\|_{\C^{\alpha}}^2[X^n]_{\C^{H-\delta}}^{2\alpha}\\
&\le N \|b\|_{\C^{\alpha}}^2 (\|b\|_{\C^{0}}^{2\alpha}+ [B^H]_{\C^{H-\delta}}^{2\alpha})=:\xi.
\end{align*}
Therefore,
\begin{equation}\label{exponentmoment}
\E e^{\lambda\xi}\le N(\|b\|_{\C^{\alpha}},\alpha,\delta,H,\lambda)<\infty,
\end{equation}
where we used the fact that the H\"older constant
$[B^H]_{\C^{H-\delta}}$ satisfies $\E \exp(\lambda [B^H]_{\C^{H-\delta}}^{2\alpha})\leq N$ for any $\lambda\ge0$. Thus, condition~\eqref{intassump} is satisfied. Hence all the conditions of Theorem~\ref{Pr:Girsanov} hold. Thus, there exists a probability measure $\wt\P$ equivalent to $\P$ such that the process $\wt B^H:=B^H+\psi^n$ is a
fractional $H$-Brownian motion on $[0,1]$ under $\wt\P$.

Now we can derive the desired bound \eqref{DKBound Xn fBM}. We have
\begin{align}\label{eq:Before-Girsanov}
&\E^{\P} \Bigl|  \int_s^t \left( f(X^n_r)- f(X^n_{\kappa_n(r)}) \right) \, dr \Bigr|^p\nn\\
&\qquad=\E^{\wt\P} \Bigl[\Bigl|  \int_s^t \left( f(X^n_r)- f(X^n_{\kappa_n(r)}) \right) \, dr \Bigr|^p\frac{d\P}{d\wt\P}\Bigr]\nn\\
&\qquad\le \Bigl(\E^{\wt\P} \Bigl|  \int_s^t \left( f(X^n_r)- f(X^n_{\kappa_n(r)}) \right) \, dr \Bigr|^{2p}\Bigr)^{1/2}\Bigl(\E^{\wt\P}\Bigl[\frac{d\P}{d\wt\P}\Bigr]^2\Bigr)^{1/2}
\nn\\
&\qquad=\Bigl(\E^{\wt\P} \Bigl|  \int_s^t \left( f(\wt B^H_r+x_0^n)- f(\wt B^H_{\kappa_n(r)}+x_0^n) \right) \, dr \Bigr|^{2p}\Bigr)^{1/2}\Bigl(\E^{\P}\frac{d\P}{d\wt\P}\Bigr)^{1/2}
\nn\\
&\qquad=\Bigl(\E^{\P} \Bigl|  \int_s^t \left( f( B^H_r+x_0^n)- f( B^H_{\kappa_n(r)}+x_0^n) \right) \, dr \Bigr|^{2p}\Bigr)^{1/2}\Bigl(\E^{\P}\frac{d\P}{d\wt\P}\Bigr)^{1/2}.
\end{align}
Taking into account \eqref{exponentmoment}, we deduce by Theorem~\ref{Pr:Girsanov} that
\begin{equation*}
\E^{\P}\frac{d\P}{d\wt\P}\le N(\|b\|_{\C^{\alpha}},\alpha,\delta,H,\lambda).
\end{equation*}
Hence, using \eqref{DKBound frac}, we can continue \eqref{eq:Before-Girsanov} in the following way:
\begin{equation*}
\E^{\P} \Bigl|  \int_s^t \left( f(X^n_r)- f(X^n_{\kappa_n(r)}) \right) \, dr \Bigr|^p\le N\|f\|_{\C^\alpha}^p n^{-p(\gamma(\alpha, H)+\eps)}|t-s|^{p(1/2+\eps)},
\end{equation*}
which implies the statement of the theorem.
\end{proof}

\subsection{A regularization lemma}\label{SS:RL}

The goal of this subsection is to establish the regularization bound \eqref{ourboundasreg}. Its proof consists of a number of steps. First, in Lemma~\ref{L:Daviebound} we derive an extension of the corresponding bound of Davie \cite[Proposition~2.1]{Davie} for the fractional Brownian motion case. It is important that the right--hand side of this bound depends on $p$ as $\sqrt p$ (rather than $p$); this will be crucial later in the proof of Lemma~\ref{L:asbound} and Theorem~\ref{thm:main fractional}. Then in Lemma~\ref{L:asbound} we obtain the pathwise version of this lemma and extend it to a wider class of processes (fractional Brownian motion with drift instead of a fractional Brownian motion). Finally, in Lemma~\ref{L:finfin} we obtain the desired regularization bound.

\begin{Lemma}\label{L:Daviebound}
Let $H\in(0,1)$, $\alpha\in(-1/(2H),0]$. Let $f\in\C^\infty$. Then there exists a constant $N=N(d,\alpha,H)$ such that for any $p\ge2$, $s,t\in[0,1]$ we have
\begin{equation}\label{ourbound}
\Bigl\|\int_s^t f(B_r^H)\,dr\Bigr\|_{L_p(\Omega)}\le N\sqrt p\|f\|_{\C^\alpha} (t-s)^{H\alpha+1}.
\end{equation}
\end{Lemma}
\begin{Remark}
Note that the right--hand side of bound \eqref{ourbound} depends only on the norm of $f$ in $\C^\alpha$ and does not depend on the norm of $f$ in other H\"older spaces.
\end{Remark}

\begin{proof}[Proof of Lemma~\ref{L:Daviebound}]

Fix $p\ge2$. We will apply Theorem~\ref{thm:Sewing-lemma-better-p} to the process
$$
\A_{t}:=\int_0^t f(B_r^H)\,dr, \quad t\in[0,1].
$$
As usual, we write $\A_{s,t}:=\A_t-\A_s$. Let us check that all the conditions of that theorem hold with $m=4$

It is very easy to see that
$$
\|\A_{s,t}\|_{L_{p\vee 4}(\Omega)}\le \|f\|_{\C^0}|t-s|.
$$
Thus \eqref{ass1} holds. By Proposition~\ref{prop:fractional} \eqref{prop 2} and Proposition~\ref{prop:HK} \eqref{eq:HK bound2} we have for some $N_1=N_1(d,\alpha,H)$ (recall that by assumptions $\alpha\le 0$)
\begin{equation}\label{thankyouverymuch}
|\E^s \A_{s,t}|\le \int_s^t |P_{c^2(s,r)}f(\E^s B_r^H)|dr\le N_1\|f\|_{\C^\alpha}(t-s)^{H\alpha+1}.
\end{equation}
Hence
$$
\|\E^s \A_{s,t}\|_{L_p(\Omega)}\le N_1\|f\|_{\C^\alpha} (t-s)^{H\alpha+1}
$$
and condition \eqref{ass1.1} is met. We want to stress here that the constant $N_1$ here \emph{does not} depend on $p$ (this happens thanks to the a.s. bound \eqref{thankyouverymuch}; it will be crucial later in the proof)

Thus, it remains to check conditions \eqref{ass11} and \eqref{ass2}. Fix $0\le s\le u\le t\le 1$. Using Proposition~\ref{prop:fractional} \eqref{prop 2}, we get
\begin{align}\label{regularitylemmanew}
\E^s\A_{u,t}-\E^u\A_{u,t}&=\int_u^t \bigl(P_{c^2(s,r)}f(\E^s B_r^H)-
P_{c^2(u,r)}f(\E^u B_r^H)\bigr)\,dr\nn\\
&=\int_u^t \bigl(P_{c^2(s,r)}f(\E^s B_r^H)-
P_{c^2(s,r)}f(\E^u B_r^H)\bigr)\,dr\nn\\
&\phantom{\le}+\int_u^t \bigl(P_{c^2(s,r)}f(\E^u B_r^H)-
P_{c^2(u,r)}f(\E^u B_r^H)\bigr)\,dr\nn\\
&=: I_1+I_2.
\end{align}
Note that by Proposition~\ref{prop:fractional} \eqref{prop H},
the random vector $\E^u B_r^H-\E^s B_r^H$ is independent of $\F^s$. Taking this into account and applying the conditional Minkowski inequality, we get
\begin{align}\label{regularitylemmanewI1}
\Bigl(\E^s |I_1|^4\Bigr)^{\frac14}&\le \int_u^t \Bigl(\E^s\bigl[P_{c^2(s,r)}f(\E^s B_r^H)-
P_{c^2(s,r)}f(\E^u B_r^H)\bigr]^4\Bigr)^{\frac14}\,dr\nn\\
&\le \int_u^t g_r(\E^s B_r^H)\,dr,
\end{align}
where for $x\in\R^d$, $r\in[u,t]$  we denoted
\begin{equation*}
g_r(x):=\|P_{c^2(s,r)}f(x)-P_{c^2(s,r)}f(x+\E^u B_r^H-\E^s B_r^H)\|_{L_4(\Omega)}.
\end{equation*}
By Proposition~\ref{prop:fractional} \eqref{prop H},
the random vector $\E^u B_r^H-\E^s B_r^H$ is Gaussian and consists of $d$ independent components with each component of mean $0$ and variance $v(s,u,t)$ (recall its definition in \eqref{funcitonv}). Hence Proposition~\ref{prop:HK} \eqref{randomresult} yields now for some $N_2=N_2(d,\alpha,H)$ and all $x\in\R^d$, $r\in[u,t]$
\begin{equation*}
g_r(x)\le N_2\|f\|_{\C^\alpha}(u-s)^{\frac12}(r-u)^{H\alpha-\frac12}.
\end{equation*}
Substituting this into \eqref{regularitylemmanewI1}, we finally get
\begin{equation}\label{regularitylemmanewI1f}
\Bigl(\E^s |I_1|^4\Bigr)^{\frac14}\le N_2\|f\|_{\C^\alpha}(u-s)^{\frac12}\int_u^t (r-u)^{H\alpha-\frac12}\,dr\le  N_3\|f\|_{\C^\alpha}(u-s)^{\frac12}
(t-u)^{H\alpha+\frac12},
\end{equation}
for some $N_3=N_3(d,\alpha,H)$ where we used that, by assumptions, $H\alpha-1/2>-1$.

Similarly, using Proposition~\ref{prop:HK} \eqref{eq:diffnewbound} with $\beta=0$, we get
for some $N_4=N_4(d,\alpha,H)$
\begin{equation}\label{regularitylemmanewI2}
|I_2|\le N\|f\|_{\C^\alpha}(u-s)^{\frac12}\int_u^t (r-u)^{H\alpha-\frac12}\,dr\le  N_4\|f\|_{\C^\alpha}(u-s)^{\frac12}
(t-u)^{H\alpha+\frac12},
\end{equation}
where again we used that, by assumptions, $H\alpha-1/2>-1$. We stress that both $N_3$, $N_4$ \emph{do not} depend on $p$.

Now to verify \eqref{ass11}, we note that by \eqref{regularitylemmanew}, \eqref{regularitylemmanewI1f},\eqref{regularitylemmanewI2}, we have
\begin{align}\label{regularitylemmanewL4}
\|\E^s\A_{u,t}-\E^u\A_{u,t}\|_{L_4(\Omega)}&\le \|I_1\|_{L_4(\Omega)}+
\|I_2\|_{L_4(\Omega)}\nn\\
&\le \bigl(\E [\E^s|I_1|^4]\bigr)^{\frac14}+\|I_2\|_{L_4(\Omega)}\nn\\
&\le (N_3+N_4)\|f\|_{\C^\alpha}(u-s)^{\frac12}.
\end{align}
 Thus, condition \eqref{ass11} holds.

In a similar manner we check \eqref{ass2}. We have
\begin{align*}
\E^s[|\E^s\A_{u,t}-\E^u\A_{u,t}|^2]&\le 2 \E^s|I_1|^2 +2 \E^s|I_2|^2\le 2\bigl(\E^s|I_1|^4\bigr)^{1/2} +2 \E^s|I_2|^2\\
&\le 2(N_3^2+N_4^2)\|f\|_{\C^\alpha}^2(u-s)(t-u)^{2H\alpha+1}.
\end{align*}
Thus,
\begin{equation*}
\bigl\|\E^s[|\E^s\A_{u,t}-\E^u\A_{u,t}|^2]\bigr\|_{L_{p/2}(\Omega)}\le 2(N_3^2+N_4^2)\|f\|_{\C^\alpha}^2(u-s)(t-u)^{2H\alpha+1}
\end{equation*}
and the constant $2(N_3^2+N_4^2)$ does not depend on $p$. Therefore condition \eqref{ass2} holds.

Thus all the conditions of Theorem~\ref{thm:Sewing-lemma-better-p} hold. The statement of the theorem follows now from \eqref{mainresult}.
\end{proof}

To establish the regularization bound we need the following simple corollary of the above lemma.
\begin{Corollary}\label{L:Davieboundholder}
Let $H\in(0,1)$, $\delta\in(0,1]$, $\alpha-\delta\in(-1/(2H),0]$. Let $f\in\C^\infty$. Then there exists a constant $N=N(d,\alpha,H,\delta)$ such that for any $p\ge2$, $s,t\in[0,1]$, $x,y\in\R^d$ we have
\begin{equation}\label{ourboundholder}
\Bigl\|\int_s^t (f(B_r^H+x)-f(B_r^H+y))\,dr\Bigr\|_{L_p(\Omega)}\le N\sqrt p\|f\|_{\C^\alpha} (t-s)^{H(\alpha-\delta)+1}|x-y|^\delta.
\end{equation}
\begin{proof}
Fix $x,y\in \R^d$. Consider a function $g(z):=f(z+x)-f(z+y)$, $z\in\R^d$. Then, by Lemma~\ref{L:Daviebound}
\begin{align*}
\Bigl\|\int_s^t (f(B_r^H+x)-f(B_r^H+y))\,dr\Bigr\|_{L_p(\Omega)}&=
\Bigl\|\int_s^t g(B_r^H)\,dr\Bigr\|_{L_p(\Omega)}\\
&\le N\sqrt p\|g\|_{\C^{\alpha-\delta}} (t-s)^{H(\alpha-\delta)+1}.
\end{align*}
The corollary follows now immediately from Proposition~\ref{P:Holdernorms}.
\end{proof}
\end{Corollary}

The next lemma provides a pathwise version of bound \eqref{ourboundholder}. It also allows to replace fractional Brownian motion by fractional Brownian motion with a drift.
\begin{Lemma}\label{L:asbound}
Let $H\in(0,1)$, $\alpha>1-1/(2H)$, $\alpha\in[0,1]$, $f\in\C^\infty$.  Let $\psi\colon\Omega\times[0,1]\to\R^d$ be an $\mathbb{F}$--adapted process such that 
$\psi_0$ is deterministic and for some $R>0$
\begin{equation}\label{psicond0}
\|\psi\|_{\C^1([0,1],\R^d)}\le R,\quad a.s.
\end{equation}
Suppose that for some $\rho>H+1/2$ we have for any $\lambda>0$
\begin{equation}\label{psicond}
\E \exp\big(\lambda\|\psi\|^2_{\C^{\rho}([0,1],\R^d)}\big)=:G(\lambda)<\infty.
\end{equation}
Then for any $M>0$, $\eps>0$, $\eps_1>0$ there exists a constant $N=N(d,\alpha,H,\eps,\eps_1,G,R,M)$ and a random variable $\xi$ finite almost everywhere such that for any $s,t\in[0,1]$, $x,y\in\R$, $|x|, |y|\le M$ we have
\begin{equation}\label{ourboundas}
\Bigl|\int_s^t (f(B_r^H+\psi_r+x)-f(B_r^H+\psi_r+y))\,dr\Bigr|\le \xi \|f\|_{\C^\alpha} (t-s)^{H(\alpha-1)+1-\eps}|x-y|
\end{equation}
and
\begin{equation}\label{finsecondmoment}
\E \exp(\xi^{2-\eps_1})<N<\infty.
\end{equation}
\end{Lemma}
\begin{proof}
First we consider the case $\psi\equiv0$. Fix $\eps,\eps_1>0$.  By the fundamental theorem of calculus we observe that for any $x,y\in\R^d$, $0\le s \le t \le 1$
\begin{equation}\label{fequality}
\int_s^t (f(B_r^H+x)-f(B_r^H+y))\,dr=(x-y)\cdot\int_0^1\int_s^t
 \nabla f(B_r^H+\theta x+(1-\theta)y)\,dr\,d\theta.
\end{equation}
Consider the process
\begin{equation*}
F(t,z):=\int_0^t \nabla f(B_r^H+z)\,dr.
\end{equation*}
Take $\delta>0$ such that $\alpha-1-\delta>1/(2H)$.
By Lemma~\ref{L:Daviebound} and Corollary~\ref{L:Davieboundholder},  there exists $N_1=N_1(\alpha,d,H,\delta)$ such that for any $p\ge2$, $s,t\in [0,1]$, $x,y\in\R^d$ we have 
\begin{align*}
\|F(t,x)-F(s,y)\|_{L_p(\Omega)}&\le \|F(t,x)-F(s,x)\|_{L_p(\Omega)}+\|F(s,x)-F(s,y)\|_{L_p(\Omega)}\\
&\le N_1\sqrt p\|\nabla f\|_{\C^{\alpha-1}}((t-s)^{H(\alpha-1)+1}+|x-y|^\delta).
\end{align*}
We stress that $N_1$ does not depend on $p$. Taking into account that the process $F$ is continuous (because $f\in\C^\infty)$, we derive from the above bound and the Kolmogorov continuity theorem (\cite[Theorem 1.4.1]{Kunita}) that for any $p$ large enough one has
\begin{equation}
\sup_{\substack{x,y\in\R^d, |x|,|y|\le M\\s,t\in[0,1]}}
\frac{|F(t,x)-F(s,y)|}{(t-s)^{H(\alpha-1)+1-\eps}+|x-y|^{\delta/2}}=: \xi\| f\|_{\C^{\alpha}}<\infty\,\,a.s.,
\end{equation}
and $\|\xi\|_{L_p(\Omega)}\le NN_1\sqrt{p}$, where $N=N(\alpha,d,H,\delta,\eps,M)$. Since $N$ and $N_1$ do not depend on $p$, we see that by the Stirling formula
\begin{equation}\label{secondmomentexpxi}
\E \exp(\xi^{2-\eps_1})=\sum_{n=0}^\infty  \frac{\E\xi^{n(2-\eps_1)}}{n!}\le
\sum_{n=0}^\infty  \frac{(NN_1)^{n(2-\eps_1)}n^{n(1-\eps_1/2)}}{n!}<\infty
\end{equation}
Therefore we obtain from \eqref{fequality} that for any $x,y\in\R^d$, $|x|,|y|\le M$ we have
\begin{align}\label{prelimfinboundcor}
\nn
 &\Bigl|\int_s^t (f(B_r^H+x)-f(B_r^H+y))\,dr\Bigr|
\\
 & \qquad \qquad \le   |x-y|\int_0^1
|(F(t,\theta x+(1-\theta)y)-F(s,\theta x+(1-\theta)y))|\,d\theta\nn\\
 &   \qquad \qquad \le \xi\| f\|_{\C^{\alpha}} (t-s)^{H(\alpha-1)+1-\eps}|x-y|.
\end{align}

Now we consider the general case. Assume that the function $\psi$ satisfies \eqref{psicond}. Then by Proposition~\ref{Pr:Girsanov}, bound \eqref{notannoyingintegralbound} and assumption \eqref{psicond} the process
\begin{equation*}
\wt B_t:=B_t+\psi_t-\psi_0
\end{equation*}
is a fractional Brownian motion with Hurst parameter $H$ under some probability measure $\wt\P$ equivalent to $\P$. This yields from \eqref{prelimfinboundcor} (we apply this bound with $M+|\psi_0|$ in place of $M$)
\begin{align*}
\Bigl|\int_s^t (f(B_r^H+\psi_r+x)-f(B_r^H+\psi_r+y))\,dr\Bigr|&
=\Bigl|\int_s^t (f(\wt B_r^H+x+\psi_0)-f(\wt B_r^H+y+\psi_0))\,dr\Bigr|\\
&\le  \eta\| f\|_{\C^{\alpha}} |x-y|
\end{align*}
where $\eta$ is a random variable with $\E^{\wt \P} \exp(\eta^{2-\eps_1})<\infty$. Note that we have used here our assumption that $\psi_0$ is non-random. The latter implies that for any $\eps_2>\eps_1$
\begin{align*}
\E^{\P} \exp(\eta^{2-\eps_2})&=\E^{\wt \P}\Bigl[ \exp(\eta^{2-\eps_2})
\frac{d \P}{d \wt \P}\Bigr]\\
&\le \Bigl(\E^{\wt \P} \exp(2\eta^{2-\eps_2})\Bigr)^{1/2}
\Bigl(\E^{ \P}\frac{d\P}{d\wt \P} \Bigr)^{1/2}\\
&\le \Bigl(\E^{\wt \P} \exp(2\eta^{2-\eps_2})\Bigr)^{1/2}
e^{NR}\E^{ \P} \exp(N \|\psi\|^2_{\C^{\rho}([0,1],\R^d)})
\end{align*}
where the last inequality follows from \eqref{densitybound}
and \eqref{notannoyingintegralbound}. This concludes the proof of the theorem.
\end{proof}

Now we are ready to present the main result of this subsection, the regularization lemma. 
\begin{Lemma}\label{L:finfin}
Let $H\in(0,1)$, $\alpha>1-1/(2H)$, $\alpha\in[0,1]$, $p\ge2$, $f\in\C^\alpha$, $\eps,\eps_1>0$. Let $\tau\in(H(1-\alpha),1)$. Let $\phi, \psi\colon\Omega\times[0,1]\to\R^d$ be $\mathbb{F}$--adapted processes satisfying  condition  \eqref{psicond0}. Assume that $\psi$ satisfies additionally  \eqref{psicond} for some $\rho>H+1/2$, $\rho\in[0,1]$. Suppose that $\psi_0$ and $\phi_0$ are deterministic.
 
Then there exists a constant $N=N(H,\alpha,p,d,\tau,G,R,\eps,\eps_1)$ such that for any $L>0$,
and any $s,t\in[0,1]$ we have
\begin{align}\label{ourboundasreg}
&\Bigl\|\int_s^t (f(B_r^H+\phi_r)-f(B_r^H+\psi_r))\,dr\Bigr\|_{L_p(\Omega)}\nn\\
&\quad\le
NL \|f\|_{\C^\alpha}
(t-s)^{H(\alpha-1)+1-\eps}\big(\|\phi_s-\psi_s\|_{L_p(\Omega)}+
\|[\phi-\psi]_{C^{\tau}([s,t])}\|_{L_p(\Omega)}(t-s)^\tau\big)\nn\\
&\qquad+ N \|f\|_{\C^0}|t-s|\exp(-L^{2-\eps_1}).
\end{align}
\end{Lemma}
\begin{proof}
We begin with assuming further that  $f\in\C^\infty(\R^d,\R^d)$.
Fix $S,T\in[0,1]_{\le}$, $\eps_1>0$. Choose any $\eps>0$ small enough such that
\begin{equation}\label{epscond}
H(\alpha-1)-\eps+\tau>0.
\end{equation}
 
 Let us apply the deterministic sewing lemma (Proposition~\ref{thm:det-Sewing-lemma}) to the process
\begin{equation*}
A_{s,t}:=\int_s^t (f(B_r^H+\psi_r+\phi_s-\psi_s)-f(B_r^H+\psi_r))\,dr,\quad (s,t)\in[S,T]_{\le}.
\end{equation*}
Let us check that all the conditions of the above lemma are satisfied.

First, the process $A$ is clearly continuous, since $f$ is bounded. Then, using Lemma~\ref{L:asbound} with $M:=4R$, we derive that for any $S\le s\le u\le T$ there exists a random variable $\xi$ with $\E \exp(\xi^{2-\eps_1})\le N=N(d,\alpha,H,\eps,\eps_1,G,|\phi_0|,|\psi_0|,R)<\infty$ such that
\begin{align*}
|\delta A_{s,u,t}|&=\Bigl|\int_u^t (f(B_r^H+\psi_r+\phi_u-\psi_u)-f(B_r^H+\psi_r+\phi_s-\psi_s))\,dr\Bigr|\\
&\le\xi \|f\|_{\C^\alpha}|(\psi_u-\phi_u)-(\psi_s-\phi_s)|(t-s)^{H(\alpha-1)+1-\eps}\\
&\le \xi \|f\|_{\C^\alpha}[\psi-\phi]_{\C^\tau([S,T])}(t-s)^{H(\alpha-1)+1-\eps+\tau}.
\end{align*}
Since, by \eqref{epscond}, $H(\alpha-1)+1-\eps+\tau>1$, we see that condition  \eqref{SSLdelta} is satisfied with $C=\xi \|f\|_{\C^\alpha}[\psi-\phi]_{\C^\tau([S,T])}$. Thus, all the conditions of Proposition~\ref{thm:det-Sewing-lemma} hold. By setting now
$$
\tilde\A_t:=\int_s^t (f(B_r^H+\phi_r)-f(B_r^H+\psi_r))\,dr,
$$
we see that for $S\le s\le t\le T$
\begin{align*}
|\tilde\A_t-\tilde\A_s-A_{s,t}|&=\Bigl|\int_s^t (f(B_r^H+\phi_r)-f(B_r^H+\psi_r+\phi_s-\psi_s))\,dr\Bigr|\Bigr|\\
&\le \|f\|_{\C^1}[\psi-\phi]_{\C^\tau([S,T])}|t-s|^{1+\tau}
\\
&\le \|f\|_{\C^1}[\psi-\phi]_{\C^\tau([S,T])}|t-s|^{H(\alpha-1)+1-\eps+\tau}.
\end{align*}
Thus, the process $\tilde\A$ satisfies \eqref{DSLuniqueness} and therefore coincides with $\A$. Proposition~\ref{thm:det-Sewing-lemma} implies now that for any $S\le s\le t\le T$
\begin{align*}
&\Bigl|\int_s^t (f(B_r^H+\phi_r)-f(B_r^H+\psi_r))\,dr\Bigr|\\
&\quad\le |A_{s,t}|+N
\xi \|f\|_{\C^\alpha}[\psi-\phi]_{\C^\tau([S,T])}(t-s)^{H(\alpha-1)+1-\eps+\tau}\\
&\quad\le N\xi \|f\|_{\C^\alpha}(t-s)^{H(\alpha-1)+1-\eps}\bigl(|\psi-\phi|_{\C^0([S,T])}+
[\psi-\phi]_{\C^\tau([S,T])}(t-s)^{\tau}\bigr),
\end{align*}
where the bound on $|A_{s,t}|$ follows again from Lemma~\ref{L:asbound}. By putting in the above bound $s=S$ and $t=T$ and using that $|\psi-\phi|_{\C^0([S,T])}\le  |\psi_S-\phi_S|+ [\psi-\phi]_{\C^\tau([S,T])}(T-S)^{\tau}$, we obtain for 
 $S,T\in[0,1]_{\le}$
\begin{multline*}
\Bigl|\int_S^T (f(B_r^H+\phi_r)-f(B_r^H+\psi_r))\,dr\Bigr|\\
\le N\xi \|f\|_{\C^\alpha}(T-S)^{H(\alpha-1)+1-\eps}\bigl(|\psi_S-\phi_S|+
[\psi-\phi]_{\C^\tau([S,T])}(T-S)^{\tau}\bigr).
\end{multline*}
On the other hand, we have the following trivial bound.
\begin{equation*}
\Bigl|\int_S^T (f(B_r^H+\phi_r)-f(B_r^H+\psi_r))\,dr\Bigr|\le 2\|f\|_{\C^0}|T-S|.
\end{equation*}
Therefore,
\begin{equs}
\Bigl\|&\int_S^T (f(B_r^H+\phi_r)-f(B_r^H+\psi_r))\,dr\Bigr\|_{L_p(\Omega)}
\\
&\leq
\Bigl\|\bone_{\xi\leq L}\int_S^T (f(B_r^H+\phi_r)-f(B_r^H+\psi_r))\,dr\Bigr\|_{L_p(\Omega)}
\\
&\qquad
+\Bigl\|\bone_{\xi\geq L}\int_S^T (f(B_r^H+\phi_r)-f(B_r^H+\psi_r))\,dr\Bigr\|_{L_p(\Omega)}
\\
&\leq LN \|f\|_{\C^\alpha}(T-S)^{H(\alpha-1)+1-\eps}\bigl(\|\psi_S-\phi_S\|_{L_p(\Omega)}+
\|[\psi-\phi]_{\C^\tau([S,T])}\|_{L_p(\Omega)}(T-S)^{\tau}\bigr)
\\
&\qquad+2\big(\P(\xi\geq L)\big)^{1/p}\|f\|_{\C^0}|T-S|.
\end{equs}
By Chebyshev inequality and \eqref{finsecondmoment}, we finally get
\eqref{ourboundasreg} for the case of smooth $f$.

Now we are ready to remove the extra assumption on the smoothness of $f$. Let us set $f_n= \cP_{1/n}f \in \C^\infty$. By applying the statement of the lemma to $f_n$ and using that $\|f_n\|_{\C^
\beta} \leq \| f\|_{\C^\beta}$ for $\beta=\alpha, 0$  we get 

\begin{align}\label{alphazerozero}
& \Bigl\|\int_s^t ( f_n(B_r^H+\phi_r)- f_n(B_r^H+\psi_r))\,dr\Bigr\|_{L_p(\Omega)}\nn
\\
&\quad\le
NL \|f\|_{\C^\alpha}
(t-s)^{H(\alpha-1)+1-\eps}(\|\phi_s-\psi_s\|_{L_p(\Omega)}+
\|[\phi-\psi]_{C^{\tau}([s,t])}\|_{L_p(\Omega)}(t-s)^\tau)\nn\\
&\qquad+ N \|f\|_{\C^0}|t-s|\exp(-L^{2-\eps_1}).
\end{align}

If $\alpha>0$, then $f_n(x) \to f(x)$ for all $x \in \R^d$ and the claim follows by Fatou's lemma. So we only have to consider the case  $\alpha=0$. Clearly, it suffices to show that for each $r>0$, almost surely 
\begin{equs}
( f_n(B_r^H+\phi_r)- f_n(B_r^H+\psi_r)) \to ( f(B_r^H+\phi_r)- f(B_r^H+\psi_r)),
\end{equs}
as $n \to \infty$. Notice that almost surely $f_n(B^H_r) \to f(B^H_r)$ as $n \to \infty$, since the law of $B^H_r$ is absolutely continuous (for $r>0$). Moreover, since $\alpha=0$, we have by assumption that $H< 1/2$. By Proposition \ref{Pr:Girsanov} (recall that $\phi$ satisfies $\eqref{psicond0}$, therefore is Lipschitz)  there exists a neasure equivalent to $\bP$ under which  $B^H+ \phi$ is a fractional brownian motion. Consequently, for all $r >0$, 
almost surely 
 $$
 f_n(B_r^H+\phi_r) \to f(B_r^H+\phi_r),
 $$
as $n \to \infty$. With the same reasoning we obtain that almost surely $f_n(B_r^H+\psi_r) \to f(B_r^H+\psi_r)$. The lemma is now proved. 
\end{proof}

\subsection{Proof of Theorem \ref{thm:main fractional}}\label{S:proof21}
\begin{proof}
Without loss of generality we assume $\alpha\neq1$. Let us denote
\begin{equ}
\psi_t:=x_0+\int_0^t b(X_r)\,dr,\quad\psi^n_t:=x^n_0+\int_0^t b(X^n_{\kappa_n(r)})\,dr,\quad t\in[0,1].
\end{equ}
Fix $\eps>0$ such that
\begin{equation}\label{defeps}
\eps<\frac12+H(\alpha-1).
\end{equation}
By assumption \eqref{eq:exponent main} such $\eps$ exists. Fix now large enough $p\ge2$
such that 
\begin{equation}\label{defp}
d/p<\eps/2.
\end{equation}

Fix $0\leq S\leq T\leq 1$.
Then, taking into account \eqref{DKBound Xn fBM}, for any $S\leq s\leq t\leq T$ we have
\begin{equs}\label{Step1 frac}
\|(\psi_t- \psi_s)-(\psi^n_t- \psi^n_s)\|_{L_p(\Omega)}&=\Bigl\|\int_s^t (b(X_r)-b(X^n_{\kappa_n(r)}))\,dr\Bigr\|_{L_p(\Omega)}\\
&\leq \Bigl\|\int_s^t (b(X_r)-b(X^n_r))\,dr\Bigr\|_{L_p(\Omega)}+N|t-s|^{1/2+\eps} n^{-\gamma+\eps}.
\end{equs}
Let $M\ge1$ be a parameter to be fixed later. We wish to apply Lemma~\ref{L:finfin} with $\psi^n$ in place of $\phi$, $\frac12+H(\alpha-1)-\eps$ in place  of $\eps$, and $\tau:=1/2+\eps/2$. Let us check that all the conditions of this lemma are satisfied. First, we note that by \eqref{defeps} we have $\frac12+H(\alpha-1)-\eps>0$, which is required by the assumptions of the lemma. Second, we note that $1/2+\eps/2>H(1-\alpha)$ thanks to  \eqref{eq:exponent main}, thus this choice of $\tau$ is allowed. Next,
it is clear that $\psi_0$ and $\psi^n_0$ are deterministic. Further, since the function $b$ is bounded, we see $\psi$ and $\psi^n$ satisfy \eqref{psicond0}. Finally, let us verify that $\psi$ satisfies \eqref{psicond}. If $H<1/2$, this condition holds automatically thanks to the boundedness of $b$. If $H\ge 1/2$ then pick $H'\in(0,H)$ such that
\begin{equation}\label{Hprimedef}
\alpha H'>H-\frac12.
\end{equation}
Note that such $H'$ exists thanks to  assumption \eqref{eq:exponent main}. Then, by definition of $\psi$, we clearly have
\begin{equation*}
[\psi]_{\C^{1+\alpha H'}}\le |x_0|+\|b\|_{\C^0}+[b(X_{\cdot})]_{\C^{\alpha H'}}
\le |x_0|+\|b\|_{\C^0}+\|b\|_{\C^0}^{\alpha}+[B^H]_{{\C^{ H'}}}^\alpha.
\end{equation*}
Therefore for any $\lambda>0$ we have
\begin{equation*}
\E e^{\lambda[\psi]_{\C^{1+\alpha H'}}^2}\le N(|x_0|,\|b\|_{\C^0})\E \exp([B^H]_{{\C^{ H'}}}^{2\alpha})<\infty.
\end{equation*}
By taking $\rho:= 1+\alpha H'$ and recalling $\eqref{Hprimedef}$, we see that $\rho>H+1/2$ and thus condition  \eqref{psicond} holds. 
Therefore  all conditions of Lemma~\ref{L:finfin} are met. Applying this lemma, we get
\begin{align*}
&\Bigl\|\int_s^t (b(X_r)-  b(X^n_r))\,dr\Bigr\|_{L_p(\Omega)}\\
&\quad =\Bigl\|\int_s^t (b(B^H_r+\psi_r)-b(B^H_r+\psi^n_r))\,dr\Bigr\|_{L_p(\Omega)}
\\
&\quad\leq M N|t-s|^{\frac12+\eps}\|\psi_S-\psi_S^n\|_{L_p(\Omega)}
\\
&\qquad+ MN|t-s|^{1+3\eps/2}
\|[\psi-\psi^n]_{\C^{1/2+\eps/2}([s,t])}\|_{L_p(\Omega)}+ N \exp(-M^{2-\eps_0})|t-s|\\
&\quad\leq M N|t-s|^{\frac12+\eps}\|\psi_S-\psi_S^n\|_{L_p(\Omega)}
\\
&\qquad+ MN|t-s|^{1+3\eps/2}
\db{\psi-\psi^n}_{\scC^{1/2+\eps}_p,[s,t]}+ N \exp(-M^{2-\eps_0})|t-s|,
\end{align*}
where the last inequality follows from the Kolmogorov continuity theorem and \eqref{defp}. Using this in \eqref{Step1 frac}, dividing by $|t-s|^{1/2+\eps}$ and taking supremum over $S\leq s\leq t\leq T$, we get for some $N_1\ge1$ independent of $M$, $n$
\begin{align}\label{eq:notfinal tau new}
\db{\psi-\psi^n}_{\scC^{1/2+\eps}_p,[S,T]}
\leq& MN_1 \|\psi_S-\psi^n_S\|_{L_p(\Omega)}
+MN_1|T-S|^{1/2+\eps/2}\db{\psi-\psi^n}_{\scC^{1/2+\eps}_p,[S,T]}\nn\\
&+ N_1 n^{-\gamma+\eps}+N_1 \exp(-M^{2-\eps_0}).
\end{align}
Fix now $m$ to be the smallest integer so that $N_1M m^{-1/2-\eps/2}\leq 1/2$ (we stress that $m$ does not depend on $n$). One gets from \eqref{eq:notfinal tau new}
\begin{equ}\label{eq:final tau}
\db{\psi-\psi^n}_{\scC^{1/2+\eps}_p,[S,S+1/m]}
\leq 2M N_1 \|\psi_S-\psi^n_S\|_{L_p(\Omega)} + 2N_1 n^{-\gamma+\eps}+2N_1 \exp(-M^{2-\eps_0}),
\end{equ}
and thus 
\begin{equ}
\|\psi_{S+1/m}-\psi^n_{S+1/m}\|_{L_p(\Omega)}
\leq 2MN_1 \|\psi_S-\psi^n_S\|_{L_p(\Omega)} + 2N_1 n^{-\gamma+\eps}+2N_1 \exp(-M^{2-\eps_0}).
\end{equ}
Starting from $S=0$ and applying the above bound $k$ times, $k=1,\ldots,m$, one can conclude
\begin{align*}
\|\psi_{k/m}-\psi^n_{k/m}\|_{L_p(\Omega)}
&\leq (2MN_1)^k \Bigl(\|\psi_0-\psi^n_0\|_{L_p(\Omega)} + 2N_1 n^{-\gamma+\eps}+
+2N_1 \exp(-M^{2-\eps_0})\Bigr)\\
&\le (2MN_1)^m \Bigl(|x_0-x^n_0| + 2N_1 n^{-\gamma+\eps}
+2N_1 \exp(-M^{2-\eps_0})\Bigr).
\end{align*}
Substituting back into \eqref{eq:final tau}, we get
\begin{align}\label{toootechnical}
\db{\psi-\psi^{n}}_{\scC^{1/2+\eps}_p([0,1])}&\leq m \sup_{k=1,\dots,m}\db{\psi-\psi^{n}}_{\scC^{1/2+\eps}_p([k/m,(k+1)/m])}\nn\\
&\le (2N_1M)^{m+5}\Bigl(|x_0-x_0^n|+N_1 n^{-\gamma+\eps}+N_1 \exp(-M^{2-\eps_0})\Bigr).
\end{align}

It follows from the definition of $m$ that $m\leq 2N_1^2M^{2-\eps}$. At this point we choose $\eps_0=\eps/2$ and note that for some universal constant $N_2$ one has
\begin{equation*}
(2N_1M)^{m+5}=e^{(m+5)\log (2 N_1M)}\le e^{(2N_1^2M^{2-\eps}+5)\log (2 N_1M)}\le N_2 e^{\frac12M^{2-\eps/2}}.
\end{equation*}
Thus, we can continue \eqref{toootechnical} as follows.
\begin{equation}\label{eq:final tau_final}
\db{\psi-\psi^{n}}_{\scC^{1/2+\eps}_p([0,1])}
\le e^{N_3M^{2-\eps}\log M}\Bigl(|x_0-x_0^n|+N_1 n^{-\gamma+\eps}\Bigr)+N_1N_2 \exp(-\frac12M^{2-\eps/2}).
\end{equation}
Fix now $\delta>0$ and choose $N_4=N_4(\delta)$ such that for all $M>0$ one has
$$
\exp(\frac12M^{2-\eps/2})\ge N_4 e^{\delta^{-1}N_3M^{2-\eps}\log M}.
$$
It remains to notice that by choosing $M>1$ such that
\begin{equ}
e^{N_3M^{2-\eps}\log M}= n^{\delta},
\end{equ}
one has 
\begin{equ}
e^{-\frac12M^{2-\eps/2}}\leq N n^{-1}.
\end{equ}
Substituting back to \eqref{eq:final tau_final} and since $X-X^n=\psi-\psi^n$, we arrive to the required bound \eqref{eq:frac main bound}.
\end{proof}

\section{Malliavin calculus for the Euler-Maruyama scheme}\label{sec:Malliavin}
In the multiplicative standard Brownian case, we first consider Euler-Maruyama schemes without drift: for any $y\in\R^d$ define the process $\bar X^n(y)$ by
\begin{equ}\label{eq:EM no drift}
d\bar X^n_t(y)=\sigma(\bar X^n_{\kappa_n(t)}(y))\,dB_t,\quad \bar X^n_0=y.
\end{equ}
This process will play a similar role as $B^H$ in the previous section.
Similarly to the proof of Lemma \ref{lem:(ii)}, we need sharp bounds on the conditional distribution of $\bar X^n_t$ given $\F_s$,
which can be obtained from bounds of the density of $\bar X^n_t$.
A trivial induction argument yields that for $t>0$, $\bar X^n_t$ indeed admits a density, but to our knowledge such inductive argument can not be used to obtain useful quantitative information.

\begin{remark}\label{rem:density}
While the densities of Euler-Maruyama approximations have
been studied in the literature, see e.g. \cite{GyK, Bally-Talay, Bally-Talay2}, none of the available estimates suited well for our purposes. 
In \cite{GyK}, under less regularity assumption on $\sigma$, $L_p$ bounds of the density (but not its derivatives) are derived. In \cite{Bally-Talay, Bally-Talay2}, smoothness of the density is obtained even in a hypoelliptic setting, but without sharp control on the short time behaviour of the norms.
\end{remark}

\begin{theorem}\label{thm:density}
Let $\sigma$ satisfy \eqref{eq:elliptic}, $\bar X^n$ be the solution of \eqref{eq:EM no drift}, and let $G\in\C^1$. Then for all $t=1/n,2/n,\ldots,1$ and $k=1,\ldots,d$ one has the bound
\begin{equ}\label{eq:main Malliavin}
|\E\d_k G(\bar X^n_t)|\leq N \|G\|_{\C^0}t^{-1/2} + N\|G\|_{\C^1}e^{-cn}
\end{equ}
with some constant $N=N(d,\lambda,\|\sigma\|_{\C^2})$ and $c=c(d,\|\sigma\|_{\C^2})>0$.
\end{theorem}

We will prove Theorem \ref{thm:density} via Malliavin calculus.
In our discrete situation, of course this could be translated to finite dimensional standard calculus,
but we find it more instructive to follow the basic terminology of \cite{Nualart-Malliavin}, which we base on the lecture notes \cite{Hairer_notes}.

\subsection{Definitions}
Define $H=\{h=(h_i)_{i=1,\ldots,n}:\,h_i\in\R^d\}$, with the norm
\begin{equ}
\|h\|^2_H=\frac{1}{n}\sum_{i=1}^n|h_i|^2=\frac{1}{n}\sum_{i=1}^n\sum_{k=1}^d|h_i^k|^2.
\end{equ}
One can obtain a scalar product from $\|\cdot\|_H$, which we denote by $\<\cdot,\cdot\>_H$.
Let us also denote $\cI=\{1,\ldots,n\}\times\{1,\ldots,d\}$.
One can of course view $H$ as a copy of $\R^\cI$, with a rescaled version of the usual $\ell_2$ norm.
We denote by $e_{(i,k)}$ the element of $H$ whose elements are zero apart from the $i$-th one, which is the $k$-th unit vector of $\R^d$.
Set $\Delta W_{(i,k)}:=W^{k}_{i/n}-W^k_{(i-1)/n}$.
Then for any $\R$-valued random variable $X$ of the form
\begin{equ}
X=F(\Delta W_{(i,k)}:\,(i,k)\in\cI),
\end{equ}
where $F$ is a differentiable function, with at most polynomially growing derivative, the Malliavin derivative of $X$ is defined as the $H$-valued random variable
\begin{equ}
\cD X :=
\sum_{(i,k)\in\cI}(\cD^k_i X)e_{(i,k)}
:=\sum_{(i,k)\in\cI}\partial_{(i,k)}F( \Delta W_{(j,\ell)}:\,(j,\ell)\in\cI)e_{(i,k)}.
\end{equ}
For multidimensional random variables we define $\cD$ coordinatewise.
In the sequel we also use the matrix norm on $\R^{d\times d}$ defined in the usual way $\|M\|:=\sup_{x\in\R^d, |x|=1}|Mx|$. Recall that if $M$ is positive semidefinite, then one has $\|M\|=\sup_{x\in\R^d, |x|=1}x^*Mx$.
It follows that $\|\cdot\|$ is monotone increasing with respect to the usual order $\preceq$ on the positive semidefinite matrices.

The following few properties are true in far larger generality, for the proofs we refer to \cite{Hairer_notes}.
One easily sees that the derivative $\cD$ satisfies the chain rule: namely, for any differentiable $G:\R^d\to\R$, one has $\cD G(X)=\nabla G(X)\cdot \cD X$.
The operator $\cD$ is closable, and its closure will also be denoted by $\cD$, whose domain we denote by $\cW\subset L_2(\Omega)$. The adjoint of $\cD$ is denoted by $\delta$. One then has that  the domain of $\delta$  is included in $\cW(H)$ and the following identity holds:
\begin{equ}\label{eq:delta}
\E|\delta u|^2=\E\|u\|^2_H+\E\frac{1}{n^2}\sum_{(i,k),(j,m)\in\cI}(\cD^k_i u^m_j)(\cD^m_j u^k_i).
\end{equ}

\subsection{Stochastic difference equations}
First let us remark that the equation \eqref{eq:EM no drift} does not define an invertible stochastic flow: indeed, for any $t>0$, $y\to\bar X^n_t(y)$ may not even be one-to-one. Therefore in order to invoke arguments from the Malliavin calculus for diffusion processes, we consider a modified process equation that does define an invertible flow. Unfortunately, this new process will not have a density, but its singular part (as well as its difference from the original process) is exponentially small.

Take a smooth function $\varrho:\R\to\R$ such that $|\varrho(r)|\leq |r|$
for all $r\in\R$,
$\varrho(r)=r$ for $|r|\leq (4\|\sigma\|_{\C^1} d^2)^{-1}$, $\varrho(r)=0$ for $|r|\geq (2\|\sigma\|_{\C^1} d^2)^{-1}$, and that satisfies $|\d^k\varrho|\leq N$ for $k=0,\ldots,3$ with some $N=N(d,\|\sigma\|_{\C^1})$.
Define the recursion, for $x\in\R^d$ and $j=1,\ldots, n$, $k=1,\ldots,d$
\begin{equ}\label{eq:recursion}
\cX_{j}^{k}(x)=\cX_{j-1}^k(x)+\sum_{\ell=1}^d\sigma^{k\ell}\big(\cX_{j-1}(x)\big)\varrho(\Delta W_{(j,\ell)}),\qquad \cX_{0}(x)=x.
\end{equ}
By our definition of $\varrho$, for any $j$, \eqref{eq:recursion} defines a diffeomorphism from $\R^d$ to $\R^d$ by $x\to \cX_{j}(x)$.
It is easy to see that its Jacobian
$J_{j}(x)=\big(J_{j}^{m,k}(x)\big)=\big(\d_{x^m}\cX^k_{j}(x)\big)_{k,m=1,\ldots,d; \,j=1,\ldots,n}$
satisfies the recursion
\begin{equ}
J_{j}^{m,k}(x)=J_{j-1}^{m,k}(x)+
\sum_{q=1}^d J_{j-1}^{m,q}(x)
\Big[\sum_{\ell=1}^d\d_{q}\sigma^{k\ell}\big(\cX_{j-1}(x)\big)\varrho(\Delta W_{(j,\ell)})\Big],\qquad J_{0}(x)=\id.
\end{equ}
It is also clear that $\cD_i^m\cX^k_j=0$ for $j<i$, while for $j>i$ we have the recursion
\begin{equs}
\cD_i^m\cX^k_j(x)=\cD_i^m\cX^k_{j-1}(x)
+
\sum_{q=1}^d\cD_i^m \cX^q_{j-1}(x) & \Big[
\sum_{\ell=1}^d \d_q\sigma^{k\ell}\big(\cX_{j-1}(x)\big)
	\varrho(\Delta W_{(j,\ell)})\Big],
\\
&\cD^m_i\cX^k_i=\sigma^{km}\big(\cX_{i-1}(x)\big)\varrho'(\Delta W_{(i,m)}).
\end{equs}
From now on we will usually suppress the dependence on $x$ in the notation.
Save for the initial conditions, the two recursions coincide for the matrix-valued processes $J_\cdot$ and $\cD_i \cX_\cdot$. Since the recursion is furthermore linear, $j\mapsto J_j^{-1}\cD_i \cX_j$ is constant in time for $j\geq i\geq 1$. In particular,
\begin{equ}
J_{j}^{-1}\cD_i \cX_j=J_i^{-1}\big(\sigma^{km}(\cX_{i-1})\varrho'(\Delta W_{(i,m)})\big)_{k,m=1,\ldots,d}\,,
\end{equ}
or, with the notation $J_{i,j}=J_jJ_i^{-1}$,
\begin{equ}
\cD_i \cX_j=J_{i,j}\big(\sigma^{km}(\cX_{i-1})\varrho'(\Delta W_{(i,m)})\big)_{k,m=1,\ldots,d}\,.
\end{equ}
Let us now define the event $\hat \Omega\subset\Omega$ by
\begin{equ}
\hat \Omega=\{|\Delta W_{(i,k)}|\leq (4\|\sigma\|_{\C^1} d^2)^{-1}, \forall (i,k)\in\cI\}
\end{equ}
as well as the (matrix-valued) random variables $\mathcal{D}_{i,j}$ by
\begin{equ}\label{eq:malliavin id}
\mathcal{D}_{i,j}=J_{i,j}\sigma(\cX_{i-1}).
\end{equ}
Clearly, on $\hat\Omega$ one has $\mathcal{D}_{i,j}=\cD_i \cX_j$.
Note that for fixed $j,m$ one may view $\ccD_{\cdot,j}^{\cdot,m}$ as an element of $H$, while for fixed $i,j$ one may view $\ccD_{i,j}$ as a $d\times d$ matrix.
One furthermore has the following exponential bound on $\hat \Omega$.
\begin{proposition}\label{prop:small Omega}
There exist $N$ and $c>0$ depending only on $d$ and $\|\sigma\|_{\C^1}$, one has $\bP(\hat\Omega)\geq 1-Ne^{-cn}$.
\end{proposition}
\begin{proof}
For each $(i,k)\in\cI$, since $\Delta W_{(i,k)}$ is zero mean Gaussian with variance $n^{-1}$, one has
\begin{equ}
\bP\big(\varrho(\Delta W_{(i,k)})\neq\Delta W_{(i,k)}\big)\leq
\bP\big(|\Delta W_{(i,k)}|\geq(4\|\sigma\|_{\C^1} d^2)^{-1}\big)\leq N'e^{-c'n}
\end{equ}
with some $N'$ and $c'>0$ depending only on $d$ and $\|\sigma\|_{\C^1}$,
by the standard properties of the Gaussian distribution. Therefore, by the elementary inequality $(1-x)^\alpha\geq 1-\alpha x$, valid for all $x\in[0,1]$ and $\alpha\geq 1$, one has
\begin{equs}
\bP(\hat\Omega)\geq \big(1-(N'e^{-c'n}\wedge 1)\big)^{nd}\geq 1-N'nde^{-c'n}\geq 1-Ne^{-(c'/2)n}.
\end{equs}
\end{proof}

We now fix $(j,k)\in\cI$, $G\in \C^1$, and we aim to bound $|\E\d_k G(X_j)|$ in terms of $t:=j/n$ and $\|G\|_0$, and some additional exponentially small error term.
To this end, we define the Malliavin  matrix $\cM\in\R^{d\times d}$
\begin{equ}
\cM^{m,q}=\<\ccD_{\cdot,j}^{\cdot,m},\ccD_{\cdot,j}^{\cdot,q}\>_H=\frac{1}{n}\sum_{(i,v)\in\cI}\ccD_{i,j}^{v,m}\ccD_{i,j}^{v,q},
\end{equ}
with $m,q=1,\ldots,d$.
As we will momentarily see (see \eqref{eq:M inverse}), $\cM$ is invertible. Define
$$
Y=\sum_{m=1}^d(\ccD_{\cdot,j}^{\cdot,m})(\cM^{-1})^{m,k}\in H.
$$
One then has by the chain rule that on $\hat \Omega$, $\d_k G(\cX_j)= \<\cD G(X_j),Y\>_H$.
Therefore,
\begin{equs}
\E \d_k G(\cX_j)&=\E \<\cD G(X_j),Y\>_H+\E\d_k G(\cX_j)\bone_{\hat\Omega^c}-\E \<\cD G(\cX_j),Y\>_H\bone_{\hat\Omega^c}
\\
&=\E(G( X_j),\delta Y)+\E\d_k G(\cX_j)\bone_{\hat\Omega^c}-\E \<\cD G(\cX_j),Y\>_H\bone_{\hat\Omega^c}
\\
&=:\E(G( \cX_j),\delta Y)+I_1+I_2.\label{eq:Mall IBP}
\end{equs}
Recalling \eqref{eq:delta}, one has
\begin{equ}\label{eq:delta Y}
\E|\delta Y|^2\leq\E\|Y\|^2_H+\E\frac{1}{n^2}\sum_{(i,q),(r,m)\in\cI}(\cD^q_i Y^m_r)(\cD^m_rY^q_i).
\end{equ}
Theorem \ref{thm:density} will then follow easily once we have the appropriate moment bounds of the objects above. Recall the notation $t=j/n$.
\begin{lemma}\label{lem:bounds Mall}
Assume the above notations and let $\sigma$ satisfy \eqref{eq:elliptic}.
Then for any $p>0$, one has the bounds
\begin{equ}\label{eq:J bound}
\E\sup_{i=1,\ldots,j}\|J_{i,j}(x)\|^p+\E \sup_{1\leq i\leq j}\|J_{i,j}^{-1}(x)\|^p\leq N,
\end{equ}
\begin{equ}\label{eq:ccD bound}
\E\sup_{i=1,\ldots,j}\|\ccD_{i,j}(x)\|^p\leq N,
\end{equ}
\begin{equ}\label{eq:M inverse bound}
\E\|\cM^{-1}(x)\|^p\leq Nt^{-p},
\end{equ}
\begin{equ}\label{eq:cD Y bound}
\sup_{i =1,\ldots, j}\E\sup_{r=1,\ldots,j}\|\cD_i Y_r(x)\|^p\leq N t^{-p}.
\end{equ}
for all $x \in \R^d$, with some $N=N(p,d,\lambda,\|\sigma\|_{\C^2})$.
\end{lemma}
\begin{proof}
As before, we omit the dependence on $x\in \R^d$ in order to ease the notation.
We first bound the moments of $\sup_j\|J_j\|$. Recall that we have the recursion
\begin{equs}            \label{eq:recursion-J}
J_j= J_{j-1}(I+ \Gamma_{j/n}),
\end{equs}
where the matrix $\Gamma_t=(\Gamma_t)_{q,k=1}^d$ is given by
\begin{equs}             \label{eq:def-gamma}
\Gamma^{q,k}_t = \sum_{\ell=1}^d \D_q \sigma^{k \ell} ( \mathcal{X}_{n \kappa_n(t)}) \varrho ( W^\ell_t -W^\ell_{\kappa_n(t)}),
\end{equs}
By It\^o's formula it follows that
\begin{equs}
\varrho ( W^\ell_t -W^\ell_{\kappa_n(t)})= \int_{{\kappa_n(t)}}^t \varrho'(W^\ell_s-W^\ell_{\kappa_n(t)}) \, dW^\ell_s + \frac{1}{2}\int_{{\kappa_n(t)}}^t\varrho''(W^\ell_s-W^\ell_{\kappa_n(t)})  \, ds.
\end{equs}
Consequently,  for $j=0, \ldots, n$ we have that $J_j= Z_{j/n}$,
where the matrix-valued process $Z_t$ satisfies
\begin{equs}                       \label{eq:representation-SDE-J}
dZ_t =  \sum_{q=1}^d Z_{\kappa_n(t)}\mathcal{A}_t \, dt  +   \sum_{\ell =1}^d   Z_{\kappa_n(t)}  \mathcal{B}^{\ell}_t dW^\ell_t,  \qquad Z_0= I,
\end{equs}
with matrices $\mathcal{A}_s=(\mathcal{A}^{q, k}_s)_{q,k=1,\ldots,d}$ and $\mathcal{B}^\ell_s=(\mathcal{B}^{\ell,q,k}_s)_{q,k=1,\ldots,d}$  given by
\begin{equs}
 \mathcal{A}^{q,k}_s &= \frac{1}{2} \sum_{\ell=1}^d \D_q \sigma^{ k \ell} (\mathcal{X}_{n \kappa_n(s)})\varrho''(W^\ell_s-W^\ell_{\kappa_n(s)})
\\
\mathcal{B}^{\ell,q,k}_s& =  \D_q \sigma^{k \ell} (\mathcal{X}_{n \kappa_n})\varrho'(W^\ell_s-W^\ell_{\kappa_n(s)}).
\end{equs}
Notice that  there exists a constant $N=N (\| \sigma\|_{\mathcal{C}^1},  \|\varrho\|_{\mathcal{C}^2})$  such that almost surely, for all $(t, x) \in [0,1] \times  \R^d$
\begin{equs}                     \label{eq:bound-A-B}
\|\mathcal{A}_t\|+ \sum_{\ell=1}^d\|\mathcal{B}^{\ell}_t\| \leq N.
\end{equs}
This bound combined with the fact that $Z_t$ satisfies \eqref{eq:representation-SDE-J} imply
the bounds
\begin{equs}
\E \sup_{t \leq 1} \|Z_t\|^p \leq N
\end{equs}
for all $p>0$. Hence,
\begin{equs}                 \label{eq:bound-Jj}
\E \sup_{j=1,..,n}\|J_j\|^p \leq \E \sup_{t \leq 1} \|Z_t\|^p \leq  N.
\end{equs}
We now bound the moments of $\sup_j \|J^{-1}_j\|$. By \eqref{eq:recursion-J} we get
\begin{equs}     \label{eq:recursion-J-inverse}
J_j^{-1}=(I+ \Gamma_{j/n})^{-1} J_{j-1}^{-1}
\end{equs}
Recall that  for $t \in [ (j-1)/n, j/n]$
\begin{equs}
\Gamma_t= \int_{(j-1)/n}^t \mathcal{A}_s \, ds + \sum_{\ell=1}^d   \int_{(j-1)/n}^t \mathcal{B}^\ell_s \, dW^\ell_s,
\end{equs}
and that  by the definition of $\varrho$ and \eqref{eq:def-gamma},  for all $t \in [0,T]$,  the matrix  $I+\Gamma_t$
is invertible. Hence,  by It\^o's formula,  we have for $t  \in [ (j-1)/n, j/n]$
\begin{equs}         \label{eq:representation-I+Gamma-inverse}
(I+\Gamma_t)^{-1}= I +\int_{(j-1)/n}^t \tilde{\mathcal{A}}_s \, ds + \sum_{\ell=1}^d \int_{(j-1)/n}^t  \tilde{\mathcal{B}}^\ell_s \, d W^\ell_s,
\end{equs}
with
\begin{equs}
 \tilde{\mathcal{A}}_s & = \sum_{\ell=1}^d
(I+\Gamma_s)^{-1} \mathcal{B}_s^\ell(I+\Gamma_s)^{-1}\mathcal{B}_s^\ell(I+\Gamma_s)^{-1}
 -(I+\Gamma_s)^{-1} \mathcal{A}_s(I+\Gamma_s)^{-1},
\\
\tilde{\mathcal{B}}_s^\ell & = -(I+\Gamma_s)^{-1} \mathcal{B}^\ell_s(I+\Gamma_s)^{-1}.
\end{equs}
Moreover,  by definition or $\varrho$,
almost surely, for all $(t,x) \in [0,T] \times \R^d$ one has
\begin{equs}                     \label{eq:bound-A,B,tilde}
 \|\tilde{\mathcal{A}}_t\|+\sum_{\ell=1}^d \| \tilde{\mathcal{B}}^\ell_t \| \leq N.
\end{equs}
By \eqref{eq:recursion-J-inverse} and \eqref{eq:representation-I+Gamma-inverse}, for $j=1,...,n$
 we have that
$ J^{-1}_j= \tilde{Z}_{j/n}$,
  where the matrix valued process $\tilde{Z}_t$ is defined by
 \begin{equs}
 d\tilde{Z}_t=  \tilde{\mathcal{A}}_t \tilde{Z}_{\kappa_n(t)} \, dt + \sum_{\ell=1}^d  \tilde{\mathcal{B}}^\ell_t \tilde{Z}_{\kappa_n(t)} \, dW^\ell_s, \qquad \tilde{Z}_0  =I .
 \end{equs}
By this and the bounds  \eqref{eq:bound-A,B,tilde} we have
the bounds
\begin{equs}
\E \sup_{t \leq 1} \|\tilde{Z}_t\|^p \leq N
\end{equs}
for all $p>0$. Consequently,
\begin{equs}              \label{eq:bound-Jj-inverse}
\E \sup_{j=1,...,n} \|J^{-1}_j\|^p \leq \E \sup_{t \leq 1} \|\tilde{Z}_t\|^p \leq N.
\end{equs}
Finally, from \eqref{eq:bound-Jj} and \eqref{eq:bound-Jj-inverse} we obtain \eqref{eq:J bound}.

The bound \eqref{eq:ccD bound} then immediately follows from \eqref{eq:J bound}, the definition \eqref{eq:malliavin id}, and the boundedness of $\sigma$.

Next, we show \eqref{eq:M inverse bound}.
On the set of positive definite matrices we have that on one hand,
matrix inversion is a convex mapping, and on the other hand, the function $\|\cdot\|^p$ is a convex increasing mapping for $p\geq 1$.
It is also an elementary fact that if $B\succeq \lambda I$, then $\|(ABA^*)^{-1}\|\leq \lambda^{-1}\|(AA^*)^{-1}\|$.
One then writes
\begin{equs}
\|\cM^{-1}\|^p & = \Big(\frac{n}{j}\Big)^p\Big\|\Big(\frac{1}{j}\sum_{i=1}^j\big[J_{i,j}\sigma(\cX_{i-1})\big]\big[J_{i,j}\sigma(\cX_{i-1})\big]^*\Big)^{-1}\Big\|^p
\\
&\leq t^{-p}\frac{1}{j}\sum_{i=1}^j\|\big(\big[J_{i,j}\sigma(\cX_{i-1})\big]\big[J_{i,j}\sigma(\cX_{i-1})\big]^*\big)^{-1}\|^p
\\
&\leq\lambda^{-p}t^{-p}\frac{1}{j}\sum_{i=1}^j\|J_{i,j}^{-1}\|^{2p}
\\
&\leq
\lambda^{-p}t^{-p}\sup_{i=1,\ldots,j}\|J_{i,j}^{-1}\|^{2p}.
\label{eq:M inverse}
\end{equs}
Therefore \eqref{eq:M inverse bound} follows from \eqref{eq:J bound}

We now move to the proof of \eqref{eq:cD Y bound}. First of all, notice that
the above argument yields
\begin{equs}             \label{eq:bound-DX}
\sup_{i = 1,...,n} \E \sup_{j=1,...,n} \| \mathscr{D}_i  \mathcal{X}_j\|^p \leq N.
\end{equs}
for all $p>0$.
Indeed, the proof of this is identical to the proof of \eqref{eq:bound-Jj} since $(\mathscr{D}_i \mathcal{X}_j)_{j \geq i}$ has the same dynamics as $(J_j)_{j\geq0} $ and initial condition $\mathscr{D}^k_i \mathcal{X}^m_i=\sigma^{km} ( \mathcal{X}_{i-1}) \varrho'(\Delta W_{(i,m)})$ which is bounded. Recall that
\begin{equs}
Y_r = \sum_{m=1}^d ( \mathcal{D}^{\cdot, m}_{r,j}) (\mathscr{M}^{-1})^{m,k}.
\end{equs}
By Leibniz's rule, for each $i, r \in \{0,..,n\}$, $\mathscr{D}_iY^r$  is a $\R^d \otimes \R^d$-valued random variable given by
\begin{equs}          \label{eq:decomp-DY}
\mathscr{D}_iY_r= \sum_{m=1}^d ( \mathscr{D}_i \mathcal{D}^{\cdot, m}_{r,j})  (\mathscr{M}^{-1})^{m,k}+ \sum_{m=1}^d \mathcal{D}^{\cdot, m}_{r,j} \otimes  \mathscr{D}_i (\mathscr{M}^{-1})^{m,k}
\end{equs}
We start with a bound for $\sup_r \|\mathscr{D}_i \mathcal{D}_{r,j}\|$. By definition of $\mathcal{D}_{i,j}$ we have that
\begin{equs}            \label{eq:decomp-DD}
\mathscr{D}_i\mathcal{D}_{r,j} = (\mathscr{D}_iJ_j ) J^{-1}_r \sigma(\mathcal{X}_{r-1})+ J_j  (\mathscr{D}_iJ^{-1}_r)  \sigma(\mathcal{X}_{r-1})+ J_j  J^{-1}_r (\mathscr{D}_i \sigma(\mathcal{X}_{r-1})),
\end{equs}
where for $A \in (\R^d)^{\otimes 2}$, $B \in (\R^d)^{\otimes 3}$, the product $AB$ or $BA$ is an element of  $(\R^d)^{\otimes 3}$ that arises by considering $B$ as a $d\times d$ matrix whose entries are elements of $\R^d$. We estimate the term $\mathscr{D} _i J_j$.  As before, we have that $\mathscr{D}_i J_j = \mathscr{D}_i Z_{j/n}$,  where $Z$ is given by \eqref{eq:representation-SDE-J}. We have that  $\mathscr{D}_i Z_t=0$ for $t <i/n$ while for $t \geq  i/n$ the process $\mathscr{D}_i Z_t=:\mathscr{Z}^i_t$ satisfies
\begin{equs}                  \label{eq:SDE-DZ}
\mathscr{Z}^i_t & =  \left( \mathscr{Z}^i_{\kappa_n(t)} \mathcal{A}_t + Z_{\kappa_n(t)} \mathscr{D}_iA_t \right) \, dt+ \sum_{\ell=1}^d \left(  \mathscr{Z}^i_{\kappa_n(t)} \mathcal{B}^\ell_t + Z_{\kappa_n(t)} \mathscr{D}_i \mathcal{B}^\ell_t \right) dW^\ell_t
\\
\mathscr{Z}^i_{i/n} & = Z_{i/n}\sum_{\ell=1}^d \mathcal{B}^\ell _{i/n}
\end{equs}
By the chain rule and \eqref{eq:bound-DX} it follows that for $p>0$ there exists $N=N(\| \sigma\|_{\mathcal{C}^2}, \| \varrho\|_{\mathcal{C}^3}, d,p)$ such that
\begin{equs}           \label{eq:moments-DA-DB}
\sup_{i=1,...,n} \E \left(\sup_{t \leq 1} \|\mathscr{D}_i \mathcal{A}_t\|^p + \sum_{\ell=1}^d\sup_{t \leq 1}\| \mathscr{D}_i \mathcal{B}^\ell_t \|^p  \right)   \leq N
\end{equs}
This combined with \eqref{eq:bound-Jj} shows that for the `free terms' of \eqref{eq:SDE-DZ} we have
\begin{equs}
\sup_{i=1,...,n} \E \left(\sup_{t \leq 1} \|Z_{\kappa_n(t)}\mathscr{D}_i \mathcal{A}_t\|^p + \sum_{\ell=1}^d\sup_{t \leq 1}\|Z_{\kappa_n(t)} \mathscr{D}_i \mathcal{B}^\ell_t \|^p  \right)   \leq N.
\end{equs}
This, along with \eqref{eq:bound-A-B} and  \eqref{eq:bound-Jj},  implies that
\begin{equs}                \label{eq:bound-DJ}
\sup_{i=1,...,n} \E \sup_{j=1,...,n} \|\mathscr{D}_i J_j\|^p \leq \sup_{i=1,...,n} \E \sup_{i/n \leq t \leq 1} \| \mathscr{Z}^i_t \|^p \leq  N.
\end{equs}
This in turn, combined with \eqref{eq:bound-Jj-inverse} and the boundedness of $\sigma$,  implies that
\begin{equs}
\sup_{i=1,...,n} \E \sup_{r=1,...,n} \|(\mathscr{D}_iJ_j ) J^{-1}_r \sigma(\mathcal{X}_{r-1})\|^p  \leq N.
\end{equs}
Next, by the chain rule we have
\begin{equs}
 \|J_j  (\mathscr{D}_iJ^{-1}_r)  \sigma(\mathcal{X}_{r-1})\| \leq  \|J_j  \|\|J_r^{-1}\|^{2}\|\mathscr{D}_iJ_r\| \| \sigma(\mathcal{X}_{r-1})\|.
 \end{equs}
 By \eqref{eq:bound-Jj}, \eqref{eq:bound-Jj-inverse}, \eqref{eq:bound-DJ}, and the boundedness of $\sigma$, we see that
 \begin{equs}
\sup_{i=1,...,n} \E \sup_{r=1,...,n}\|J_j  (\mathscr{D}_iJ^{-1}_r)  \sigma(\mathcal{X}_{r-1})\|^p \leq N.
\end{equs}
Finally, from \eqref{eq:bound-Jj}, \eqref{eq:bound-Jj-inverse}, the boundedness of $\nabla \sigma$, and \eqref{eq:bound-DX} we get
 \begin{equs}
\sup_{i=1,...,n} \E \sup_{r=1,...,n}\|J_j  J^{-1}_r (\mathscr{D}_i \sigma(\mathcal{X}_{r-1})\|^p \leq N.
\end{equs}
Recalling \eqref{eq:decomp-DD}, we obtain
\begin{equs}              \label{eq:bound-DD}
\sup_{i=1,...,n} \E \sup_{r=1,...,n}\| \mathscr{D}_i\mathcal{D}_{r,j}\|^p \leq N,
\end{equs}
which combined with \eqref{eq:M inverse bound} gives
\begin{equs}                        \label{eq:part-1}
\sup_{i=1,...,n} \E \sup_{r=1,...,n}\|\sum_{m=1}^d ( \mathscr{D}_i \mathcal{D}^{\cdot, m}_{r,j})  (\mathscr{M}^{-1})^{m,k} \|^p \leq N t^{-p}.
\end{equs}
We proceed by obtaining a similar bound for the second term at the right hand side of \eqref{eq:decomp-DY}.
First, let us derive a bound for $\mathscr{D}_i \mathscr{M}$. For each entry $\mathscr{M}^{m,q}$ of the matrix $\mathscr{M}$ we have
\begin{equs}
\mathscr{D}_i \mathscr{M}^{m,q} = \frac{1}{n} \sum_{\ell=1}^n \sum_{v=1}^d \left( \mathcal{D}_{\ell,j}^{v,q}\mathscr{D}_i \mathcal{D}_{\ell,j}^{v,m}
+  \mathcal{D}_{\ell,j}^{v,m} \mathscr{D}_i\mathcal{D}_{\ell,j}^{v,q}\right).
\end{equs}
Then, notice that on $\hat{\Omega}$, for $\ell >j$ we have  $ \mathcal{D}_{\ell,j}= \mathscr{D}_\ell \mathcal{X}_j=0$.
Hence, by taking into account \eqref{eq:ccD bound} and \eqref{eq:bound-DD} we get
\begin{equs}
\sup_{i=1,...,n} \big(\E  \|\mathscr{D}_i \mathscr{M}^{m,q}\|^p \big) ^{1/p} \leq N\big(  \frac{j}{n}+ n (\bP(\hat{\Omega}^c))^{1/p}\big)\leq N \big(  \frac{j}{n}+ n e^{-cn/p}\big) \leq N \frac{j}{n}=Nt .
\end{equs}
Summation over $m,q$ gives
\begin{equs}                 \label{eq:bound-DM}
\sup_{i=1,...,n} \big(\E  \|\mathscr{D}_i \mathscr{M}\|^p \big) ^{1/p} \leq  N t .
\end{equs}
Therefore, we get
\begin{equs}
\|\sum_{m=1}^d \mathcal{D}^{\cdot, m}_{r,j} \otimes  \mathscr{D}_i (\mathscr{M}^{-1})^{m,k}\|\leq N \| \mathcal{D}_{r,j}\| \| \mathscr{M}^{-1}\|^2  \| \mathscr{D}_i\mathscr{M}\|,
\end{equs}
which by virtue of \eqref{eq:ccD bound}, \eqref{eq:M inverse bound}, and \eqref{eq:bound-DM} gives
\begin{equs}
\E \|\sum_{m=1}^d \mathcal{D}^{\cdot, m}_{r,j} \otimes  \mathscr{D}_i (\mathscr{M}^{-1})^{m,k}\|^p \leq N t^{-p}.
\end{equs}
This combined with \eqref{eq:part-1}, by virtue of \eqref{eq:decomp-DY}, proves \eqref{eq:cD Y bound}. This finishes the proof.
\end{proof}

\subsection{Proof of Theorem \ref{thm:density}}
\begin{proof}
Recalling that $Y_i=0$ for $i>j$, we can write, using \eqref{eq:ccD bound} and \eqref{eq:M inverse bound},
\begin{equ}
\E\|Y\|_H^2\leq \E\frac{1}{n}\sum_{i=1}^j(\sup_{i=1,\ldots,j}\|\ccD_{i,j}\|\|\cM^{-1}\|)^2\leq N(j/n)t^{-2}\leq Nt^{-1}.
\end{equ}
One also has
\begin{equ}
|\E\frac{1}{n^2}\sum_{(i,q),(r,m)\in\cI}(\cD^q_i Y^m_r)(\cD^m_rY^q_i)|
\leq
t^2 \E\sup_{i,r=1,\ldots j}\|\cD_i Y_r\|^2\leq N.
\end{equ}
Therefore, by \eqref{eq:delta Y}, we have the following bound on the main (first) term on the right-hand side of \eqref{eq:Mall IBP}
\begin{equ}
|\E(G(\cX_j),\delta Y)|\leq \|G\|_{\C^0}(\E|\delta Y|^2)^{1/2}\leq N t^{-1/2}\|G\|_{\C^0}.
\end{equ}
As for the other two terms, Proposition \ref{prop:small Omega} immediately yields
\begin{equ}
|I_1|\leq N\|G\|_{\C^1}e^{-cn},
\end{equ}
while for $I_2$ we can write
\begin{equs}
|I_2|&\leq Ne^{-cn}\Big[\E\Big(\frac{1}{n}\sum_{i=1}^j(\cD_iG(\cX_j),Y_i)\Big)^2	\Big]^{1/2}
\\
&\leq N e^{-cn} t\frac{1}{j}\sum_{i=1}^j
\big(\E\sup_{i=1,\ldots,j}|\cD_i G(\cX_j)|^6\big)^{1/6}
\big(\E\sup_{i=1,\ldots,j}\|\ccD_{i,j}\|^6\big)^{1/6}
\big(\E\|\cM^{-1}\|^6\big)^{1/6}
\\
&\leq N\|G\|_{\C^1}e^{-cn}.
\end{equs}
Therefore, by \eqref{eq:Mall IBP}, we obtain
\begin{equ}
|\E\partial_k G(\cX_j)\|\leq N \|G\|_{\C^0}t^{-1/2} + N\|G\|_{\C^1}e^{-cn},
\end{equ}
and since on $\hat \Omega$, one has $\cX_j=\bar X^n_{j/n}=\bar X^n_t$, the bound \eqref{eq:main Malliavin} follows.
\end{proof}

\section{Multiplicative Brownian noise}\label{sec:mult}
\subsection{Quadrature estimates}
\begin{Lemma}\label{lem:(ii) mult}
Let $y\in\R^d$, $\eps_1\in(0,1/2)$, $\alpha \in (0,1)$,  $p>0$.
Suppose that $\sigma$ satisfies \eqref{eq:elliptic} and that $\bar X^n:=\bar X^n(y)$ is the solution of \eqref{eq:EM no drift}.
Then for all $f\in \C^\alpha$, $0\leq s\leq t\leq 1$, $n\in\N$, one has the bound
\begin{equation}\label{DKBound}
\big\|\int_s^t (f(\bar X_r^n)-f(\bar X_{\kappa_n(r)}^n))\, dr\big\|_{L_p(\Omega)}
\leq N\|f\|_{\C^\alpha} n^{-1/2+2 \eps_1}|t-s|^{1/2+\eps_1} ,
\end{equation}
with some $N=N(\alpha, p, d,\eps_1,\lambda,\|\sigma\|_{\C^2})$.
\end{Lemma}
\begin{proof}
It clearly suffices to prove the bound for $p\geq 2$,
and, as in \cite{DK}, for $f\in\C^\infty$.
We put for $0\leq s\leq t\leq T$
$$
A_{s,t}:=\E^s \int_s^t (f(\bar X_r^n)-f(\bar X_{\kappa_n(r)}^n))\, dr.
$$
Then, clearly, for any $0\leq s\leq u\leq t\leq T$
\begin{align*}
\delta A_{s,u,t}:&=A_{s,t}-A_{s,u}-A_{u,t}\\
&=\E^s \int_u^t (f(\bar X_r^n)-f(\bar X_{\kappa_n(r)}^n))\, dr-\E^u \int_u^t(f(\bar X_r^n)-f(\bar X_{\kappa_n(r)}^n))\, dr.
\end{align*}

Let us check that all the conditions \eqref{SSL1}-\eqref{SSL2} of the stochastic sewing lemma are satisfied.
Note that
\begin{equation*}
\E^s \delta A_{s,u,t}=0,
\end{equation*}
and so condition \eqref{SSL2} trivially holds, with $C_2=0$.
As for \eqref{SSL1}, let $s \in [k/n, (k+1)/n)$ for some $k \in \mathbb{N}_0$.
Suppose first that $t \in [(k+4)/n, 1]$. We write
\begin{equs}
|A_{s,t}|= | I_1+I_2|:= \Big|\Big(\int_s^{(k+4)/n}
 +\int_{(k+4)/n}^t\Big) \E^s \big( f(\bar X^n_r)-f(\bar X^n_{k_n(r)})\big)\, dr\Big|.
\end{equs}
For $I_2$ we write,
\begin{equs}
I_2 =  \E^s \int_{(k+4)/n}^t  \E^{(k+1)/n}\big(\E^{\kappa_n(r)} f(\bar X^n_r)-f(\bar X^n_{k_n(r)})\big) \, dr.
\end{equs}
Next,  denote by $p_{\Sigma}$ the density of a Gaussian vector in $\R^d$ with covariance matrix $\Sigma$ and let $\cP_{\Sigma} f =p_{\Sigma}* f$ (recall that for $\theta \geq 0$, we denote $p_\theta := p _{\theta I}$, where $I$ is the $d \times d $ identity matrix). With this notation, we have 
\begin{equ}
\E^{k_n(r)} f\left(\bar X^n_{k_n(r)}+\sigma(\bar X^n_{k_n(r)}) (W_r-W_{k_n(r)})\right)=\cP_{\sigma \sigma^{\intercal}(\bar X^n_{k_n(r)})(r-k_n(r))} f  (\bar X^n_{k_n(r)}),
\end{equ}
so with
\begin{equ}
g(x):=g^n_r(x):=f(x)-\cP_{\sigma\sigma^{\intercal} (x)(r-\kappa_n(r))}f(x)
\end{equ}
we have
\begin{equ}\label{eq:I2 mult}
I_2=\E^s\int_{(k+4)/n}^t\E^{(k+1)/n}g^n_r(\bar X^n_{\kappa_n(r)})\,dr.
\end{equ}
Moreover, notice that by \eqref{eq:elliptic}              
we have for a constant $N=(\|\sigma\|_{\C^1}, \alpha)$
\begin{equs}           \label{eq:g-controlled-by-f}
\|g\|_{\C^{\alpha/2}} \leq N \|f\|_{\C^\alpha}.
\end{equs}
Let us use the shorthand $\delta=r-\kappa_n(r)\leq n^{-1}$. We can then write
\begin{equs}
\cP_\eps g (x)
= &  \int_{\R^d}\int_{\R^d}  p_\eps(z) p_{\sigma \sigma^{\intercal} (x-z)\delta}( y) \big(f(x-z)- f(x-y-z) \big) \, dy \,dz
\\
= &  \int_{\R^d}\int_{\R^d} p_\eps(z) p_{\sigma \sigma^{\intercal} (x-z)\delta}( y) \int_0^1 y_i \D_{z_i}f(x-z-\theta y)  \,d\theta dy \,dz
\\
= & \int_{\R^d}\int_{\R^d} \D_{z_i}\big(  p_\eps(z) p_{\sigma \sigma^{\intercal} (x-z)\delta}( y) \big) \int_0^1 y_i f(x-z-\theta y)  \,d\theta dy \,dz.                           \label{eq:whatever1}
\end{equs}
with summation over $i$ implied.
It is well known that
\begin{equs}              \label{eq:whateve2}
| \D_{z_i} p_\eps(z)| \leq N |z|\eps^{-1} p_\eps(z).
\end{equs}
Furthermore, with the notation $\Sigma (z):= \sigma \sigma^{\intercal} (x-z) $, we have 
\begin{equs}
 |\D_{z_i} p_{ \Sigma(z) \delta }( y)|
&=  \Big|  \frac{ \D_{z_i} ( y^{\intercal} \Sigma^{-1}(z) y ) }{2\delta}
+ \frac{\D_{z_i} \det \Sigma(z) }{ 2 \det \Sigma(z)}  \Big| p_{\Sigma(z)\delta}( y)
\\
& \leq N (\delta^{-1}|y|^2+1) p_{\Sigma(z)\delta}( y),
\label{eq:whatever3}
\end{equs}
where for the last inequality we have used  \eqref{eq:elliptic}. 
Therefore, by \eqref{eq:whatever1}, \eqref{eq:whateve2}, and \eqref{eq:whatever3} we see that
\begin{equs}
\|\cP_\eps g\|_{\C^0}
& \leq N\|f\|_{\C^0}
\int_{\R^d}\int_{\R^d}\Big(\eps^{-1}|z|+\delta^{-1}|y|^2+1\Big)
\Big( |y| p_\eps(z) p_{\sigma \sigma^{\intercal} (x-z)\delta}( y)\Big)\,dy\,dz
\\
&\leq N|f\|_{\C^0}(\eps^{-1/2}\delta^{1/2}+\delta^{1/2})
\leq N\|f\|_{\C^0}\eps^{-1/2}n^{-1/2}.
\end{equs}
One also has the trivial estimate $\|\cP_\eps g\|_{\C^0}\leq 2 \|f\|_{\C^0}$,
and combining these two bounds yields
\begin{equ}\label{eq: g bound}
\|g\|_{\C^\beta}\leq N\|f\|_{\C^0} n^{\beta/2}.
\end{equ}
for all $\beta\in[-1,0)$.
Note that
the restriction of $\bar X^n_t(\cdot)$ to the gridpoints $t=0,1/n,\ldots,1$ is a Markov process with state space $\R^d$.
Therefore we can write
\begin{equs}
|\E^{(k+1)/n}g\big(\bar X^n_{\kappa_n(r)}(y)\big)|
&=|\E g\big(\bar X^n_{\kappa_n(r)-(k+1)/n}(x)\big)|\Big|_{x=\bar X^n_{(k+1)/n}(y)}
\\
&\leq
\sup_{x\in\R^d} |\E g\big(\bar X^n_{\kappa_n(r)-(k+1)/n}(x)\big)|.\label{eq:another whatever}
\end{equs}
Since $g \in \C^{\alpha/2}$ we have that $(I+\Delta)u= g$ where $u \in \C^{2+(\alpha/2)}$ and
\begin{equs}          \label{eq:Schauder-estimates}
\| u\|_{\C^{2+(\alpha/2)}} \leq N \|g\|_{\C^{\alpha/2}}, \qquad \| u\|_{\C^{{1+2\eps_1}}} \leq N \|g\|_{\C{-1+2\eps_1}}.
\end{equs}
Hence, by combining \eqref{eq:another whatever}, \eqref{eq:main Malliavin},  \eqref{eq:Schauder-estimates},  \eqref{eq: g bound}, and \eqref{eq:g-controlled-by-f},   we get
\begin{equs}
|\E^{(k+1)/n}g\big(\bar X^n_{\kappa_n(r)}(y)\big)| & \leq  \sup_{x\in\R^d} |\E (u+\Delta u)  \big(\bar X^n_{\kappa_n(r)-(k+1)/n}(x)\big)|
\\
& \leq N \|u \|_{\C^1} |\kappa_n(r)-(k+1)/n|^{-1/2}+N \|u \|_{\C^2} e^{-cn}
\\
& \leq  N \|u \|_{\C^{1+2\eps_1}}  |\kappa_n(r)-(k+1)/n|^{-1/2}+N \|u \|_{\C^2} e^{-cn}
\\
& \leq  N  \|g\|_{\C^{-1+2\eps_1}}  |\kappa_n(r)-(k+1)/n|^{-1/2}+N \|g \|_{\C^{\alpha/2}} e^{-cn}
\\
& \leq N \|f\|_{\C^\alpha} n^{-1/2+\eps_1}|\kappa_n(r)-(k+1)/n|^{-1/2}
 \end{equs}
Putting this back into \eqref{eq:I2 mult} one obtains
\begin{equs}
\|I_2\|_{L_p(\Omega)} & \leq N\|f\|_{\C^0}n^{-1/2+\eps_1}\int_{(k+4)/n}^t|\kappa_n(r)-(k+1)/n|^{-1/2}\,dr
\\
&\leq
N\|f\|_{\C ^\alpha}|t-s|^{1/2}n^{-1/2+\eps_1}
\\
& \leq N\|f\|_{\C ^\alpha}|t-s|^{1/2+\eps_1}n^{-1/2+2\eps_1},
\end{equs}
where we have used that $n^{-1} \leq |t-s|$.
The bound for $I_1$ is straightforward:
\begin{equs}
\| I_1\|_{L_p(\Omega)} &\leq
\int_s^{(k+4)/n} \| f(\bar X_r)-f(\bar X_{k_n(r)}) \|_{L_p(\Omega)} \, dr
\\
&\leq N
\| f\|_{\C^0}n^{-1} \leq N \| f\|_{\C^0} n^{-1/2+\eps_1}|t-s|^{1/2+\eps_1}.
\end{equs}
Therefore,
\begin{equ}
\| A_{s,t}\|_{L_p(\Omega)}\leq N \| f\|_{\C^\alpha} n^{-1/2+2\eps_1}|t-s|^{1/2+\eps_1}.
\end{equ}
It remains to show the same bound for $t \in (s, (k+4)/n]$.
Similarly to the above we write
\begin{equs}
\|A_{s,t}\|_{L_p(\Omega)} &\leq \int_s^t \| f(\bar X_r)-f(\bar X_{k_n(r)}) \|_{L_p(\Omega)} \, dr
\\ &  \leq N \|f\|_{\C^0} |t-s|
 \leq N \| f\|_{\C^0} n^{-1/2+\eps_1}|t-s|^{1/2+\eps_1}.
\end{equs}
using that $|t-s|\leq 4 n^{-1}$ and $\eps_1<1/2$.
Thus, \eqref{SSL1} holds with $C_1=N \| f\|_{\C^\alpha} n^{-1/2+2\eps_1}$.
From here we conclude the bound \eqref{DKBound} exactly as is Lemma \ref{lem:(ii)}.
\end{proof}

\begin{Lemma}\label{lem:we need better labels mult}
Let $\alpha\in[0,1]$, take $\eps_1\in(0,1/2)$.
Let $b\in \C^0$, $\sigma$ satisfy \eqref{eq:elliptic}, and $X^n$ be the solution of \eqref{eq:approx EM}.
Then for all $f\in\C^\alpha$, $0\le s\le t\le 1$, $n\in\N$, and $p>0$, one has the bound
\begin{equation}\label{DKBound X mult}
\big\|\int_s^t (f(X_r^n)-f(X_{\kappa_n(r)}^n))\, dr\big\|_{L_p(\Omega)}
\le N\|f\|_{\C^\alpha} n^{-1/2+2\eps_1}|t-s|^{1/2+\eps_1}
\end{equation}
with some $N=N(\|b\|_{\C^0},p, d,\alpha,\eps_1, \lambda,\|\sigma\|_{\C^2})$.
\end{Lemma}
\begin{proof}
Let us set
\begin{equs}
\rho = \exp\left(-\int_0^1  (\sigma^{-1}b)(X_{\kappa_n(r)}^n) \, dB_r - \frac{1}{2}\int_0^1  \big|(\sigma^{-1}b)(X_{\kappa_n(r)}^n)\big|^2 \, dr  \right)
\end{equs}
and define the measure $\tilde\bP$ by  $d \tilde\bP = \rho d \bP$.
By Girsanov's theorem, $X^n$ solves \eqref{eq:EM no drift} with a $\tilde\bP$-Wiener process $\tilde B$ in place of $B$.
Since Lemma \ref{lem:(ii) mult} only depends on the distribution of $\bar X^n$, we can apply it to $X^n$, to bound the desired moments with respect to the measure $\tilde\bP$. Going back to the measure $\bP$ can then be done precisely as in \cite{DK}: the only property needed is that $\rho$ has finite moments of any order, which follows easily from the boundedness of $b$ and \eqref{eq:elliptic}.
\end{proof}

\subsection{A regularization lemma}
The replacement for the heat kernel bounds from Proposition \ref{prop:HK} is the following estimate on the transition kernel $\bar P$ of \eqref{eq:main mult}.
Similarly to before, we denote $\bar P_t f(x)=\E f(X_t(x))$, where $X_t(x)$ is the solution of \eqref{eq:main mult} with initial condition $X_0(x)=x$.
The following bound then follows from \cite[Theorem~9/4/2]{Friedman}.
\begin{proposition}
Assume $b\in \C^\alpha$, $\alpha>0$ and $f\in \C^{\alpha'}$, $\alpha'\in[0,1]$.
Then for all $0< t\leq 1$, $x,y\in\R^d$ one has the bounds
\begin{equ}\label{eq:HK bound2 bar}
|\bar\cP_tf(x)-\bar\cP_tf(y)|\leq N\|f\|_{\C^{\alpha'}}|x-y| t^{-(1-\alpha')/2}
\end{equ}
with some $N=N(d,\alpha,\lambda,\|b\|_{\C^\alpha},\|\sigma\|_{\C^1})$.
\end{proposition}

\begin{lemma}\label{lem:gronwall mult}
Let $\alpha\in(0,1]$ and $\tau\in(0,1]$ satisfy
\begin{equ}\label{eq:exponents Gronwall mult}
\tau+\alpha/2-1/2>0.
\end{equ}
Let $b\in\C^\alpha$, $\sigma$ satisfy \eqref{eq:elliptic}, and $X$ be the solution of \eqref{eq:main mult}.
Let $\varphi$ be an adapted process.
Then for all sufficiently small $\eps_3,\eps_4>0$, for all $f\in\C^\alpha$, $0\leq s\leq t\leq 1$,  and $p>0$, one has the bound
\begin{equs}[eq:Gronwall bound mult]
\big\|\int_s^t f(X_r) & -f(X_r+\varphi_{r})  \,dr\big\|_{L_p(\Omega)}	
\\
&\leq N |t-s|^{1+\eps_3}
\db{\varphi}_{\scC^\tau_p,[s,t]}
+
N|t-s|^{1/2+\eps_4}\db{\varphi}_{\scC^0_p,[s,t]}.
\end{equs}
with some $N=N(p,d,\alpha,\tau,\lambda,\|\sigma\|_{\C^1}).$
\end{lemma}
\begin{proof}
Set, for $s\leq s'\leq t'\leq t$,
\begin{equ}
A_{s',t'}=\E^{s'}\int_{s'}^{t'} f(X_r)-f(X_r+\varphi_{s'})\,dr.
\end{equ}
Let us check the conditions of the stochastic sewing lemma. We have
\begin{equ}
\delta A_{s',u,t'}=\E^{s'} \int_{u}^{t'} (f(X_r)-f(X_r+\varphi_{s'}))\, dr-\E^u \int_u^{t'}(f(X_r)-f(X_r+\varphi_u))\, dr,
\end{equ}
so $\E^{s'}\delta A_{s',u,t'}=\E^{s'}\hat{\delta}A_{s',u,t'}$, with
\begin{equs}
\hat\delta A_{s',u,t'}&=\E^u\int_u^{t'}\big(f(X_r)-f(X_r+\varphi_{s'})\big)-\big(f(X_r)+f(X_r+\varphi_u)\big)\,dr
\\
&=\int_u^{t'} \bar\cP_{r-u}f(X_u+\varphi_{s'})-\bar\cP_{r-u}f(X_u+\varphi_u)\,dr.
\end{equs}
Invoking \eqref{eq:HK bound2 bar}, we can write
\begin{equs}
|\hat\delta A_{s',u,t'}|&\leq N \int_{u}^{t'}|\varphi_{s'}-\varphi_u||r-u|^{-(1-\alpha)/2}\,dr.
\end{equs}
Hence, using also Jensen's inequality,
\begin{equs}
\|\E^{s'}\delta A_{s',u,t'}\|_{L_p(\Omega)}
\leq
\|\hat\delta A_{s',u,t'}\|_{L_p(\Omega)}
&\leq N\db{\varphi}_{\scC^\tau_p,[s,t]}|t'-s'|^{1+\tau-(1-\alpha)/2}
\end{equs}
The condition \eqref{eq:exponents Gronwall mult} implies that for some $\eps_3>0$, one has
\begin{equ}
\|\E^{s'}\delta A_{s',u,t'}\|_{L_p(\Omega)}
\leq N |t'-s'|^{1+\eps_3}
\db{\varphi}_{\scC^\tau_p,[s,t]}.
\end{equ}
Therefore \eqref{SSL2} is satisfied with $C_2=N\db{\varphi}_{\scC^\tau_p,[s,t]}$.
Next, to bound $\|A_{s',t'}\|_{L_p(\Omega)}$, we write
\begin{equs}
|\E^s f(X_r)- \E^s f(X_r+\varphi_{s'})|&=|\bar\cP_{r-s'}f(X_{s'})-\bar\cP_{r-s'}f(X_{s'}+\varphi_{s'})|
\\
&\leq N |\varphi_{s'}||r-s'|^{-(1-\alpha)/2}.
\end{equs}
So after integration with respect to $r$ and by Jensen's inequality, we get the bound, for any sufficiently small $\eps_4>0$,
\begin{equ}
\|A_{s',t'}\|_{L_p(\Omega)}\leq N|t'-s'|^{1/2+\eps_4}\db{\varphi}_{\scC^0_p,[s,t]}.
\end{equ}
Therefore \eqref{SSL1} is satisfied with $C_1=N\db{\varphi}_{\scC^0_p,[s,t]}$,
and we can conclude the bound \eqref{DKBound} as usual.
\end{proof}

\subsection{Proof of Theorem \ref{thm:main multiplicative}}
First let us recall the following simple fact: if $g$ is a predictable process, then by the Burkholder-Gundy-Davis and H\"older inequalities one has
\begin{equ}
\E\big|\int_s^t g_r\,dB_r\big|^p\leq N\E\int_s^t|g_r|^p\,dr|t-s|^{(p-2)/2}\end{equ}
with $N=N(p)$. This in particular implies
\begin{equ}\label{eq:BDG}
\db{g}_{\scC^{1/2-\eps}_p,[s,t]}\leq N \|g\|_{L_p(\Omega\times[s,t])}.
\end{equ}
whenever $p\geq 1/\eps$.
\begin{proof}
Without the loss of generality we will assume that $p$ is sufficiently large and $\tau$ is sufficiently close to $1/2$.
Let us rewrite the equation for $X^n$ as
\begin{equ}
dX^n_t=b(X^n_{\kappa_n(t)})\,dt+\big[\sigma(X_t)+(\sigma(X^n_t)-\sigma(X_t))+R^n_r\big]\,dB_t,
\end{equ}
where $R^n_t=\sigma(X^n_{\kappa_n(t)})-\sigma(X^n_t)$
is an adapted process such that one has
\begin{equ}
\|R^n_t\|_{L_p(\Omega)}\leq N n^{-1/2}
\end{equ}
for all $t\in[0,1]$. Let us denote
\begin{equs}
-\varphi^n_t&=x_0-x^n_0+\int_0^t b(X_r)\,dr-\int_0^t b(X_{\kappa_n(r)}^n)\,dr,
\\
\cQ^n_t&=\int_0^t\sigma(X^n_r)-\sigma(X_r)\,dB_r,
\\
\cR^n_t&=\int_0^tR^n_r\,dB_r.
\end{equs}
Take some $0\leq S\leq T\leq 1$.
Choose $\eps_1\in(0,\eps/2)$ so that
$(1/2-2\eps_1)\geq 1/2-\eps$.
Then, taking into account \eqref{DKBound X mult}, for any $S\leq s< t\leq T$,  we have
\begin{equs}\label{Step1}
\|\varphi^n_t-\varphi^n_s\|_{L_p(\Omega)}&=\big\|\int_s^t (b(X_r)-b(X^n_{\kappa_n(r)}))\,dr\big\|_{L_p(\Omega)}\\
&\leq \big\|\int_s^t (b(X_r)-b(X^n_r))\,dr\big\|_{L_p(\Omega)}+N|t-s|^{1/2+\eps_1} n^{-1/2+\eps}.
\end{equs}
We wish to apply Lemma \ref{lem:gronwall mult}, with $\varphi=\varphi^n+\cQ^n+\cR^n$.
It is clear that for sufficiently small $\eps_2>0$, $\tau=1/2-\eps_2$ satisfies \eqref{eq:exponents Gronwall mult}.
Therefore,
\begin{equs}
\big\|\int_s^t &(b(X_r)-b(X^n_r))\,dr\big\|_{L_p(\Omega)}
=\big\|\int_s^t (b(X_r)-b(X_r+\varphi_r))\,dr\big\|_{L_p(\Omega)}
\\
&\leq
N|t-s|^{1/2+\eps_4\wedge(1/2+\eps_3)}\big(\db{\varphi^n}_{\scC^\tau_p,[s,t]}
			+\db{\cQ^n}_{\scC^\tau_p,[s,t]}	
			+\db{\cR^n}_{\scC^\tau_p,[s,t]}\big)
\end{equs}
By \eqref{eq:BDG}, for sufficiently large $p$, we have
\begin{equs}
\db{\cQ^n}_{\scC^\tau_p,[s,t]} & \leq N\|X-X^n\|_{L_p(\Omega\times[0,T])},
\\
\db{\cR^n}_{\scC^\tau_p,[s,t]} & \leq Nn^{-1/2}.
\end{equs}
Putting these in the above expression, and using $\tau<1/2$ repeatedly, one gets
\begin{equs}
\big\|\int_s^t &(b(X_{r})-b(X^n_{r}))\,dr\big\|_{L_p(\Omega)}
\\
&\leq N |t-s|^{\tau}|T-S|^{\eps_5}
\big(\db{\varphi^n}_{\scC^\tau_p,[S,T]}+\|X-X^n\|_{L_p(\Omega\times[0,T])}+n^{-1/2}\big)
\end{equs}
with some $\eps_5>0$.
Combining with \eqref{Step1}, dividing by $|t-s|^\tau$ and taking supremum over $s<t\in[S,T]$, we get
\begin{equs}[Step3]
\db{\varphi^n}_{\scC^\tau_p,[S,T]}
&\leq
N\|\varphi^n_S\|_{L_p(\Omega)}+|T-S|^{\eps_5}\db{\varphi^n}_{\scC^\tau_p,[S,T]}
\\
&\qquad
+N\|X-X^n\|_{L_p(\Omega\times[0,T])}+Nn^{-1/2+\eps}.
\end{equs}
Fix an $m\in\N$ (not depending on $n$) such that $Nm^{-\eps_5}\leq 1/2$.
Whenever $|S-T|\leq m^{-1}$, the second term on the right-hand side of \eqref{Step3} can be therefore discarded, and so one in particular gets
\begin{equ}\label{eq:Step6}
\db{\varphi^n}_{\scC^\tau_p,[S,T]}
\leq
N\|\varphi^n_S\|_{L_p(\Omega)}+N\|X-X^n\|_{L_p(\Omega\times[0,T])}+Nn^{-1/2+\eps},
\end{equ}
and thus also
\begin{equ}
\|\varphi^n_{T}\|_{L_p(\Omega)}
\leq
N\|\varphi^n_S\|_{L_p(\Omega)}+N\|X-X^n\|_{L_p(\Omega\times[0,T])}+Nn^{-1/2+\eps}.
\end{equ}
Iterating this inequality at most $m$ times, one therefore gets
\begin{equ}\label{eq:Step7}
\|\varphi^n_T\|_{L_p(\Omega)}
\leq
N\|\varphi^n_0\|_{L_p(\Omega)}+N\|X-X^n\|_{L_p(\Omega\times[0,T])}+Nn^{-1/2+\eps}.
\end{equ}
We can then write, invoking again the usual estimates for the stochastic integrals $\cQ^n$, $\cR^n$
\begin{equs}
\sup_{t\in[0,T]}\big\|X_t-X_t^n\big\|_{L_p(\Omega)}^p
&\leq N\sup_{t\in[0,T]}\big\|\varphi^n_t\big\|_{L_p(\Omega)}^p
\\
&\quad+N\sup_{t\in[0,T]}\big\|\cQ^n_t\big\|_{L_p(\Omega)}^p
+N\sup_{t\in[0,T]}\big\|\cR^n_t\big\|_{L_p(\Omega)}^p
\\
&\leq N\|\varphi^n_0\|_{L_p(\Omega)}^p+N\int_0^T\|X_t-X^n_t\|_{L_p(\Omega)}^p\,dt+Nn^{-p(1/2-\eps)}.
\end{equs}
Gronwall's lemma then yields
\begin{equ}\label{eq:Step8}
\sup_{t\in[0,T]}\big\|X_t-X_t^n\big\|_{L_p(\Omega)}\leq N\|\varphi^n_0\|_{L_p(\Omega)}+Nn^{-1/2+\eps}.
\end{equ}
Putting \eqref{eq:Step6}-\eqref{eq:Step7}-\eqref{eq:Step8} together, we obtain
\begin{equ}
\db{\varphi^n}_{\scC^\tau_p,[0,1]}
\leq
N\|\varphi^n_0\|_{L_p(\Omega)}+Nn^{-1/2+\eps}.
\end{equ}
Therefore, recalling \eqref{eq:BDG} again,
\begin{equs}
\db{X-X^n}_{\scC^\tau_p,[0,1]}&\leq
\db{\varphi^n}_{\scC^\tau_p,[0,1]}
+\db{\cQ^n}_{\scC^\tau_p,[0,1]}
+\db{\cR^n}_{\scC^\tau_p,[0,1]}
\\
&\leq N\|\varphi^n_0\|_{L_p(\Omega)}+Nn^{-1/2+\eps}+\sup_{t\in[0,1]}\big\|X_t-X_t^n\big\|_{L_p(\Omega)}
\\
&\leq N\|\varphi^n_0\|_{L_p(\Omega)}+Nn^{-1/2+\eps},
\end{equs}
as desired.
\end{proof}

\begin{appendices}

\section{Proofs of the auxiliary bounds from Section~\ref{sec:prelim}}
\label{A:AAAAA}

\begin{proof}[Proof of Proposition~\ref{prop:fractional}]
\eqref{prop 1}. Fix $0\leq s\leq t\leq 1$. It follows from the definition of $B^H$ that $B^H_t-B^H_s$ is a Gaussian vector consisting of $d$ independent components, each of them having zero mean and variance
$$
C(t,t)-2C(s,t)+C(s,s)=c_H(t-s)^{2H},
$$
where the function $C$ was defined in \eqref{covfunct}. This implies the statement of the proposition.

\eqref{prop H}. We have
\begin{equation*}
\E^uB_t^{H,i}-\E^sB_t^{H,i}=\int_s^u (t-r)^{H-1/2} dW_r^i.
\end{equation*}
Therefore, $\E^sB_t^{H,i}-\E^uB_t^{H,i}$ is a Gaussian random variable independent of $\F_s$. It is of mean $0$ and variance $c^2(s,t)-c^2(u,t)$. This implies the statement of the lemma.

\eqref{prop 2}. It suffices to notice that the random vector $B^H_t-\E^s B^H_t$ is Gaussian, independent of $\mathcal{F}_s$, consists of $d$ independent components, and each of its components has zero mean and variance
\begin{equ}
\E\big(\int_s^t|t-r|^{H-1/2}\,dW_r\big)^2=c^2(s,t).
\end{equ}

\eqref{prop 3}. One can simply write by the Newton-Leibniz formula
\begin{equ}
c^2(s,t)-c^2(s,u)\leq N \int_u^t |r-s|^{2H-1}\,dr\leq N |t-u||t-s|^{2H-1},
\end{equ}
since by our assumption on $s,u,t$, for all $r\in [u,t]$ one has $r-s\leq t-s\leq 2(r-s)$.

\eqref{prop 5}. It follows from \eqref{eq:Mandelbrot} that
\begin{equation*}
\E^sB^H_t-\E^sB^H_u=\int_{-\infty}^s (|t-r|^{H-1/2}-|u-r|^{H-1/2})\,dW_r.
\end{equation*}
Therefore, by the Burkholder--Davis--Gundy inequality one has
\begin{equs}
\|\E^sB^H_t-\E^sB^H_r\|_{L_p(\Omega)}^2&\le N \int_{-\infty}^s
 \bigl(|t-r|^{H-1/2}-|u-r|^{H-1/2}\bigr)^2\,dr\\
&\leq N\int_{-\infty}^s\Bigl(\int_u^t|v-r|^{H-3/2}\,dv\Bigr)^2\,dr
\\
&\leq N\int_{-\infty}^s|t-u|^2|u-r|^{2H-3}\,dr
\\
&\leq N(t-u)^2 (u-s)^{2H-2}\leq N(t-u)^2 (t-s)^{2H-2},
\end{equs}
where the last inequality follows from the fact that by the assumption $u-s\ge (t-s)/2$.
\end{proof}

\begin{proof}[Proof of Proposition~\ref{prop:HK}]

\eqref{eq:HK bound2}.  \emph{Case  $\beta\le\alpha$}: There is nothing to prove since
$$
\|\cP_tf\|_{\C^\beta(\R^d)}\le \|\cP_tf\|_{\C^\alpha(\R^d)}\le N\|f\|_{\C^\alpha(\R^d)}.
$$

\emph{Case $\beta=0$, $\alpha<0$}:  The bound follows immediately from the definition of the norm. 

\emph{Case $\alpha=0$, $\beta \in (0, 1]$}: 
By differentiating the Gaussian density we have 
\begin{equs}
\| \nabla \cP_t f\|_{\C^0} \leq N t^{-1/2} \| f\|_{\C^0}. 
\end{equs}
Consequently,
\begin{equs}
| \cP_tf(x)-\cP_tf(y)|&  \leq | \cP_tf(x)-\cP_tf(y)|^\beta \| f\|_{\C^0}^{(1-\beta)}
\\
& \leq N t^{-\beta/2} |x-y|^\beta \|f\|_{\C^0},
\end{equs}
which implies that 
$$
[\cP_t f]_{\C^\beta} \leq N t^{-\beta/2} \|f\|_{\C^0}.
$$

This, combined with the trivial estimate $ \| \cP_tf\|_{\C^0} \leq \| f\|_{L_\infty}$ give the desired estimate.

\emph{Case $0< \alpha< \beta <1$}: We refer the reader to \cite[Lemma A.7]{perkowski} where the estimate is proved in the Besov scale. The desired estimate then follows from the equivalence $\mathcal{B}^\gamma_{\infty, \infty} \sim \C^\gamma$ for $\gamma \in (0,1)$. 

\emph{Case $\alpha \in (0,1)$, $\beta =1$}: 
We have

\begin{equs}
\| \nabla \cP_t f \|_{L\infty}& = \sup_{x \in \R^d} \Big| \int_{\R^d} \nabla p_t(x-y)f(y) \, dy \Big|
\\
& = \sup_{x \in \R^d} \Big| \int_{\R^d} \nabla p_t(x-y)\big(f(y)-f(x) \big) \, dy \Big|
\\
& \leq N [f]_{\C^\alpha} \int_{\R^d} | \nabla p_t(y)| |y|^\alpha \, dy  
\\
& \leq  N  [f]_{\C^\alpha}  t^{(\alpha-1)/2},
\end{equs}
which again combined with $\| \cP_tf\|_{\C^0} \leq \|f \|_{\C^0}$ proves the claim. 

\emph{Case $\alpha< 0$, $\beta \in [0,1]$}: 

\begin{equs}
\| \cP_t f\|_{\C^\beta} = \| \cP_{\frac{t}{2}+\frac{t}{2}}  f\|_{\C^\beta}  \leq N t^{-\beta/2} \|\cP_{t/2}f\|_{\C^0} & \leq N  t^{(\alpha-\beta)/2} \sup_{\eps \in (0,1]} \eps^{-\alpha/2}\|\cP_\eps f\|_{\C^0}
\\
&= N  t^{(\alpha-\beta)/2} \|f\|_{\C^\alpha}.
\end{equs}

\eqref{eq:HK bound}. Fix $\delta\in(0,1]$ such that $\delta\ge{\frac\alpha2-\frac\beta2}$. Then we have
\begin{align*}
\|\cP_tf-\cP_sf\|_{\C^{\beta}(\R^d)}&\le \int_s^t
\Bigl\|\frac{\partial}{\partial r}\cP_rf\Bigr\|_{\C^{\beta}(\R^d)}\,dr\\
&=\int_s^t
\Bigl\|\cP_r \Delta f\Bigr\|_{\C^{\beta}(\R^d)}\,dr\\
&\le N\int_s^t r^{\frac{\alpha-\beta-2}{2}}
\Bigl\|\Delta f\Bigr\|_{\C^{\alpha-2}(\R^d)}\,dr\\
&\le N\| f\|_{\C^{\alpha}(\R^d)}\int_s^t r^{\frac\alpha2-\frac\beta2-\delta}r^{-1+\delta}\,dr\\
&\le N\| f\|_{\C^{\alpha}(\R^d)}s^{\frac\alpha2-\frac\beta2-\delta}(t-s)^{\delta},
\end{align*}
where the last inequality follows from the facts that $r\ge s$ and $r\ge r-s$, and that both of the exponents in the penultimate inequality are nonpositive thanks to the conditions on $\delta$.
This yields the statement of \eqref{eq:HK bound}.

\eqref{eq:diffnewbound}.
First let us deal with the case $H\le 1/2$. Then the bound follows easily by applying part \eqref{eq:HK bound} of the proposition with $\delta=1/2$.
Indeed, for any $0\le s\le u \le t$  we have
\begin{align*}
&\|\cP_{c^2(s,t)}f-\cP_{c^2(u,t)}f\|_{\C^\beta}\le N \|f\|_{\C^\alpha} c^{\alpha-\beta-1}(u,t)
(c^2(s,t)-c^2(u,t))^{\frac12}\\
&\le N\|f\|_{\C^\alpha}(t-u)^{H(\alpha-\beta-1)}(u-s)^{\frac12} (t-u)^{H-\frac12}\\
&=N\|f\|_{\C^\alpha}(u-s)^{\frac12} (t-u)^{H(\alpha-\beta)-\frac12},
\end{align*}
where we also used the fact that
\begin{equation}\label{onlyifHislessthanhalf}
c^2(s,t)-c^2(u,t)\le N (u-s)(t-u)^{2H-1}.
\end{equation}
This establishes the desired bound.

Now let us consider the case $H>1/2$ (in this case $2H-1>0$ and thus bound
\eqref{onlyifHislessthanhalf} does not hold).  Put for $0\le s\le u \le t$
\begin{equation}\label{defk}
k(s,u,t):=c^2(u,t)+(u-s)\partial_tc^2(u,t)=(2H)^{-1}(t-u)^{2H}+(u-s)(t-u)^{2H-1}.
\end{equation}
Note that by convexity of the function $z\mapsto z^{2H}$  one has for any $0\le z_1\le z_2$
\begin{equation*}
z_1^{2H}+2H(z_2-z_1)z_1^{2H-1}\le z_2^{2H} \le z_1^{2H}+2H(z_2-z_1)z_1^{2H-1}+(z_2-z_1)^{2H}.
\end{equation*}
Hence for $0\le s\le u \le t$ we have
\begin{equation}\label{recall}
c^2(u,t)\le k(s,u,t)\le c^2(s,t) \le k(s,u,t)+c^2(s,u)
\end{equation}
Now we are ready to obtain the desired bound. We have
\begin{align}\label{Hgehalf0}
\|\cP_{c^2(s,t)}f-\cP_{c^2(u,t)}f\|_{\C^\beta}&\le \|\cP_{c^2(s,t)}f-\cP_{k(s,u,t)}f\|_{\C^\beta} +\|\cP_{k(s,u,t)}f-\cP_{c^2(u,t)}f\|_{\C^\beta}\nn\\
&\le I_1+I_2.
\end{align}
We bound $I_1$ and $I_2$ using  part \eqref{eq:HK bound} of the proposition but with different $\delta$. First, we apply  part \eqref{eq:HK bound} with $\delta=\frac1{4H}\vee(\alpha/2-\beta/2)$. Recalling \eqref{recall}, we deduce
\begin{equation}\label{Hgehalf1}
I_1\le N\|f\|_{C^\alpha}k(s,u,t)^{\frac\alpha2-\frac\beta2-\delta}c^{2\delta}(s,u)
\le  N\|f\|_{C^\alpha}(u-s)^{\frac12}(t-u)^{(H(\alpha-\beta)-\frac12)\wedge0}.
\end{equation}
Applying now part \eqref{eq:HK bound} with $\delta=1/2$, we obtain
\begin{equation*}
I_2\le N \|f\|_{C^\alpha}c^{\alpha-\beta-1}(u,t)(u-s)^{\frac12}(t-u)^{H-\frac12}\le N \|f\|_{C^\alpha} (u-s)^{\frac12} (t-u)^{H(\alpha-\beta)-\frac12}.
\end{equation*}
This, combined with \eqref{Hgehalf0} and \eqref{Hgehalf1} implies the desired bound for the case $H>1/2$.

\eqref{randomresult}. We begin with the case $H\le 1/2$. Then, applying part \eqref{eq:HK bound2} of the theorem with $\beta=1$, we deduce for $0\le s \le u
\le t\le 1$
\begin{equation*}
|\cP_{c^2(u,t)}f(x)-\cP_{c^2(u,t)}f(x+\xi)|\le N\|f\|_{C^\alpha}(t-u)^{H(\alpha-1)}|\xi|.
\end{equation*}
Hence for any $p\ge 2$ we have
\begin{align*}
\|\cP_{c^2(u,t)}f(x)-\cP_{c^2(u,t)}f(x+\xi)\|_{L_p(\Omega)}&\le N\|f\|_{C^\alpha}(t-u)^{H(\alpha-1)}\|\xi\|_{L_p(\Omega)}\\
&\le N\|f\|_{C^\alpha}(u-s)^{\frac12}(t-u)^{H\alpha-\frac12},
\end{align*}
where the last inequality follows from the bound \eqref{onlyifHislessthanhalf} and the definition of the random variable $\xi$. This completes the proof for the case $H\le1/2$.

Now let us deal with the case $H\in(1/2,1)$. Fix $0\le s \le u
\le t\le 1$.
Let $\eta$ and $\rho$ be independent Gaussian random vectors consisting of $d$ independent identically distributed  components each. Suppose that for any $i=1,\hdots,d$ we have $\E \eta^i=\E \rho^i=0$ and
$$
\Var (\eta^i)=(u-s)(t-u)^{2H-1};\quad\Var (\rho^i)=v(s,u,t)-(u-s)(t-u)^{2H-1}.
$$
It is clear that
\begin{align}\label{appenixcombo1}
&\|\cP_{c^2(u,t)}f(x)-\cP_{c^2(u,t)}f(x+\xi)\|_{L_p(\Omega)}\nn\\
&\quad= \|\cP_{c^2(u,t)}f(x)-\cP_{c^2(u,t)}f(x+\eta+\rho)\|_{L_p(\Omega)}\nn\\
&\quad\le  \|\cP_{c^2(u,t)}f(x)-\cP_{c^2(u,t)}f(x+\eta)\|_{L_p(\Omega)}+
\|\cP_{c^2(u,t)}f(x+\eta)-\cP_{c^2(u,t)}f(x+\eta+\rho)\|_{L_p(\Omega)}\nn\\
&\quad=:I_1+I_2.
\end{align}
Applying part \eqref{eq:HK bound2} of the theorem with $\beta=1$, we get
\begin{equation}\label{appenixcombo2}
I_1
\le N\|f\|_{\C^\alpha}c^{\alpha-1}(u,t)\|\eta\|_{L_p(\Omega)}\le
N\|f\|_{\C^\alpha}(u-s)^{\frac12}(t-u)^{\alpha H-\frac12}.
\end{equation}
Similarly, using part \eqref{eq:HK bound2} of the theorem with $\beta=\frac1{2H}\vee\alpha$ and recalling \eqref{recall}, we deduce
\begin{equation*}
I_2
\le N\|f\|_{\C^\alpha}c^{(\alpha-\frac1{2H})\wedge0}(u,t) \,\|\,|\rho|^{\frac1{2H}\vee\alpha}\,\|_{L_p(\Omega)}\le
N\|f\|_{\C^\alpha}(u-s)^{\frac12}(t-u)^{(\alpha H-\frac12)\wedge0}.
\end{equation*}
Combined with \eqref{appenixcombo1} and \eqref{appenixcombo2}, this yields the required bound.
\end{proof}

\begin{proof}[Proof of Proposition \ref{P:Holdernorms}] Obviously it suffices to show it for  $k=1$. 

\emph{1. Case $\alpha-\delta =0$}:
The statement  follows directly by definition of the $\C^\alpha$ norm.

\emph{2. Case $\alpha-\delta \in (0,1]$}:
 First, let us consider $\alpha \in (0,1]$.  For all $\beta \in [0,1]$ we have 
\begin{equs}
|f(y+x)-f(y)-f(z+x)-f(z)| & \leq   (2|x|^\alpha [f]_{\C^\alpha})^ \beta  (2|y-z|^{\alpha} [f]_{\C^\alpha})^{(1-\beta)}
\end{equs}
which upon dividing by $|y-z|^{\alpha-\delta}$, choosing $\beta= \delta/ \alpha$, and taking suprema over $y \neq z$ gives 
\begin{equation*}
[f(\cdot+x)-f(\cdot)]_{\C^{\alpha-\delta}} \leq 4 |x|^\delta [f]_{\C^\alpha}.
\end{equation*}
Similarly, we have 
\begin{equs}
\|f(\cdot+x)-f(\cdot)\|_{\C^0} \leq |x|^\delta [f]_{\C^\alpha}^{\delta/\alpha} (2\|f\|_{\C^0})^{1-\delta/\alpha} \leq 2 |x|^\delta \| f\|_{\C^\alpha},
\end{equs}
which combined with the  inequality above gives 
\begin{equs}
\|f(\cdot+x)-f(\cdot)\|_{\C^{\alpha-\delta}} \leq 6 |x|^\delta \|f\|_{\C^\alpha}.
\end{equs}
Now let us consider the case  $\alpha \in (1,2]$. By the fundamental theorem of calculus we have for any $\beta \in [0,1]$ 
\begin{equs}
& \frac{|f(y+x)-f(y)-f(z+x)-f(z)|}{|y-z|^{\alpha-\delta}} 
\\
= &\frac{1}{|y-z|^{\alpha-\delta}} \Big|\int_0^1 x_i\big( \d_{x_i}f(y+\theta x)-\d_{x_i}f(z+\theta x)\big) \, d \theta \Big |^\beta
\\ 
& \qquad \qquad \times \Big|\int_0^1 (y_i-z_i)\big(\d_{x_i}f(z+x + \theta (y-z))- \d_{x_i}f(z + \theta (y-z)) \big) \, d \theta \Big|^{(1-\beta)} 
\\
\leq & N \frac{( | x| [\nabla f]_{\C^{\alpha-1}}|y-z|^{\alpha-1} )^\beta (|y-z| [ \nabla f ]_{C^{\alpha-1}} |x|^{\alpha-1})^{1-\beta} }{|y-z|^{\alpha-\delta}}
\\
\leq & N  |x|^{\beta+(\alpha-1)(1-\beta)}\|f\|_{\C^\alpha} |y-z|^{(\alpha-1)\beta+1-\beta-\alpha+\delta},
\end{equs}
which upon choosing $\beta=(\delta+1-\alpha)/2\alpha$ and taking suprema over $ y \neq z$ gives
\begin{equation*}
[f(x+\cdot)-f(\cdot)]_{\C^{\alpha-\delta}} \leq N |x|^\delta \| f\|_{\C^\alpha}. 
\end{equation*}
In addition, we have 
\begin{equs}
\|f(\cdot+x)-f(\cdot)\|_{\C^0} \leq |x|^\delta [f]_{\C^\delta} \leq N |x|^\delta \| f\|_{C^\alpha},
\end{equs}
which combined with the above proves the claim. 

\emph{3. Case  $\alpha-\delta \in (k, k+1]$ for $k \in \mathbb{N}$}: The statement follows by proceeding as above,   considering also derivatives of $f$ up to sufficiently high order.

\emph{4. Case  $\alpha- \delta<0$}:  We first consider the case $\alpha \in [0,1)$, for which we have by virtue of Proposition \ref{prop:HK} \eqref{eq:HK bound2}
\begin{equs}
\| f(x+\cdot)- f(\cdot)\|_{\C^{\alpha-\delta}}&  = \sup_{\eps \in (0,1]} \eps^{\frac{\delta-\alpha}{2}} \| \cP_\eps f(x+ \cdot)- \cP_\eps f(\cdot)\|_{\C^0} 
\\
& 
\leq \sup_{\eps \in (0,1]} \eps^{\frac{\delta-\alpha}{2}} |x|^ \delta \| \cP_\eps f \|_{\C^\delta}
\\
& 
 \leq N \sup_{\eps \in (0,1]} \eps^{\frac{\delta-\alpha}{2} }|x|^\delta \eps^{\frac{\alpha-\delta}{2}} \| f\|_{\C^\alpha}= N |x|^\delta \| f\|_{\C^\alpha}. 
\end{equs}
We move to the case $\alpha<0$. We have 
\begin{equs}
\| f(x+\cdot)- f(\cdot)\|_{\C^{\alpha-\delta}}&  = \sup_{\eps \in (0,1]} \eps^{\frac{\delta-\alpha}{2}} \| \cP_\eps f(x+ \cdot)- \cP_\eps f(\cdot)\|_{\C^0} 
\\
& 
\leq \sup_{\eps \in (0,1]} \eps^{\frac{\delta-\alpha}{2}} |x|^ \delta \| \cP_\eps f \|_{\C^\delta}
\\
& =  \sup_{\eps \in (0,1]} \eps^{\frac{\delta-\alpha}{2}} |x|^ \delta \| \cP_{\frac{\eps}{2}+\frac{\eps}{2}} f \|_{\C^\delta}
\\
& 
 \leq N \sup_{\eps \in (0,1]} \eps^{\frac{\delta-\alpha}{2} }|x|^\delta \eps^{\frac{-\delta}{2}} \| \cP_{\frac{\eps}{2}}f\|_{\C^0} \leq  N |x|^\delta \| f\|_{\C^\alpha}. 
\end{equs}
The proposition is proved. 
\end{proof}

\section{Proofs of the results from Section~\ref{S:Girsan} related to the Girsanov theorem}
\label{B:BBBBB}

\begin{proof}[Proof of Proposition~\ref{Pr:Girsanov}]
If $H=1/2$, then there is nothing to prove; the statement of the proposition follows from the standrad Girsanov theorem for Brownian motion.
Otherwise, if $H\neq1/2$, let us verify that all the conditions of the Girsanov theorem in the form of \cite[Theorem~2]{Nualart} are satisfied. Note that even though this theorem is stated in  \cite{Nualart} in the one--dimensional setting, its extension to the multidimensional setup is immediate.

First, let us check condition (i) of \cite[Theorem~2]{Nualart}. If $H<1/2$, then $\int_0^1 u_s^2 ds\le M^2<\infty$  and thus this condition is satisfied by the statement given at \cite[last paragraph of Section 3.1]{Nualart}. If $H>1/2$, then
\begin{equation*}
\bigl[D_{0+}^{H-1/2}u\bigr](t)=Nu_tt^{-H+1/2}+N(H-1/2)\int_0^t\frac{u_t-u_s}{(t-s)^{H+1/2}}\,ds,
\end{equation*}
where $D_{0+}^{\beta}$ denotes the left-sided Riemann–Liouville derivative of of order $\beta$ at $0$, $\beta\in(0,1)$, see \cite[formula~$(4)$]{Nualart}. Therefore, taking into account that $H<1$ and assumption \ref{intassump},
\begin{equation*}
\int_0^1\Bigl|\bigl[D_{0+}^{H-1/2}u\bigr](t)\Bigr|^2\,dt\le NM^2+ N\int_0^1\Bigl(\int_0^t\frac{|u_t-u_s|}{(t-s)^{H+1/2}}\,ds\Bigr)^2\,dt<\infty\,\,\text{a.s.}.
\end{equation*}
Thus,  $D_{0+}^{H-1/2}u\in L_2([0,1])$ a.s. and hence condition (i) of \cite[Theorem~2]{Nualart} is satisfied.

Now let us verify condition (ii) of \cite[Theorem~2]{Nualart}. Consider the following kernel:
\begin{equation*}
K_H(t,s):=(t-s)^{H-1/2}F(H-1/2,1/2-H,H+1/2,1-t/s),\quad 0\le s\le t\le1,
\end{equation*}
where $F$ is the Gauss hypergeometric function, see \cite[equation $(2)$]{Ustunel}.
It follows from  \cite[Corollary~3.1]{Ustunel}, that there exists a constant $k_H>0$ and $d$--dimensional Brownian motion $\wt W$ such that
\begin{equation*}
B^H(t)=k_H\int_0^t K_H(t,s)\, d\wt W_s,\quad 0\le t\le 1.
\end{equation*}
Consider a random variable
\begin{equation*}
\rho:=\exp\Bigl(-\int_0^1  v_s d\wt W_s-\frac12\int_0^1 |v_s|^2 ds\Bigr),
\end{equation*}
where the vector $v$ is defined in the following way. If $H<1/2$, then
\begin{equation}\label{Hlesshalf}
v_t:=\frac{\sin(\pi(H+1/2))}{\pi k_H}t^{H-1/2}\int_0^t (t-s)^{-H-1/2}s^{1/2-H}u_s\, ds,
\end{equation}
and if $H>1/2$, then
\begin{equation}\label{Hbiggerhalf}
v_t:=\frac{\sin(\pi(H-1/2))}{\pi k_H (H-1/2)}\Bigl(t^{1/2-H}u_t+(H-1/2)\int_0^t \frac{u_t-t^{H-1/2}s^{1/2-H}u_s}{{(t-s)}^{H+1/2}}\, ds \Bigr).
\end{equation}
Taking into account \cite[formulas $(11)$ and $(13)$]{Nualart}, we see that condition (ii) of \cite[Theorem~2]{Nualart} is equivalent to the following one: $\E \rho=1$. We claim that actually
\begin{equation}\label{tocheck}
\E \exp(\lambda \int_0^1|v_t|^2 dt)\le R(\lambda)<\infty
\end{equation}
where
\begin{align*}
&R(\lambda):=\exp(\lambda N(H)M^2)\quad\text{if $H<1/2$};\\
&R(\lambda):=\exp(\lambda N(H)M^2)\E\exp(\lambda N(H)\xi)\quad\text{if $H\in(1/2,1)$}.
\end{align*}
By the Novikov theorem this, of course, implies that $\E \rho=1$.

Now let us verify \eqref{tocheck}. If $H<1/2$, then it follows from \eqref{Hlesshalf} and \eqref{eq:integral} that
\begin{equation*}
|v_t|\le N(H)Mt^{-H+1/2},
\end{equation*}
which immediately yields \eqref{tocheck}.

If $H>1/2$, then we make use of
\eqref{Hbiggerhalf} and \eqref{eq:integral2} to deduce
\begin{align*}
|v_t|\le& N(H)Mt^{1/2-H}+N(H)\int_0^t\frac{|u_t| (t^{H-1/2}s^{1/2-H}-1)}{{(t-s)}^{H+1/2}}\, ds\\
&+N(H)\int_0^t\frac{|u_t-u_s|t^{H-1/2}s^{1/2-H}}{{(t-s)}^{H+1/2}}\, ds\\
\le& N(H)Mt^{1/2-H}+N(H)\int_0^t\frac{|u_t-u_s|t^{H-1/2}s^{1/2-H}}{{(t-s)}^{H+1/2}}\, ds.
\end{align*}
Taking into account assumption \eqref{intassump}, we obtain \eqref{tocheck}. Thus, by above, condition (ii) of \cite[Theorem~2]{Nualart} is satisfied.

Therefore all the conditions of \cite[Theorem~2]{Nualart} are satisfied. Hence the process $\wt B^H$ is indeed a fractional Brownian motion with Hurst parameter $H$ under $\wt\P$ defined by $d\wt\P/d\P=\rho$.

Finally, to show \eqref{densitybound}, we fix $\lambda>0$. Then, applying the Cauchy--Schwarz inequality, we get
\begin{align*}
\E \rho^\lambda=&\E \exp\Bigl(-\lambda\int_0^1  v_s d\wt W_s-\frac\lambda2\int_0^1 |v_s|^2 ds\Bigr)\\
=&\E \exp\Bigl(-\lambda\int_0^1  v_s d\wt W_s-\lambda^2\int_0^1 |v_s|^2 ds+(\lambda^2-\lambda/2)\int_0^1 |v_s|^2 ds\Bigr)\\
\le& \Bigl[\E \exp\Bigl(-2\lambda\int_0^1  v_s d\wt W_s-2\lambda^2\int_0^1 |v_s|^2 ds\Bigr)\Bigr]^{1/2}
\Bigl[\E \exp\Bigl((2\lambda^2-\lambda)\int_0^1  |v_s|^2 ds\Bigr)\Bigr]^{1/2}\\
=&\Bigl[\E \exp\Bigl((2\lambda^2-\lambda)\int_0^1  |v_s|^2 ds\Bigr)\Bigr]^{1/2}\\
\le& R(2\lambda^2)^{1/2} <\infty,
\end{align*}
where the last inequality follows from \eqref{tocheck}. This completes the proof of the proposition.
\end{proof}

\begin{proof}[Proof of Lemma~\ref{L:annoyingintegral}]
We begin with establishing bound \eqref{annoyingintegralbound}.
Fix $n\in\N$ and let us split the inner integral in
\eqref{annoyingintegralbound} into two parts: the integral over $[0,\kappa_n(t)-(2n)^{-1}]$ and $[\kappa_n(t)-(2n)^{-1},t]$. For the first part we have
\begin{align}\label{annoyingintegral1}
I_1(t)&:=\int_0^{\kappa_n(t)-(2n)^{-1}}\frac{(t/s)^{H-1/2} |f_{\kappa_n(t)}-f_{\kappa_n(s)}|} {(t-s)^{H+1/2}}\,ds\nn\\
&\le [f]_{\C^\rho}t^{H-1/2}
\int_0^{\kappa_n(t)-(2n)^{-1}} s^{1/2-H} |\kappa_n(t)-\kappa_n(s)|^\rho (t-s)^{-H-1/2}\,ds\nn\\
&\le N[f]_{\C^\rho}t^{H-1/2}
\int_0^{\kappa_n(t)-(2n)^{-1}} s^{1/2-H} |t-s|^{\rho-H-1/2}\,ds\nn\\
&\le N[f]_{\C^\rho}t^{H-1/2}
\int_0^{t} s^{1/2-H} |t-s|^{\rho-H-1/2}\,ds\nn\\
&\le N[f]_{\C^\rho}t^{\rho-H+1/2},
\end{align}
where we used bound \eqref{eq:integral}, the assumption $\rho-H-1/2>-1$, and the fact that for $s\in [0,\kappa_n(t)-(2n)^{-1}]$ one has
$$
\kappa_n(t)-\kappa_n(s)\leq t-s+1/n\le  3(t-s).
$$

Now let us move on and estimate the second part of the inner integral in \eqref{annoyingintegralbound}. If $t\ge1/n$, then we have
\begin{align}\label{annoyingintegral2}
I_2(t)&:=\int_{\kappa_n(t)-(2n)^{-1}}^t\frac{(t/s)^{H-1/2} |f_{\kappa_n(t)}-f_{\kappa_n(s)}|} {(t-s)^{H+1/2}}\,ds\nn\\
&= t^{H-1/2}|f_{\kappa_n(t)}-f_{\kappa_n(t)-1/n}|
\int_{\kappa_n(t)-(2n)^{-1}}^{\kappa_n(t)} s^{1/2-H} (t-s)^{-H-1/2}\,ds\nn\\
&\le N [f]_{\C^\rho}n^{-\rho}\frac{t^{H-1/2}}{(\kappa_n(t)-(2n)^{-1})^{H-1/2}}
\int_{\kappa_n(t)-(2n)^{-1}}^{\kappa_n(t)}  (t-s)^{-H-1/2}\,ds\nn\\
&\le  N [f]_{\C^\rho}n^{-\rho}(t-\kappa_n(t))^{-H+1/2},
\end{align}
where in the last inequality we used that for $t\ge1/n$ one has
$$
t\le \kappa_n(t)+\frac1n\le 4\kappa_n(t)-\frac2n=4\Bigl(\kappa_n(t)-\frac1{2n}\Bigr).
$$

Now, using \eqref{annoyingintegral2} and \eqref{annoyingintegral2}, we can
bound the left--hand side of \eqref{annoyingintegralbound}. We deduce
\begin{align*}
&\int_0^1\Bigl(\int_0^t\frac{(t/s)^{H-1/2} |f_{\kappa_n(t)}-f_{\kappa_n(s)}|} {(t-s)^{H+1/2}}\,ds\Bigr)^2\,dt\\
&\qquad\le N\int_0^1 I_1(t)^2\,dt+
N\int_0^1 I_2(t)^2\,dt\\
&\qquad\le N[f]_{\C^\rho}^2+N[f]_{\C^\rho}^2n^{-2\rho}
\sum_{i=1}^{n-1}
\int_{\frac{i}{n}}^{\frac{i+1}n}|t-\kappa_n(t)|^{1-2H}\,dt\\
&\qquad\le N[f]_{\C^\rho}^2+N[f]_{\C^\rho}^2n^{-2\rho}
\sum_{i=1}^{n-1} n^{-(2-2H)}\\
&\qquad\le N[f]_{\C^\rho}^2+N[f]_{\C^\rho}^2n^{2H-1-2\rho}\le N[f]_{\C^\rho}^2,
\end{align*}
where the very last inequality follows from the assumption $\rho>H-1/2$. This establishes \eqref{annoyingintegralbound}.

Not let us prove \eqref{notannoyingintegralbound}. Using the assumption $\rho>H-1/2$ and identity \eqref{eq:integral}, we deduce
\begin{align*}
\int_0^1\Bigl(\int_0^t\frac{(t/s)^{H-1/2} |f_{t}-f_{s}|} {(t-s)^{H+1/2}}\,ds\Bigr)^2\,dt&\le[f]_{\C^\rho}^2
\int_0^1\Bigl(\int_0^t s^{-H+1/2} (t-s)^{\rho-H-1/2}\,ds\Bigr)^2\,dt\\
&\le  N[f]_{\C^\rho}^2.
\end{align*}
This proves \eqref{notannoyingintegralbound}.
\end{proof}

\end{appendices}

\bibliographystyle{Martin}
\bibliography{SSLEuler1}

\end{document}